\documentclass[twocolumn]{svjour3}          
\usepackage{amssymb}
\usepackage{amsmath}
\usepackage{amsfonts}
\usepackage{mathdots}

\usepackage[sort]{cite}
\usepackage{psfrag}
\usepackage{epsfig}
\usepackage{amsmath,bbm,times,stmaryrd}

 \usepackage[mathscr]{eucal}

\usepackage{tabularx}
\usepackage{multirow}
\usepackage{enumerate}

\usepackage{epstopdf}
\usepackage[normal]{subfigure}

\usepackage{colortbl}
\usepackage{dcolumn}
\newcolumntype{.}{D{.}{.}{1.3}}

\usepackage{txfonts}
\usepackage[subnum]{cases}

\usepackage{equationarray}

\newcommand\dashrule{\leavevmode\xleaders\hbox{-}\hfill\kern0pt}

\def\argmin{\operatornamewithlimits{arg\,min}}
\def\diag{\operatorname{diag}}

\newcommand{\ba}{{\bvec{a}}}
\newcommand{\bb}{\bvec{b}}
\newcommand{\bc}{\bvec{c}}

\newcommand{\bg}{\bvec{g}}

\newcommand{\bs}{\bvec{s}}
\newcommand{\bt}{\bvec{t}}
\newcommand{\bu}{\bvec{u}}
\newcommand{\bv}{\bvec{v}}

\newcommand{\bx}{\bvec{x}}
\newcommand{\by}{\bvec{y}}
\newcommand{\bz}{\bvec{z}}

\newcommand{\bA}{{\bf A}}
\newcommand{\bB}{{\bf B}}

\newcommand{\bD}{{\bf D}}

\newcommand{\bF}{{\bf F}}
\newcommand{\bG}{{\bf G}}
\newcommand{\bH}{{\bf H}}
\newcommand{\bI}{{\bf I}}

\newcommand{\bL}{{\bf L}}

\newcommand{\bN}{{\bf N}}

\newcommand{\bQ}{{\bf Q}}

\newcommand{\bS}{{\bf S}}

\newcommand{\bT}{{\bf T}}
\newcommand{\bU}{{\bf U}}
\newcommand{\bV}{{\bf V}}

\newcommand{\bX}{{\bf X}}
\newcommand{\bY}{{\bf Y}}
\newcommand{\bZ}{{\bf Z}}

\newcommand{\calI}{{\mathcal{I}}}

\newcommand{\calL}{{\mathcal{L}}}

\newcommand{\balpha}{\mbox{\boldmath $\alpha$}}

\newcommand{\bsigma}{\mbox{\boldmath $\sigma$}}

\newcommand{\bPi}{\mbox{\boldmath $\Pi$}}

\newcommand{\1}{\mbox{\boldmath $1$}}
\newcommand{\0}{\mbox{\boldmath $0$}}

\newcommand{\be}{\begin{eqnarray}}
\newcommand{\ee}{\end{eqnarray}}

\newcommand{\matrixb}{\left[ \begin{array}}
\newcommand{\matrixe}{\end{array} \right]}

\newcommand{\tr}{\mathop{\rm tr}\nolimits}

\def\*{\circledast}

\newcommand{\bvec}[1]{\boldsymbol{#1}}

\newcommand{\ve}{\bvec{e}}

\def\vectorize{\operatorname{vec}}
\newcommand{\vtr}[1]{\vectorize\hspace{-.3ex}\left(#1\right)}

\newcommand{\tensor}[1]{\boldsymbol{\mathscr{\MakeUppercase{#1}}}} 

\newcommand{\tG}{\tensor{G}}

\newcommand{\tX}{\tensor{X}}
\newcommand{\tY}{\tensor{Y}}

\usepackage[vlined,ruled,commentsnumbered]{algorithm2e}

\usepackage{aurical}

\usepackage{arydshln}

\usepackage{url}

\usepackage{mathtools}

\usepackage{relsize}
\usepackage{amssymb}
\renewcommand{\ltimes}{\mathlarger{\mathlarger{\ltimes}}}

\smartqed  

\makeatletter
\newsavebox{\@brx}
\newcommand{\llangle}[1][]{\savebox{\@brx}{\(\m@th{#1\langle}\)}%
  \mathopen{\copy\@brx\kern-0.5\wd\@brx\usebox{\@brx}}}
\newcommand{\rrangle}[1][]{\savebox{\@brx}{\(\m@th{#1\rangle}\)}%
  \mathclose{\copy\@brx\kern-0.5\wd\@brx\usebox{\@brx}}}
\makeatother

\usepackage{graphicx}
\graphicspath{{./SQP/}}

\begin{document}
\sloppy
\title{Quadratic Programming Over Ellipsoids}
\subtitle{with Applications to Constrained Linear Regression and Tensor Decomposition}


\author{Anh-Huy Phan
\and Masao Yamagishi 
\and Danilo Mandic
\and Andrzej Cichocki
}


\institute{A.-H. Phan and A. Cichocki \at Lab for Advanced Brain Signal Processing, Brain Science Institute, RIKEN, Wakoshi, Japan \\
\email{(phan,cia)@brain.riken.jp} \\
{A. Cichocki \at Skolkovo Institute of Science and Technology (Skoltech), Russia}\\
M.Yamagishi	\at Tokyo Institute of Technology, Japan\\
\email{myamagi@sp.ce.titech.ac.jp} 
{D. Mandic \at  Imperial College, London, United Kingdom}\\
\email{d.mandic@imperial.ac.uk} 
}

%
\date{Received: date / Accepted: date}

\maketitle

\begin{abstract}
A novel algorithm to solve the quadratic programming problem over ellipsoids is proposed. This is achieved by splitting the problem into two optimisation sub-problems, quadratic programming over a sphere and orthogonal projection. Next, an augmented-Lagrangian algorithm is developed for this multiple constraint optimisation. Benefit from the fact that the QP over a single sphere can be solved in a closed form by solving a secular equation \cite{GANDER1989815} and \cite{Hager2001}, we derive a tighter bound of the minimiser of the secular equation. 
We also propose to generate a new psd matrix with a low condition number from the matrices in the quadratic constraints. This correction method improves convergence of the proposed augmented-Lagrangian algorithm.
Finally, applications of the quadratically constrained QP to bounded linear regression and tensor decompositions are presented. 
\end{abstract}

\section{Introduction}

Quadratic programming over a single sphere is one of basic optimisation problems and has been extensively studied. 
For example, the problem was first considered as an eigenvalue problem with linear constraints $\bN^T \bx = \bt$
 by Gander, Golub and Matt \cite{GANDER1989815}. After eliminating the linear constraint, the constrained eigenvalue problem becomes a QP problem over a single sphere. The authors solved a secular equation using an iterative algorithm which starts from an initial point determined based on the eigenvalue of the quadratic term.

A similar study was presented by Hager as a minimization of a quadratic function over a sphere \cite{Hager2001}. Hager also considered solving a rational function. 
Rojas, Santos and Sorensen \cite{Rojas:2008:ALM:1326548.1326553} developed a trust-region algorithm which can be applied to this problem.
Some extensions to solving large-scale problems were proposed in \cite{Sorensen:1997:MLS,DBLP:journals/mp/RendlW97}.

The spherically constrained QP (SCQP) problem was reinvented many times. In \cite{ChenGao2013}, Chen and Gao presented a globally optimal solution to the QP with a variable vector constrained inside a ball. They formulated the problem as one-dimensional canonical duality problem and proposed associated numerical algorithms.

For this particular problem, by considering the variable vector in the Stiefel manifold, we can also apply optimization algorithms on a manifold to solve this problem, e.g., using the ManOpt toolbox \cite{manopt} or  the Cran-Nicholson update scheme\cite{Zaiwen_2012}.

\begin{figure*}
\centering
\includegraphics[width=.71\linewidth, trim = 0.0cm .0cm 0cm 0cm,clip=true]{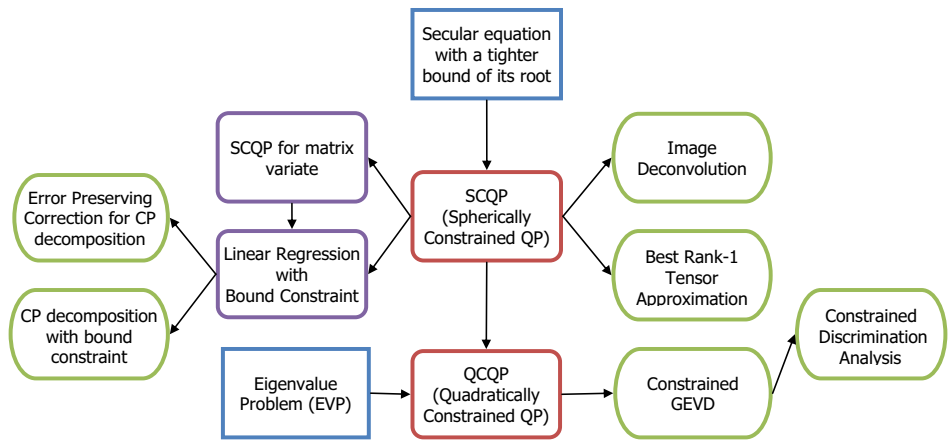}
\caption{Flow chart of the Spherically Constrained QP and Quadratically Constrained QP and their novel applications presented in this paper.}
\label{fig_flowchart}
\end{figure*}

A more sophisticated problem is that of minimizing a convex quadratic function over an intersection of ellipsoids $\bx^T \, \bH_m \, \bx = 1$, 
bearing in mind that quadratic equalities characterize non-convex sets. The non-convex quadratic optimisation problem with quadratic equality constraints is known to be NP-hard \cite{Linderoth2005,DBLP:journals/siamjo/ArimaKK13,Burer2014}.
%
%
%
%
Nevertheless, since gradients and Hessians of the constrained functions can be derived in an analytical form, the problem can be solved using interior-point algorithms for nonlinearly constrained minimization\cite{Holmstršm97tomlab}. Alternatively, one can cast a convex quadratic and quadratically constrained optimization problem into a conic optimization problem which can be solved efficiently, e.g, using the Mosek optimisation toolbox\cite{mosek}. 
The QP problem with quadratic inequality constraints can also be solved efficiently using the Modified proportioning with gradient projections \cite{Dostl:2009:OQP:1593031,Dostal2012}. 
Some other common approaches are to convexify the problem using 
semidefinite relaxation techniques \cite{DBLP:journals/jacm/GoemansW95}, second-order cone programming\cite{Kim00secondorder}, or mixed SOCP-SDP relaxations \cite{Burer2014}.

For some particular cases, e.g., a quadratic function with two quadratic constraints, \cite{Ben-Tal1996} shows that under a suitable assumption, the problem can be solved in polynomial time. Similarly, with some simple convex relaxations the solution can even return the optimal values  \cite{Locatelli:2015:RQP:2902389.2902762}.

In this paper, we develop algorithms for QP with quadratic constraints $\bx^T \bH_m \, \bx = 1$, and present novel applications of this optimisation. First, we consider the simple QP over a sphere. In the same spirit as Gander, Golub and Matt \cite{GANDER1989815}, Hager \cite{Hager2001}, we solve the problem by finding a root of a secular equation.
Normalisation and conversion methods are introduced to simplify the problem to the one with a smaller number of parameters, when the vector in the linear term comprises zero-entries, or when the matrix in the quadratic term has identical eigenvalues.
The conversion is particularly useful for the SCQP for matrix variate in Section~\ref{sec:sqp_matrixvariate}.
We show that the solution to such a constrained QP problem can be deduced from a minimiser of a much smaller similar QP for a vector variate.
For the ordinary SCQP, we present new results for finding good bounds of the minimiser. To this end, we perform a slightly different normalisation to that in \cite{GANDER1989815} and \cite{Hager2001}. With this new bound, we can even find a good estimate to the global minimiser through solving a truncated problem with a few terms. It is shown that the solution can be found in closed form for some particular cases without resorting iterative algorithms.

In Section~\ref{sec:regression}, we present the linear regression with a bound constraint and formulate it as an equivalent SCQP. 
This problem has applications in deriving the norm correction method for the CANDECOMP/PARAFAC tensor decomposition (CPD), and for developing the algorithm for the bounded CPD\cite{Phan_CPnormcorrection}.  

In Section~\ref{sec:qpqc}, we will present an algorithm to solve the Quadratic programming over elliptic constraints. The problem with multiple constraints is split into two optimisation sub-problems, one is the quadratic programming over a sphere, and the other being the orthogonal projection. An augmented-Lagrangian algorithm is next developed for this problem. We suggest generating a new psd matrix with a low condition number from the matrices in the quadratic constraints. This correction method is proved to improve convergence of the proposed augmented-Lagrangian algorithm.

We present novel applications of the quadratic programming over a sphere to tensor decompositions, including finding a best rank-1 tensor approximation to symmetric tensors of order-4 \cite{7952616}, and constrained discrimination analysis. 

In Section~\ref{sec:TT_geig}, we introduce constrained generalized eigenvalue decomposition, in which eigenvectors impose low-rank structures. The problem is then converted to sub-problems related to the ordinary GEVD and the QP over multiple quadratic constraints. Throughout the paper, we provide many examples, including image deconvolution, best rank-1 tensor approximation and image classification, to verify and illustrate our algorithms. 
In addition, a flow chart  in Fig.~\ref{fig_flowchart} summarises the studied methods and their applications.

\section{Quadratic Programming Over A Single Sphere}

Consider a quadratic programming problem with a constraint that the variable vector is on a sphere, i.e., unit-length vector. 

\begin{definition}[Quadratic Programming over a Single Sphere]\label{prob_qps}
Given a positive semi-definite matrix $\bQ$ of size $K \times K$  and a  vector $\bb$ of length $K$, the quadratic programming over a sphere solves the optimisation problem 
\be
&\min  \quad  &\frac{1}{2} \, \bx^T \, \bQ \bx + \bb^T \bx  \label{equ_qp_sphere} \, ,\quad\mathrm{{s.t.}}  \quad   \bx^T \bx = 1  \, .
\ee
\end{definition}

For the case when $\bb$ is a zero vector, the problem (\ref{equ_qp_sphere}) becomes that of finding the smallest eigenvectors of the matrix $\bQ$. Here, we do not consider this case. In addition, the matrix $\bQ$ only needs to be symmetric so that the positive semi-definite condition on matrix $\bQ$ can be relaxed.
We first show that the QP in (\ref{equ_qp_sphere}) can be converted to a problem whereby the matrix $\bQ$ is diagonal, and has positive eigenvalues. Then, we simplify the optimization task to that with distinct eigenvalues and non-zero entries $\bb$.

\subsection{Normalisation, reparameterization and simplication}

We shall denote the eigenvalue decomposition of the matrix $\bQ$ in (\ref{equ_qp_sphere}) by
\be 
\bQ = \bU \, \diag(\bsigma) \bU^T \notag 
\ee
where $\bsigma = [0 \le \sigma_1 \le \sigma_2 \le \cdots \le \sigma_K]$ comprises the eigenvalues of $\bQ$, $\bU$ is an orthonormal matrix of size $K\times K$, which consists of eigen-vectors of $\bQ$. 

Since the vector $\bb$ is non-zero, we can perform the following normalisation and reparameterization 
\be
	{\tilde\bx} &= &\bU^T \, \bx    \, , \quad  
	\bc  =  \bU^T \frac{\bb}{\|\bb\|} \,, \notag\\
	\bs &=& [s_1, \ldots, s_K]^T  \, ,  \quad  s_k = \frac{\sigma_k - \sigma_1}{\|\bb\|} + 1  \,  \notag
\ee
so that ${\tilde\bx}$ and $\bc$ are unit-length vectors, 
 ${\tilde\bx}^T {\tilde\bx} = 1$ and $\bc^T \bc = 1$, and $s_1 = 1 \le s_2 \le \cdots \le s_K$.
Hence, the optimal solution $\bx$ to the QP problem in (\ref{equ_qp_sphere}) can be derived from the following QP
\be
&\min  \quad  & \frac{1}{2} \, {\tilde\bx}^T \, \diag({\bs})  \, {\tilde\bx} + \bc^T {\tilde\bx}\label{equ_qp_sphere2} , \quad
\text{s.t.}  \quad   {\tilde\bx}^T {\tilde\bx} = 1  \,,
\ee
where $\bc^T \bc = 1$ and $\bs = [s_1 = 1 \le s_2 \le \cdots \le s_K]$.

We next show that the problem (\ref{equ_qp_sphere2}) can be simplified to the case with distinct eigenvalues, i.e., $s_1 < s_2 < \cdots < s_K$. 

We shall denote by $J$ the number of distinct eigenvalues, $\tilde{\bs} = [\tilde{s}_1 = 1 < \tilde{s}_2 < \cdots < \tilde{s}_{J}]$, over a set of $K$ eigenvalues, $s_k$,  in (\ref{equ_qp_sphere2}), and classify $\bc = [{\bc}_{1}, {\bc}_{2}, \ldots, {\bc}_{J}]$ into $J$ sub-vectors,  whereby each ${\bc}_{j}$ consists of entries $c_{k}$ such that $s_k = \tilde{s}_j$, i.e., $\bc_j = [c_{k \in \calI_j}]$, where $\calI_j = \{k:  s_k = \tilde{s}_j\}$. In addition, we shall define a vector 
\be
\tilde{\bc}  = [\|\bc_1\|, \|\bc_2\|, \ldots, \|\bc_J\|]\,.\notag
\ee
Then, the following relation holds, the proof of which is provided in Appendix~\ref{sec:proof_lem_sqp_simplify}.
\begin{lemma}\label{lem_sqp_simplify}
The minimiser to (\ref{equ_qp_sphere2}) can be deduced from the minimiser to an SCQP with distinct eigenvalues, that is 
\be
&\min  \quad  & \frac{1}{2} \, \bz^T \, \diag(\tilde{\bs} )  \, {\bz} + \tilde\bc^T {\bz}  \quad 
\text{subject to}  \quad   {\bz}^T {\bz} = 1 \notag \,,
\ee
as 
$\displaystyle 
\bx_{\calI_j} = \frac{z_j}{\|\bc_j\|}  \, \bc_j$, 
for a non zero vector $\bc_j$, or an arbitrary vector on the ball $\|\bx_{\calI_j}\|^2 =  z_j^2$ for a zero $\bc_j$, $j = 1, \ldots, J$. 
\end{lemma}

In the sequel, Sections~\ref{sec:cn_zeros}-\ref{sec:c1_zero} show that for most cases with zero entries  $\tilde{c}_j = 0$, e.g., when $j> 1$, the optimal $z_j^{\star}$ is zero. Hence, $\tilde{\bx}_{\calI_j}$ is also a zero vector. 

Next, we consider the case when the entries of the vector $\bc$ are nonzero. The case with zero entries $c_k$ can be deduced from the former case.  

\subsection{The case when all $c_k$ are non-zeros}\label{sec::nonzero_cn}

The Lagrangian function of the problem in (\ref{equ_qp_sphere2}) is given by
\be
\calL({\tilde\bx}, \lambda) = \frac{1}{2} \, {\tilde\bx}^T \, \diag({\bs})  \, {\tilde\bx} + \bc^T {\tilde\bx}\ - \frac{1}{2} \lambda ({\tilde\bx}^T {\tilde\bx} - 1)\, .	\notag
\ee
Following the first-order optimality condition, there exists a Lagrange  multiplier $\lambda$ such that 
\be
\frac{\partial \calL({\tilde\bx}, \lambda)}{\partial {\tilde\bx} } = (\diag(\bs)  - \lambda \bI) {\tilde\bx}   + \bc = \0. \label{equ_1st_opt_condtion}
\ee
Since $c_k$ are non-zeros, the multiplier $\lambda$ must not be any $s_k$, i.e., $\lambda \neq s_k$ for $k =1, \ldots, K$, thus implying that the minimiser ${\tilde\bx}^{\star}$ can be expressed as 
\be
{\tilde\bx}^{\star} 
&=& \left[ \frac{c_1}{\lambda - s_1}, \ldots, \frac{c_K}{\lambda - s_K} \right]	\notag
\ee
and the Lagrangian function at ${\tilde\bx}^{\star}$ is given by 
\be
\calL({\tilde\bx}^{\star}, \lambda)  
&=&   \frac{1}{2} \sum_{k = 1}^{K}  \frac{c_k^2 \, s_k}{(\lambda - s_k)^2}   + \sum_{k = 1}^{K}  \frac{c_k^2 }{\lambda - s_k} -  
\frac{\lambda}{2} \sum_{k = 1}^{K}  \frac{c_k^2}{(\lambda - s_k)^2}    + \frac{\lambda}{2}  \notag\\
&=& 
\frac{\lambda}{2}  + \frac{1}{2} \sum_{k = 1}^{K}  \frac{c_k^2}{\lambda - s_k} \,. \notag
\ee
This leads to finding a root $\lambda$ of the first derivative $f^{\prime}(\lambda)$, as
\be
f^{\prime}(\lambda) = 1 - \sum_{k = 1}^{K} \frac{c_k^2}{(\lambda - s_k)^2}\, , \label{equ_lambda_zero}
\ee
which minimises the following function 
\be
\min_{\lambda}   \quad f(\lambda) = \lambda  +  \sum_{k = 1}^{K}  \frac{c_k^2}{\lambda - s_k} \, \label{func_lambda} \,.
\ee
The \emph{secular equation} in (\ref{equ_lambda_zero}) is in a similar form to those derived in \cite{GANDER1989815} and \cite{Hager2001}, but here, the coefficients $c_k$ are with an additional constraint $\bc^T \bc = 1$ and $s_1 = 1$. This constraint will later help to derive a tighter bound for the roots of $f^{\prime}(\lambda) = 0$.

We will next show that the minimiser $\lambda^{\star}$ is a minimum root of $f^{\prime}(\lambda)$. To this end, we first illustrate that $f^{\prime}(\lambda)$ has a root less than $s_1 = 1$, and prove that this root is the global minimiser to $f(\lambda)$.

\begin{lemma}\label{lem_root_lessthan1}
The first derivative of $f(\lambda)$ in (\ref{func_lambda}) has only one root $\lambda < s_1 = 1$, which lies in the interval $(0, 1-|c_1|)$.
\end{lemma}


\begin{lemma}\label{lem_minimumroot}
The solution to the problem in (\ref{func_lambda}) is the minimum root of the first derivative of $f(\lambda)$.
\end{lemma}

Proofs of Lemmas~\ref{lem_root_lessthan1} and \ref{lem_minimumroot} are given in Appendices~\ref{sec:proof_lem_root_lessthan1}-\ref{sec:proof_lem_minimumroot}. 
We proceed to show that $\lambda^{\star}$ can be found with a  tighter bound.

\begin{lemma}\label{lem_bound_lda_1term}
The function $f(\lambda)$  in (\ref{func_lambda}) has a unique global minimizer in the interval $(1- t_1, 1- t_2)$, where $t_1$ and $t_2$ are the roots which lie in the interval of $\displaystyle \left[{|c_1|},1 \right]$ of two degree-4  polynomials $p_1(t)$ and $p_2(t)$, given by 
\be
p_i(t)  &=&  t^4 + 2  d_i \, t^3 + (d_i^2-1) \, t^2 - 2c_1^2 d_i \, t - c_1^2 d_i^2  \notag
\ee
where $d_1 = s_2 - s_1$  and $d_2 = s_K - s_1$.
\end{lemma}

We provide proof of Lemma~\ref{lem_bound_lda_1term} in Appendix~\ref{sec:proof_lem_bound_lda_1term}, and illustrate the polynomials $p_i(t)$ in $[0, 1]$ for various $d_i$ in Fig.~\ref{fig_bound_lda}. 
The roots $t_i$ approach $|c_1|$ when $d_i$ are large, and 1 when $d_i$ are small. 

\begin{figure}
\centering
\includegraphics[width=.41\textwidth, trim = 0.0cm .0cm 0cm 0cm,clip=true]{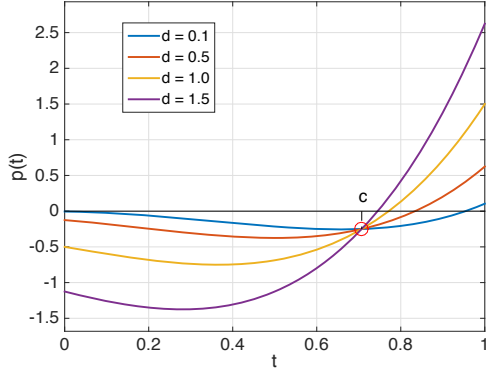}
\caption{Illustration of the polynomial $p(t) = t^4 + 2  d \, t^3 + (d^2-1) \, t^2 - 2c^2 d \, t - c^2 d^2  $ with $c^2 = 0.5$ and various $d$, and its unique roots in the interval $[c, 1]$.}
\label{fig_bound_lda}
\end{figure}

If $s_K = s_2$, i.e., $(K-1)$ eigenvalues $s_2$, $s_3$, ..., $s_K$ are identical, then $t_1 = t_2 = t^{\star}$, and $\lambda = 1-t^{\star}$ is a root of $f^{\prime}(\lambda) =  0$.
When $s_2$ and $s_K$ are relatively close, the bound $[1 - t_1, 1-t_2]$ is tight and provides a good approximation to the root $\lambda$ as illustrated in Fig.~\ref{fig_bound_lda}.
 
When $d_1 \ge 1$, it follows that $d_2 \ge 1$, and the bound width $(t_1 -t_2)$ is relatively small.
 For example, when $d_2 = 2 d_1$ and $d_1>1$,  the bound width $(t_1 - t_2)$ is often less than 0.1, while the width is even less than 0.01 when $d_1 $ exceeds 3, and is less than 0.001 when $d_1 \ge 10$, despite of values of $s_K$, as seen in Fig~\ref{fig_bound_lda2} for the cases $d_2 = 2 d_1$ and $d_2 = 1000 d_1$.

\begin{figure}[t]
\centering
\includegraphics[width=.41\textwidth, trim = 0.0cm .0cm 0cm 0cm,clip=true]{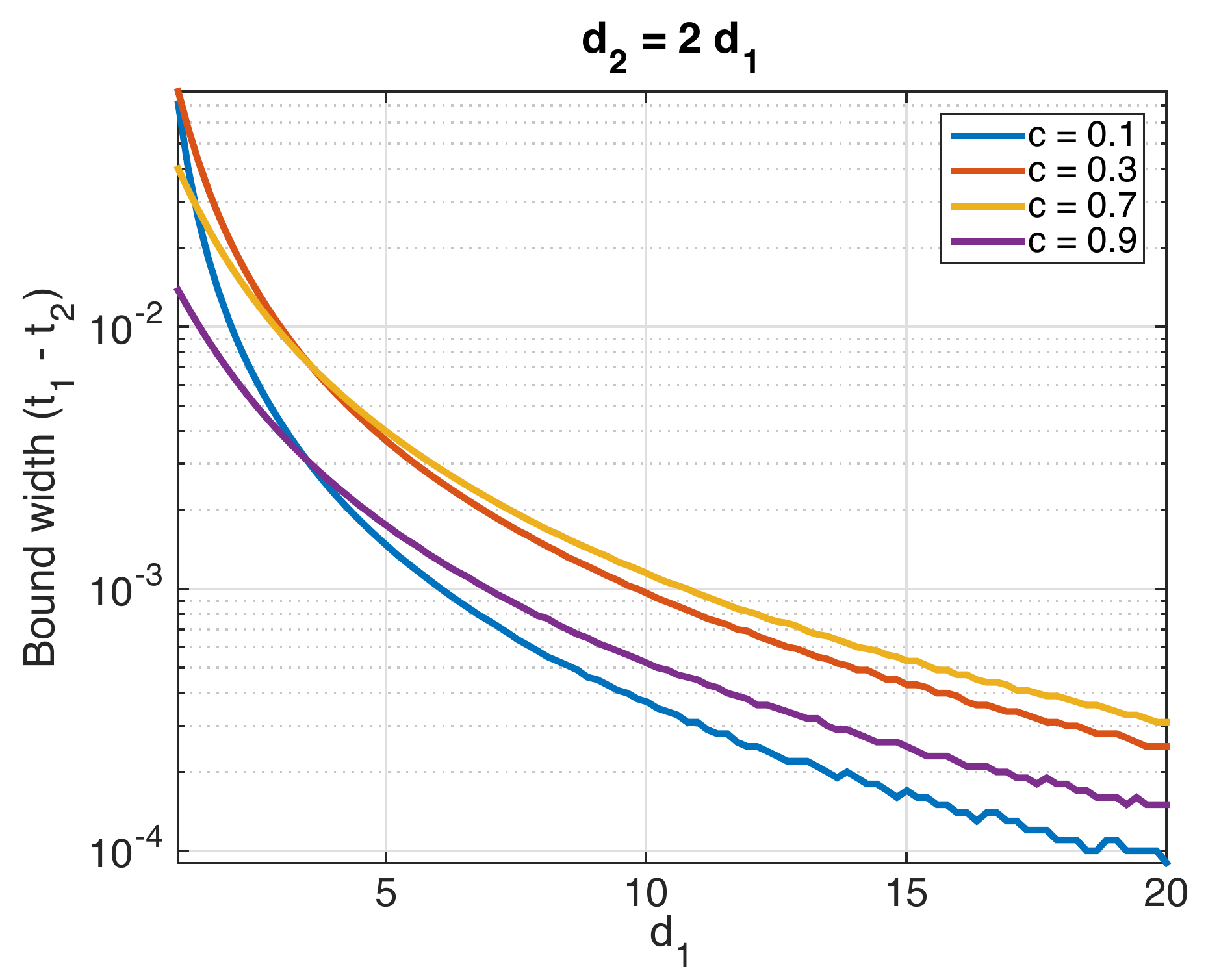}
\includegraphics[width=.41\textwidth, trim = 0.0cm .0cm 0cm 0cm,clip=true]{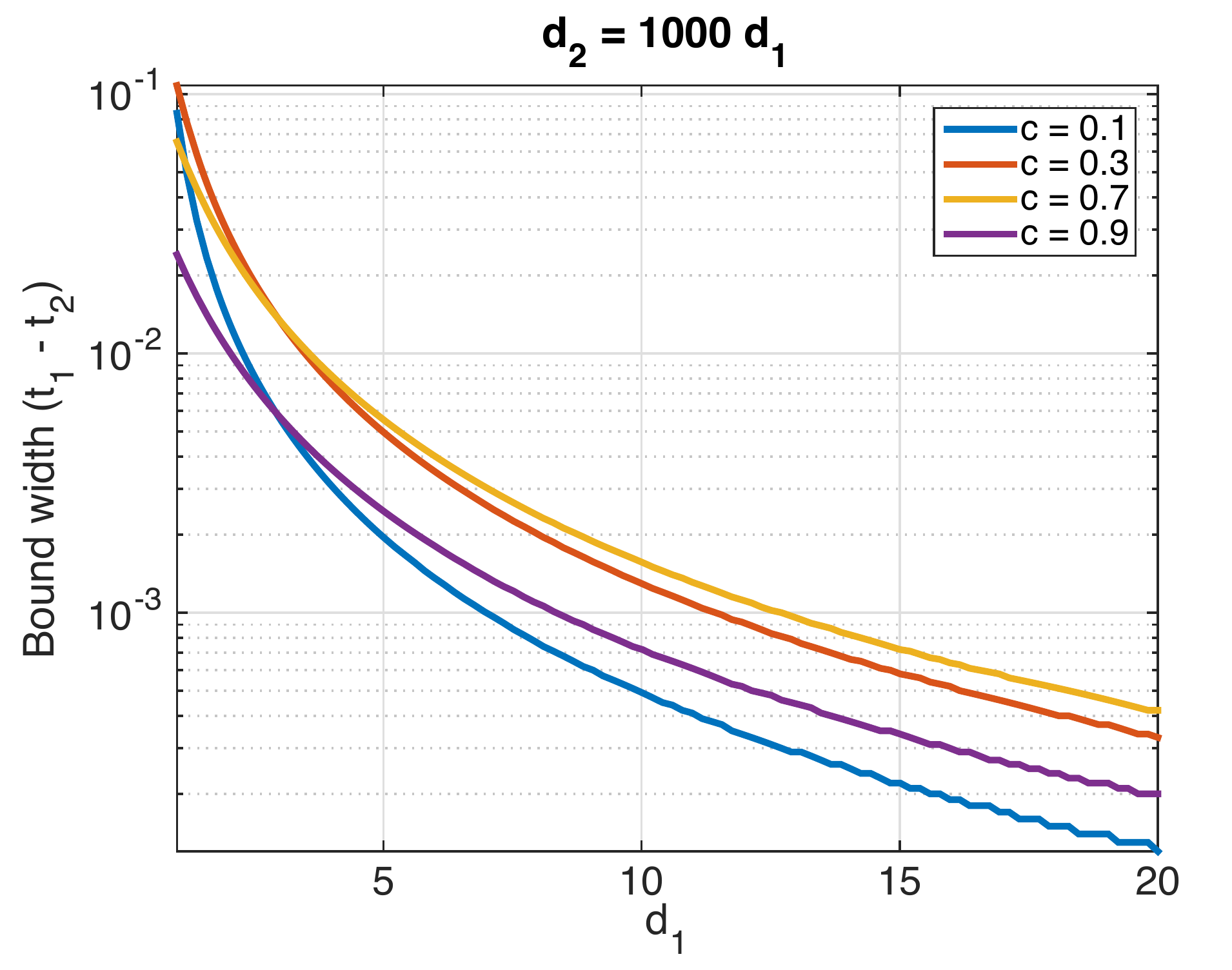}
\caption{Illustration of the bound width $(t_1 - t_2)$ when $d_2 = 2 d_1$ and $d_2 = 1000 d_1$ for various values of $c_1 = c$.}
\label{fig_bound_lda2}
\end{figure}

In general, the bound width $(t_1 - t_2)$ is tight when $d_1 \approx d_2$, i.e., the eigenvalues $s_2$, \ldots, $s_K$ are located in a narrow range,
or when $d_1$ exceeds 1, i.e., $s_2 > 2$.
However, the bound width is not sufficiently good when $d_1< 1 < d_2$, 
especially when $d_1$ is small, $t_1$ approaches 1, and $t_2$ approaches $|c_1|$. Hence, there is no much improvement on the bound for $\lambda$, compared to the obvious bound $[0, 1-|c_1|]$.

In order to improve the bound of the minimiser $\lambda^{\star}$, when $s_2 - s_1 < 1$, we propose to solve a similar equation to (\ref{equ_lambda_zero}) but with a smaller number of terms. We shall refer it to as the \emph{truncated problem}.
Let 
$
	\tilde{c}_{L} = \sqrt{\sum_{k = L}^{K} c_k^2}$. 
We define a set of equations 
$f^{(L)}_{l}(\lambda)$ and $f^{(L)}_{u}(\lambda)$ constructed from the first $L$ terms of the equation  $f^{\prime}(\lambda)$ in (\ref{equ_lambda_zero})
\be
f^{(L)}_{l}(\lambda) &=& 1  -  \sum_{l = 1}^{L} \frac{c_l^2}{(s_l - \lambda)^2}   - 	\frac{\tilde{c}_{L+1}^2 }{(s_{L+1} - \lambda)^2}	\,,  \notag  \\
f^{(L)}_{u}(\lambda) &=&1  -  \sum_{l = 1}^{L} \frac{c_l^2}{(s_l - \lambda)^2}   - 	\frac{\tilde{c}_{L+1}^2}{(s_{K} - \lambda)^2}	\, . \notag 
\ee

\begin{lemma}\label{lem_bound_lda_Lterms}
The roots $\lambda_{l,L}^{\star}$ of $f^{(L)}_{l}(\lambda)$ 
and the roots $\lambda_{u,L}^{\star}$ of $f^{(L)}_{u}(\lambda)$ in $[0, 1-|c_1|]$ are unique, and form the lower and upper bounds of the root $\lambda^{\star}$ of $f^{\prime}(\lambda)$ in (\ref{equ_lambda_zero})
\be
\lambda_{l,1}^{\star} \le \lambda_{l,2}^{\star} \le \cdots \le \lambda_{l,K-2}^{\star} \le \lambda^{\star}  \le \lambda_{u,K-2}^{\star} \le \cdots \le 
\lambda_{u,2}^{\star} \le \lambda_{u,1}^{\star} \, . \label{equ_sequence_lambda}
\ee
\end{lemma}
The proof is given in Appendix~\ref{sec:proof_lem_bound_lda_Lterms}.
We note that the bound derived in Lemma~\ref{lem_bound_lda_1term}  is a particular case of Lemma~\ref{lem_bound_lda_Lterms} with  $\lambda_{l,1}^{\star} = 1 - t_1$ and  $\lambda_{u,1}^{\star} = 1 - t_2$. 

Lemma~\ref{lem_bound_lda_Lterms} states that we can obtain a tighter bound for the minimiser $\lambda^{\star}$ of $f^{\prime}$ by solving a truncated secular equation with only a few terms $\frac{c_l^2}{(s_l - \lambda)^2}$.
The method is particularly useful when the first $L$ eigenvalues, $s_1, s_2, ..., s_{L}$, are very close to each other, while $s_{L+1}$ exceeds 1 significantly.

\begin{example}
In Fig.~\ref{fig_bound_vs_Lterms}, we demonstrate good estimates of the minimiser $\lambda^{\star}$ of the equation $f^{\prime}(\lambda)$ which has $K = 1000$ terms. 
The eigenvalues $s_k$ are randomly generated such that some of the first $T$ eigenvalues, $s_k$, are smaller than 2, where $T$ = 5 or 10. The eigenvalues, $s_k$, are plotted in Fig.~\ref{fig_bound_vs_Lterms}. The bound width $(\lambda_{u,L}^{\star} - \lambda_{l,L}^{\star})$ is computed for various $L = 1, 2, \ldots, K-1$. For the first case, we can obtain a bound of less than 0.01 when solving the truncated problem of only $4$ or $5$ terms.
For the second case, a bound of less than 0.01 is achieved when solving a truncated equation with $L = 11$ terms.
The bound is tighter, less than $10^{-3}$ when the truncated equation has 20-40 terms.
Moreover, solving the truncated problems with 200 terms provides good approximation to the global minimiser $\lambda$ with an error less than $10^{-5}$.
\end{example}
\begin{figure}
\centering
\subfigure[The case with 5 eigenvalues smaller than 2.]{
\includegraphics[width=.41\textwidth, trim = 0.0cm .0cm 0cm 0cm,clip=true]{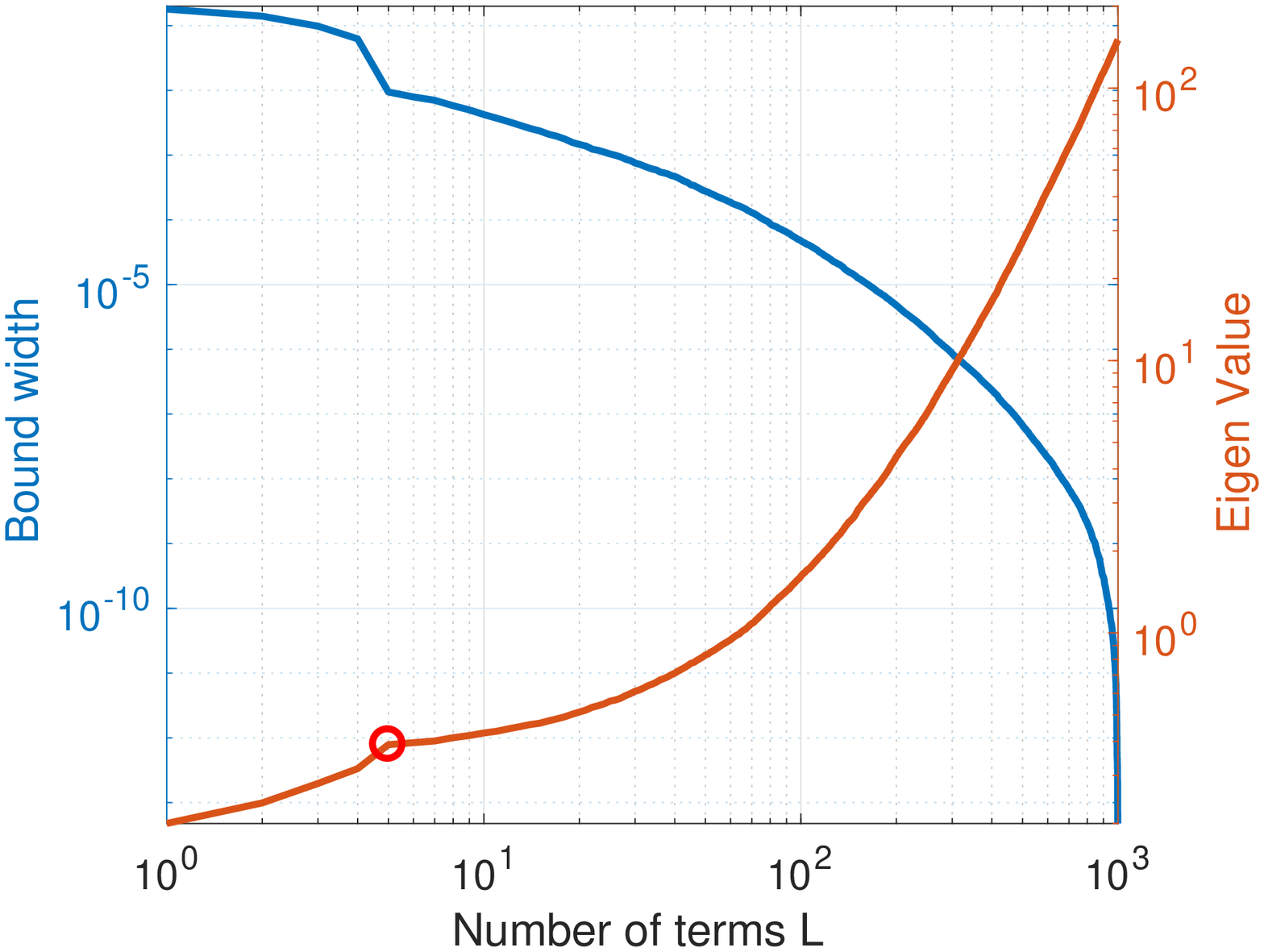}}
\hfill
\subfigure[The case with 10 eigenvalues smaller than 2.]{
\includegraphics[width=.41\textwidth, trim = 0.0cm .0cm 0cm 0cm,clip=true]{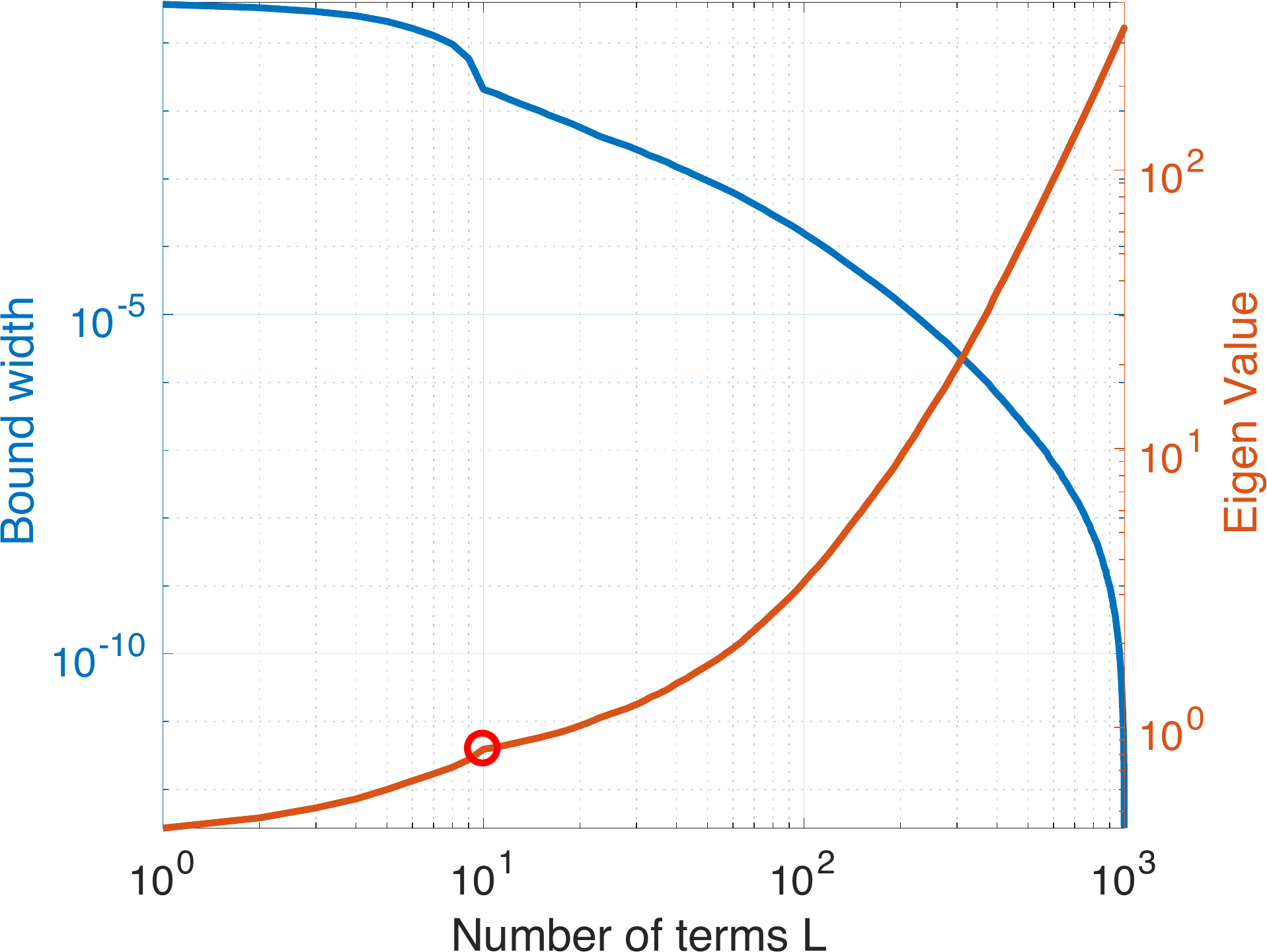}}
\caption{Illustration of a bound width of $\lambda$ by solving the reduced problem using $L$ terms. The bound width can be less than $0.001$ when solving the truncated equations with dozens of terms.}
\label{fig_bound_vs_Lterms}
\end{figure}

\subsection{The case when more than one coefficients $c_n$ are zeros}
\label{sec:cn_zeros}

Assume that there are more than one zero coefficients $c_{l} = 0$, we then denote their index set by $\calI_0 = \{l :  c_l = 0\}$, 
and by $n$ the smallest index of this set, i.e., $c_n = 0$.
We shall first show that the entries $\tilde{x}_{l}^{\star}$ of the minimiser $\tilde{\bx}^{\star}$ are zeros, where $l \in \calI_0$ and $l \neq n$, 
and the optimization problem can be converted to the case with only one zero coefficient $c_n = 0$.

The objective function (\ref{equ_qp_sphere2}) can be rewritten as 
\be
\min  \;\;  \left(\frac{1}{2} \sum_{k \notin \calI_0} s_k \, \tilde{x}_k^2   + \sum_{k \notin \calI_0}   c_k \, \tilde{x}_k \right) + \frac{1}{2} \sum_{l \in \calI_0}  \tilde{x}_l^2 \, s_l \, , \quad 
\text{s.t.}  \;\; {\tilde\bx}^T {\tilde\bx} = 1    \notag 
\ee
and achieves a minimum when the subset of the variables $\tilde{\bx}_{\calI_0} = [\tilde{x}_l:  l\in \calI_0]$ is a minimiser of the following problem
\be
& \min   \qquad   & \sum_{l \in \calI_0}  \tilde{x}_l^2 \, s_l \, \notag  \, \quad
\text{s.t.}  \quad {\tilde\bx_{\calI_0}}^T {\tilde\bx}_{\calI_0} =  r \notag 
\ee
where $r = 1 - \sum_{k \notin \calI_0}  x_k^2 >0$. The problem now boils down to finding an eigenvector associated with the smallest eigenvalue, i.e., $s_n$, of the diagonal matrix $\diag(\bs_{\calI_0})$. This implies that $\tilde{x}_n^2  =  r$, 
and the other entries  $\tilde{x}_{l}$ are zeros, where $l \in \calI_0$, $l\neq n$.
The problem is now simplified into a problem formulated for $s_k$ and $c_k$ where $k \in \{1,\ldots, K\} \setminus \calI_0 \cup \{n\}$, which has at most one zero coefficient $c_n = 0$.
We will next show that $x_n$ is also zero if $n>1$.

\subsection{The case when only one coefficient $c_n$ is zero with $n>1$}\label{sec:cn_zero}


\begin{lemma}\label{lem_cn_zero}
When there is only one $c_n = 0$ with $n > 1$, the $n$-th variable of the minimiser is zero, i.e., $\tilde{x}_n^{\star} = 0$.
\end{lemma}

The proof of this case is given in Appendix~\ref{sec:proof_lem_cn_zero}. In summary, as shown in this and previous sub-sections, if the coefficients $c_n$, with $n>1$, are zeros, the corresponding parameters of the minimiser $\tilde{x}_n$ are zeros as well, and the remaining variables are a solution to a similar problem but with a reduced number of parameters.
 
\subsection{The case when $c_1 = 0$}\label{sec:c1_zero}

When $c_1 = 0$, we consider the two sub-cases, when 
$d = \sum_{k>1}^{K} \frac{c_k^2}{(s_k-1)^2}$  is less than or greater than 1.

\begin{lemma}\label{lem_c1_zero}
Consider the case $c_1= 0$, let $d = \sum_{k>1}^{K} \frac{c_k^2}{(s_k-1)^2}$.
\begin{itemize}
\item If $d \le 1$, then the following $\tilde{\bx}^{\star}$ is a minimiser to the problem in (\ref{equ_qp_sphere2})
\be
	\tilde{x}^{\star}_{k} =  \frac{c_k}{1-s_k} ,  \quad k >1 	\notag 
\ee 
and $\tilde{x}^{\star}_{1}$ can take one of the two values
 $\pm \sqrt{1 - d}$.
\item Otherwise, the minimiser has $\tilde{x}_{1} = 0$, and the remaining $(K-1)$ variables  $[\tilde{x}_2, \ldots, \tilde{x}_K]$ are a solution to a reduced problem 
\be
&\min  \quad   & \sum_{k>1} \frac{1}{2} \, s_k \, \tilde{x}_k^2   + c_k \, \tilde{x}_k  \label{prob_x_2_K} \,, \quad \mathrm{{s.t.}}  \quad  \sum_{k>1}  \tilde{x}_k^2 = 1  \,.
\ee
\end{itemize}
\end{lemma}
%

Proof of Lemma~\ref{lem_c1_zero} is presented in Appendix~\ref{sec:proof_lem_c1_zero}.

\subsection{Algorithm}

Steps to solve the QP over a sphere are summarised in Algorithm~\ref{alg_qp_sphere}. The algorithm first normalizes the parameters $\bb$ and $\bQ$, and converts the considered problem to a QP problem with a diagonal matrix $\diag(\bs)$, $s_1 = 1$ and $\bc^T \bc  = 1$.

Zero coefficients $\bc_l$, where $l>1$, are verified in order to simplify the problem to that with a fewer number of parameters of $\bs_{\calI}$ and $\bc_{\calI}$, where $\calI = \{1\}  \cup \{l : c_l \neq  0, l>1\}$ is the index set of ${1}$ and non-zeros $c_l$.

Next, identical eigenvalues, $s_k$, are identified and the problem is simplified again to the one with distinct eigenvalues. 

For the reduced problem with $\tilde{\bs}$ and $\tilde{\bc}$, the solution can be found in closed-form in the following particular cases 
 \begin{itemize}
 \item $\tilde{c}_1 = 0$
 \item $s_2 > s_1  =1$ 
\item  $d = \sum_{k>1} \frac{\tilde{c}_k^2}{(\tilde{s}_k - 1)^2} < 1$ .
\end{itemize}

In other cases, we find the lower and upper bounds of the minimiser $\lambda^{\star}$ by finding roots of two polynomials of degree-4 or by solving truncated equations with a few rational terms. The global minimiser is then found using an iterative algorithm in the estimated bounds.

\setlength{\algomargin}{1em}
\begin{algorithm}[t!]
\SetFillComment
\SetSideCommentRight
\CommentSty{\footnotesize}
\caption{Spherically Constrained Quadratic Programming (SCQP)}\label{alg_qp_sphere}
\DontPrintSemicolon \SetFillComment \SetSideCommentRight
\KwIn{$\bQ$ and $\bb$} 
\KwOut{$\bx$  minimises $\frac{1}{2} \, \bx^T \bQ \bx + \bc^T \bx$, s.t., $\bx^T \bx = 1$} 
\SetKwFunction{card}{card} 
\SetKwFunction{coreqps}{qps\_nnz} 
\Begin{
\nl Eigenvalue decomposition $\bQ = \bU \diag(\bsigma) \, \bU^T$\;
\nl $\bs = \frac{\boldsymbol{\sigma} - \sigma_1}{\|\bb\|} + 1$,  $\bc = \bU^T \frac{\bb}{\|\bb\|}$\;
\nl $\calI_0 = \{l :  c_l = 0, l>1 \}$, $\calI = \{1,\ldots, K\} \setminus \calI_{0}$ , $\tilde{K} = \card{$\calI$}$\;
\nl  $\tilde{\bs} = \bs_{\calI}$,  $\tilde{\bc} = \bc_{\calI}$   \;
\If{$\tilde{c}_1 = 0$}
{ 
\nl $d =  \sum_{k \in \calI } \frac{\tilde{c}_k^2}{(\tilde{s}_k  - 1)^2}$\;
\If {$d<1$ and $\tilde{s}_2 > \tilde{s}_1 = 1$}
{
\nl $\tilde{x}_1 =  1 \pm \sqrt{d}$, $\tilde{x}_k = \frac{\tilde{c}_k}{1-\tilde{s}_k}$, $k = 2, \ldots, \tilde{K}$
}
\Else
{
\nl $\tilde{\bx} = \coreqps(\tilde{\bs}(2:\tilde{K})  - \tilde{s}(2) + 1,\tilde{\bc}(2:\tilde{K}))$\;
\nl $\tilde{\bx}  = [0, \tilde{\bx}]$;
}
}
\Else
{
\nl $\tilde{\bx} = \coreqps(\tilde{\bs},\tilde{\bc})$
}
\nl $\bx_{\calI_0} = 0$, $\bx_{\calI} = \tilde{\bx}$\;
\nl $\bx \leftarrow \bU \, {\bx}$\;
}
\BlankLine
{\bf function} $\bx  = \coreqps(\bs,\bc)$\;
\KwIn{$\bs = [s_1 =1 \le s_2 \le \ldots \le s_K]$, unit-length vector $\bc$, $\bc^T \bc = 1$ and $c_k \neq 0$} 
\KwOut{$\bx$  minimises $ \frac{1}{2} \, \bx^T \diag(\bs) \bx + \bc^T \bx$, s.t., $\bx^T \bx = 1$} 
\Begin{
\nl Compute the roots $t_1$ and $t_2$ in $[{|c_1|},1)$ of the polynomials  \\
$p_i(t)  =  t^4 + 2  d_i \, t^3 + (d_i^2-1) \, t^2 - 2c_1^2 d_i \, t - c_1^2 d_i^2   $, where $d_1 = s_2 - s_1$ and $d_2 = s_K - s_1$\;
\nl Find a root $\lambda$ in $(1- t_1, 1- t_2)$ of 
$f(\lambda) = 1 - \sum_{k} \frac{c_k^2}{(\lambda - s_k)^2}$\;
\nl  $\bx = [x_1, \ldots, x_K]$ where $x_k =  \frac{c_k}{\lambda - s_k}$
}	
\end{algorithm}

\begin{figure}[t]
\centering
\begin{minipage}[b]{.48\textwidth}
\includegraphics[width=.48\textwidth, trim = 0.0cm 7cm 8.7cm 0cm,clip=true]{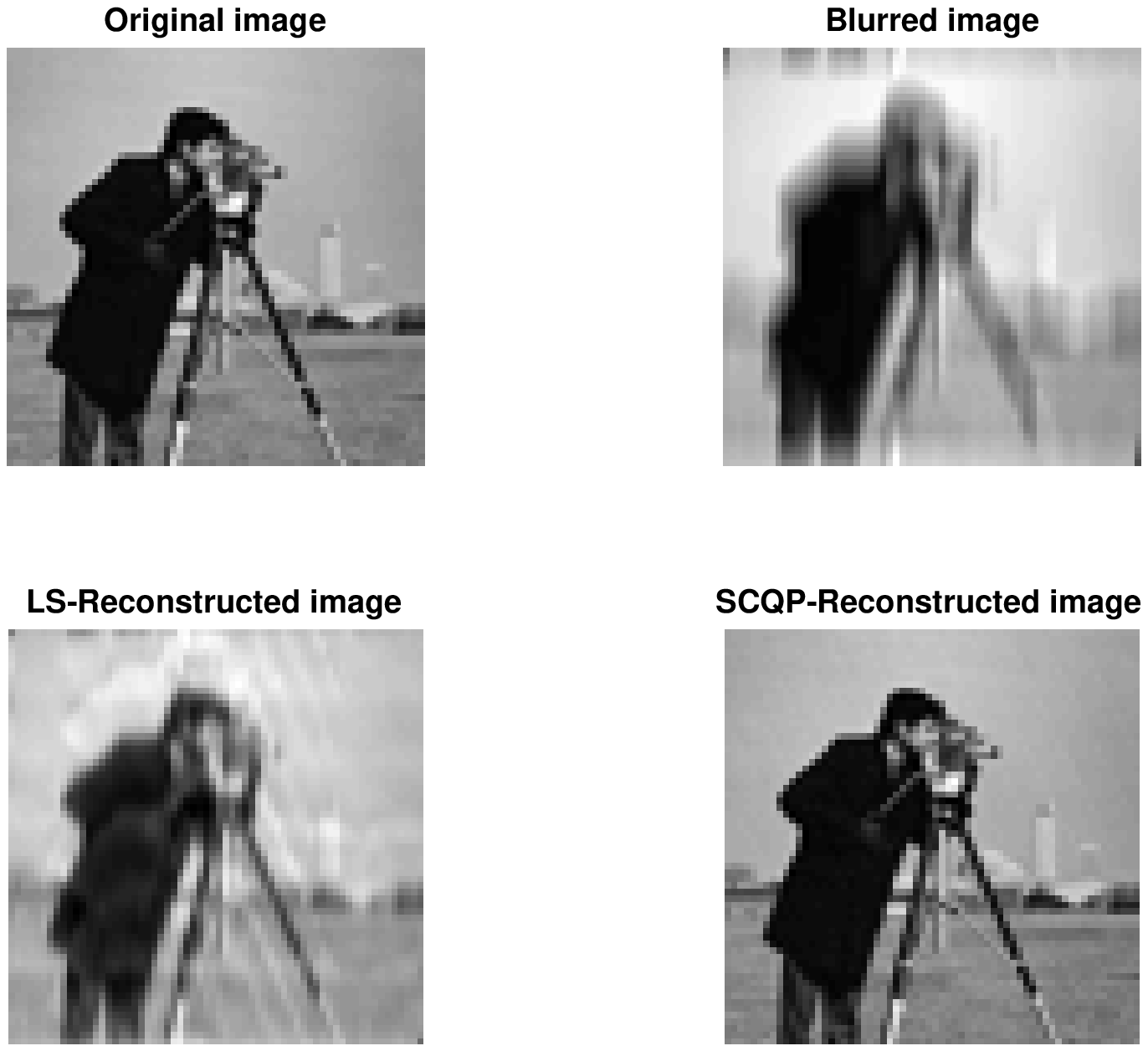}
\includegraphics[width=.48\textwidth, trim = 8.7cm 7cm 0cm 0cm,clip=true]{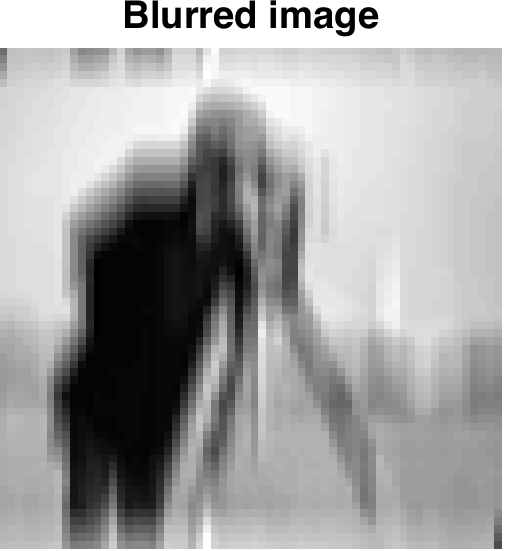}\\
\includegraphics[width=.48\textwidth, trim = 0.0cm 0cm 8.7cm 7cm,clip=true]{fig_sqp_deconvolution_cmp}
\includegraphics[width=.48\textwidth, trim = 8.7cm 0cm 0cm 7cm,clip=true]{fig_sqp_deconvolution_cmp}
\end{minipage}
\\
\includegraphics[width=.43\textwidth, trim = 0.0cm .0cm 0cm 0cm,clip=true]{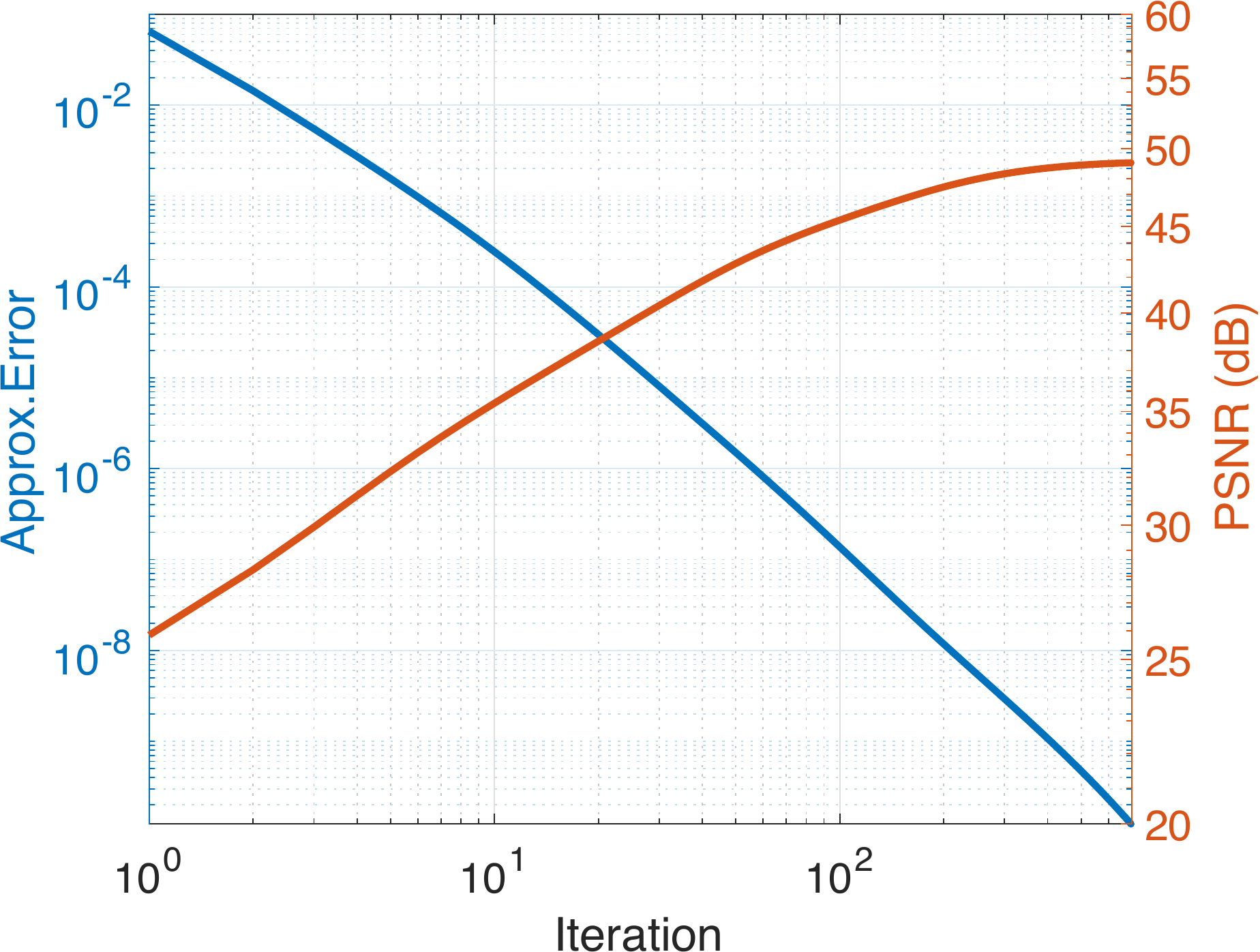}
\caption{Image deconvolution in Example~\ref{ex:im_deblur}. (Top) Comparison between reconstructed images using the regularization filtering and the SCQP-based reconstruction method. (Bottom) Objective function values and PSNR of the estimated image change over the iterations.}
\label{fig_deblur}
\end{figure}

\begin{example}[SCQP as a tool for image deconvolution]\label{ex:im_deblur}
\end{example}
This example demonstrates an application of the SCQP to image deconvolution.
Consider a grayscale image of size $64 \times 64$ (see Fig.~\ref{fig_deblur}(top)), where each pixel is blurred by vertical motion of the width of 5 pixels above and below
\be	
 	\by = \bH \, \bs \notag 
\ee
where $\bs$ and $\by$ are vectorisations of the original and blurred images, respectively, and $\bH$ is a sparse blurring matrix. We note that this matrix is of size 4096 $\times$ 4096, and has rank of 4094. In order to reconstruct the image $\bs$, 
one can apply the Wiener filtering or equivalently solve an optimisation problem which minimises the approximation error and the difference between each pixel and those surrounding it, i.e., to enhance smoothness in the image \cite{REEVES2014165}, as
\be
\min_{\hat{\bx}} \quad \|\by - \bD \hat{\bx} \|_2^2 + \mu \|\bL \, \hat{\bx}\|_2^2 \notag 
\ee 
where $\mu>0$, and $\bL$ is the discrete Laplacian, which plays a role of a high-pass filter. 
Different from the regularization filter, we express the estimated image as $\hat{\bx} = \alpha \bx$ where $\bx$ is a unit-length vector, $\bx^T \bx = 1$, and minimise the reconstruction error 
\be
\min_{\alpha,{\bx}}  \; \|\by - \alpha \bD {\bx} \|_2^2 =  \|\by\|_2^2 + \alpha^2 \,  {\bx^T \bQ \bx } - 2 \alpha \bb^T \bx	\label{eq_deblur_cost} ,\;\;
\mathrm{{s.t.}}  \;\;  \|\bx\|_2^2 = 1 \notag,
\ee
where $\bQ = \bD^T \bD$ and $\bb = \bD^T\by$. It is obvious that the optimal $\alpha^{\star}$ is given by 
\be
	\alpha^{\star} =  \frac{\bb^T \bx}{\bx^T \bQ \bx} \label{eq_deblur_alpha} \, ,
\ee
and $\bx$ is a solution to the following SCQP 
\be
&\min \quad   & {\bx^T \bQ \bx } + \frac{2}{\alpha}  \bb^T \bx  \,, \label{eq_deblur_x} \quad 
\mathrm{{s.t.}}  \quad \bx^T \bx = 1   \, .
\ee
Following this, we perform an alternating estimation process between $\bx $ and $\alpha$. We first initialize a unit-length vector $\bx$, compute $\alpha$ as in (\ref{eq_deblur_alpha}), then update $\bx$ by solving (\ref{eq_deblur_x}), and update $\alpha$ again. The process is executed until there is no significant change in the object function value. 

In Fig.~\ref{fig_deblur}, we show the reconstructed image using the regularisation filtering with $\mu = 0.0025$. The image achieved a PSNR = 24.6 dB. The image reconstructed using the SCQP based method obtained a PSNR =  49.05 dB after 672 iterations, as shown in Fig.~\ref{fig_deblur}(bottom). We note that the performance of the regularization filtering is affected by the choice of the regularisation parameter $\mu$.

\subsection{QP with inequality constraint $\bx^T \bx \le 1$}\label{sec:x_inside_sphere}

For completeness of this section, we now present the QP with an inequality quadratic constraint
\be
&\min  \quad  &\frac{1}{2} \, \bx^T \, \bQ \bx + \bb^T \bx  \label{equ_qp_in_sphere} , \quad \text{s.t.}  \quad   \bx^T \bx \le 1   \, .
\ee
First, the vector $\bx$ is expanded by an extra parameter $z_1$, where  $z_1^2 = 1 - \bx^T \bx \,$, to yield a new unit length vector $\bz = [z_1, x_1, \ldots, x_K]^T$.
The vector $\bz$ is a global minimiser to the following SCQP
\be
&\min  \quad  &\frac{1}{2} \, \bz^T \, \left[ \begin{array}{cc} 0 \\ & \bQ \end{array} \right]  \, \bz + \bb_z^T\, [0, \bb^T]  \notag , \quad
\text{s.t.}  \quad  \bz^T \bz  = 1 \notag  \, .
\ee
Or, in other words, $\tilde{\bz} = \left[\begin{array}{c} z_1 \\ \tilde{\bx}\end{array}\right]$ is a minimiser to a simplified problem 
\be
	&\min  \quad  &\frac{1}{2} \, \tilde{\bz}^T \, \diag(\bs_z) \, \tilde{\bz} + \bc_z^T\, \tilde{\bz} \notag  , \quad \text{s.t.}  \quad   \tilde{\bz}^T \tilde{\bz}  = 1 \notag 
\ee
where $\bc_z  = [0, \bc^T]^T$ and $\bs_z = \frac{1}{\|\bb\|} \,  [ 0, \bsigma]  + 1$.
%
%
%
Since the first entry $c_z(1)$ is zero, the problem falls into the case stated in Lemma~\ref{lem_c1_zero}. 
This can happen in two cases for
\be
d = \sum_{k = 1}^{K} \frac{c_k^2}{(s_z(k+1) - 1)^2} = \|\bb\|^2 \, \sum_{k = 1}^{K} \frac{c_k^2}{\sigma_k^2} 	\notag
\ee
\begin{itemize}
\item  If $d\le 1$, we obtain a solution $\displaystyle \tilde{x}_k = \|\bb\| \frac{c_k}{\sigma_k}$,
\item Otherwise, $\tilde{z}_1  = {z}_1 = 0$, and we solve a QP with the equality constraint $\bx^T \bx = 1$.
\end{itemize}

\begin{example}({\bf A toy example})\label{ex_in_sphere1}
We replicate Example~1 in \cite{ChenGao2013} for a minimization problem (\ref{equ_qp_in_sphere}) with 
$
	\bQ = \left[\begin{array}{cc}
		-1 & 0 \\
		0 & 1
	\end{array}
	\right]$, $
	\bb = \left[\begin{array}{c}
		0 \\ 1.8\end{array}
		\right] 
$.
In order to solve the problem, we expand $\bQ$ with one row and column of zeros, and $\bb$ with one zero entry.
The newly expanded matrix of $\bQ$ has  eigenvalues $[-1, 0, 1]$.
After the normalisation of $\bb$ and $\bQ$, we obtain the expanded parameters 
\be	
	\bc = [0, \frac{\bb}{\|\bb\|}] = [0, 0, 1]^T,		\notag
\ee
and the shifted eigenvalues
\be
	\bs = \frac{[-1, 0, 1]^T +1}{\|\bb\|}  + 1 =   [1, \, 1.5556,\, 2.1111]^T \, .	\notag
\ee
Since $c_2$ is zero, $\tilde{x}_2 = 0$, the problem boils down to finding the two variables $[\tilde{x}_1, \tilde{x}_3]$.
Since $c_1$ is zero, and $d = \frac{c_3^2}{(s_3-1)^2} = 0.81 < 1$, according to Lemma~\ref{lem_c1_zero}, 
\be
\tilde{x}_3 =  \frac{c_3}{s_3-1} = -0.9\notag
\ee and $\tilde{x}_1$ can take one of two values, $\tilde{x}_1 = \pm \sqrt{1 - \tilde{x}_3^2} = \pm 0.4359$.
Finally, we convert $\tilde{\bx} = [\pm 0.4359, 0, -0.9]^T$ to the original space of the expanded vector $\bx$ by multiplying it with the eigenvectors of $\bQ$ to give 
\be
	\bx = \left[\begin{array}{ccc}
		 0 &    1&     0\\
  		 1 &    0   &  0\\
          	0    & 0 &    1
	\end{array}
	\right] \, \tilde{\bx}  = [0, \pm 0.4359, -0.9]^T. 	\notag
\ee
That is, there are two global minimisers $[\pm 0.4359, -0.9]^T$. 
Fig.~\ref{fig_ex_sqp_n2} illustrates the solution of the problem, where the shaded region shows the objective function when the points $[x_1, x_2]$ are on a unit circle. 
\end{example}


\begin{figure}\centering
\includegraphics[width=.8\linewidth, trim = 0.0cm 2cm 0cm 0cm,clip=true]{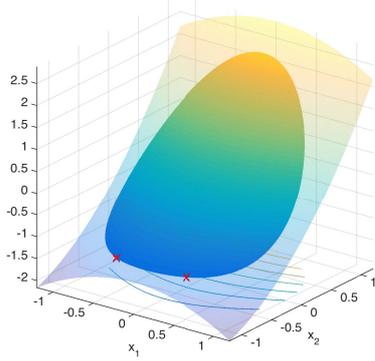}
\caption{Minimization of a quadratic function in a sphere in Example~\ref{ex_in_sphere1} has two global minimisers $[\pm 0.4359, -0.9]$, which are red ``cross'' points.}\label{fig_ex_sqp_n2}
\end{figure}

\begin{example}
We now change the vector $\bb$ in the previous example to $[0, 3]^T$. With this setting, the vector is still $\bc = [0, 0, 1]$, but the eigenvalues are $\bs = [  1, 1.3333,   1.6667]$.
Again, since $c_2 = 0$, we still have $\tilde{x}_2 = 0$.
However, because $d = \frac{c_3^2}{(s_3-1)^2} = 2.25 > 1$, according to Lemma~\ref{lem_c1_zero}, $\tilde{x}_1 = 0$. 
We need to find only $\tilde{x}_3$. For this particular case, it turns out that $\tilde{x}_3 = -1$. Finally the global minimiser is $[0, -1]^T$.
\end{example}

\section{SCQP with Matrix-variates}\label{sec:sqp_matrixvariate}

Consider an extension of the SCQP in (\ref{equ_qp_sphere}) for a matrix-variate. The problem can be formulated for a matrix $\bX$ of size $I\times R$ as 
\be
\min\quad f(\bX) = \frac{1}{2} \, \tr(\bX^T \bQ \bX)  + \tr(\bB^T \bX)  \quad \text{s.t.} \;\; \|\bX\|_F^2 = 1  \label{eq_spq_matrix}
\ee
where $\bQ$ is a psd matrix of size $I \times I$  and $\bB$ is of size $I \times R$.
A straightforward approach to (\ref{eq_spq_matrix}) is to rewrite it in the form of an ordinary SCQP for the vectorisation $\vtr{\bX}$, 
\be
&\min\quad &f(\bX) = \frac{1}{2} \vtr{\bX}^T (\bI_R  \otimes \bQ) \vtr{\bX}  +  \vtr{\bB}^T \vtr{\bX}  \notag \\ &\text{s.t.} \quad& \|\bX\|_F = 1 \, \notag,
\ee
and then apply the algorithm in the previous section to find $\bX$.
The symbol ``$\otimes$'' stands for the Kronecker product. 

An alternative method would be to rewrite the objective function in a form similar to  (\ref{equ_qp_sphere2}), as
\be
f(\bX) 
&=&  \frac{1}{2}  {\bx}^T (\bI_R \otimes \diag(\bsigma) )    {\bx}  +  \bv^T  {\bx} \notag 
\ee
where 
${ \bx} = \vtr{\bX^T \bU}$, $\bv = \vtr{\bB^T \bU}$ and $\bQ = \bU \diag(\bsigma)\bU^T$ is an EVD of $\bQ$. Due to the Kronecker product, each eigenvalue $\sigma_i$, $i = 1, \ldots, I$, is replicated $R$ times. 
Hence, according to Lemma~\ref{lem_sqp_simplify}, 
we can deduce the minimiser to (\ref{eq_spq_matrix}) from the minimiser $\bz^{\star}$ to an SCQP of a smaller scale 
\be
\min\quad \frac{1}{2} \, \bz^T  \diag(\bsigma) \bz  +  \bc^T \bz  \quad \text{s.t.} \quad \bz^T \bz = 1 \, \notag,
\ee
where  $\bc = [c_1, \ldots, c_I]$,  $c_i =  \|\bB^T \bu_i\|$. 
More specifically, $\bx_i = \displaystyle \frac{z_i}{c_i} \bB^T \bu_i$ for a nonzero coefficient $c_i$. Otherwise, $\bx_i$ is often a zero vector for a zero $c_i$, except only the case $c_1 = 0$. 

%
%

\section{SCQP for large scale data}\label{sec:largscaleSCQP}

The most computationally demanding step in the closed-form method for SCQP is the EVD of the matrix $\bQ$ of size $K \times K$. When the vector $\bx$ comprises hundreds of thousands of entries, this computation may not be executed in a computer. To this end, we convert the large scale SCQP to sub-problems of smaller scale, each of which can be solved in closed-form. 

First we partition the index set $\calI = [1, 2, \ldots, K]$ into $L$ disjoint segments $\calI_l$ of size $K_l$, $K = K_1 + K_2 + \cdots + K_L$,  such that EVDs of matrices of size $K_l \times K_l$ can be performed on a computer.
For each sub-vector $\bx_l = \bx(\calI_l)$,  we denote by $\alpha_l$  and $\tilde{\bx}_l$ its $\ell_2$-norm and normalized vector, $l = 1, \ldots, L$, respectively, i.e., $\bx(\calI_l)  = \alpha_l \, \tilde{\bx}_l$ where $\tilde{\bx}_l^T  \tilde{\bx}_l = 1$.
We also denote a complement set by $\calI_{\bar{l}} = \{1, \ldots, K\} \setminus \calI_l$. Similarly, we define $\alpha_{\bar{l}}$ and sub-vectors $\tilde{\bx}_{\bar{l}}$, $\bb_l$ and $\bb_{\bar{l}}$.
 Note that $\bx^T \bx = \alpha_l^2 + \alpha_{\bar{l}}^2 = 1$ and $\alpha_{\bar{l}} = \|\balpha_{m \neq l}\|_2$. Hence, the vector $ [\alpha_l, \alpha_{\bar{l}}]^T$ also has a unit length. 
 
For convenience, we consider again the SCQP problem for $\bx$
\be
\min_{\bx} \quad y = \frac{1}{2} \bx^T \bQ \bx + \bb^T \bx \quad \text{s.t.} \;\; \bx^T \bx = 1\,\label{eq_scqp_x}
\ee 
and rewrite it as SCQP sub-problems for unit-length vectors $[\alpha_l, \alpha_{\bar{l}}]^T$ and $\tilde{\bx}_{l}$, for $l = 1, \ldots, L$.
For example, an SCQP for only two parameters $[\alpha_l, \alpha_{\bar{l}}]^T$ is given by 
\be
\min &&    y  =  \frac{1}{2} [\alpha_l, \alpha_{\bar{l}}] \bT_l [\alpha_l\, \alpha_{\bar{l}}]^T +  [\alpha_l, \alpha_{\bar{l}}] \,   \bu  \label{eq_SCQP_alpha_l}\\
\text{s.t.} &&  \alpha_l^2 + \alpha_{\bar{l}}^2 = 1 \notag 
\ee
where $\bT_l = \left[\begin{array}{c|c} \tilde{\bx}_l^T \bQ_{l,l} \tilde{\bx}_l & \tilde{\bx}_l^T \bQ_{l,\bar{l}} \tilde{\bx}_{\bar{l}}  \\  \hline \tilde{\bx}_l^T \bQ_{l,\bar{l}} \tilde{\bx}_{\bar{l}}  & \tilde{\bx}_{\bar{l}}^T \bQ_{\bar{l},\bar{l}} \tilde{\bx}_{\bar{l}}  \end{array} \right]$ is of size $2 \times 2$ and $\bu = [\bb_{l}^T \tilde{\bx}_{l}, \bb_{\bar{l}}^T \tilde{\bx}_{\bar{l}}]^T$.
The above problem can be straightforwardly solved in a closed-form, while keeping $\tilde{\bx}_l$ fixed. Once $\alpha_l$ and $\alpha_{\bar{l}}$ are updated, the other scaling coefficients $\alpha_m$ for $m \neq l$ are then scaled by a factor of $\displaystyle \frac{\alpha_{\bar{l}}}{\|\balpha_{m \neq l}\|_2}$ 
\be
\alpha_{m \neq l} \leftarrow \frac{\alpha_{\bar{l}}}{\|\balpha_{m \neq l} \|_2}  \, \alpha_{m \neq l}   \,. \label{eq_update_alpha_m}
\ee
Next we rewrite the SCQP for $\bx$ in (\ref{eq_scqp_x}) as an SCQP for $\tilde{\bx}_{l}$, for $l = 1, \ldots, L$, while keeping the other parameters fixed as
\be
\min_{\tilde{\bx}_l} && y = \frac{\alpha_l^2}{2} \tilde{\bx}_l^T \bQ_{l,l} \tilde{\bx}_l  +  \alpha_l (\bb_l +   \bQ_{l,\bar{l}}  {\bx}_{\bar{l}})^T \tilde{\bx}_l   +  c_l\label{eq_SCQP_xl}\\
\text{s.t.} && \tilde{\bx}_l^T \tilde{\bx}_l = 1\,, \notag
\ee
where $c_l$ is independent of $\tilde{\bx}_l$. 
Because $\bQ_{l,l}$ are of relatively small sizes $K_l \times K_l$, update of $\tilde{\bx}_l$ can be proceeded in closed-form.

Finally, by alternating between the updates in (\ref{eq_SCQP_alpha_l}), (\ref{eq_update_alpha_m}) and (\ref{eq_SCQP_xl}), we can update entire parameters $\alpha_l$ and $\tilde{\bx}_l$. We summarise the update procedure in Algorithm~\ref{alg_qp_sphere_blockupdate}. For each partitioning of $[1, \ldots, K]$, EVDs of $\bQ_{l,l}$ are computed only once, then we perform an inner loop to update $\tilde{\bx}_l$ and $\alpha_l$ until there is no further improvement.

\setlength{\algomargin}{1em}
\begin{algorithm}[t!]
\SetFillComment
\SetSideCommentRight
\CommentSty{\footnotesize}
\caption{Block update for large scale SCQP}\label{alg_qp_sphere_blockupdate}
\DontPrintSemicolon \SetFillComment \SetSideCommentRight
\KwIn{$\bQ$ and $\bb$} 
\KwOut{$\bx$  minimises $\frac{1}{2} \, \bx^T \bQ \bx + \bc^T \bx$, s.t., $\bx^T \bx = 1$} 
\SetKwFunction{card}{card} 
\SetKwFunction{coreqps}{qps\_nnz} 
\Begin{
\Repeat{a stopping criterion is met}{
\nl Partition $\{1, 2, \ldots, K\}$ into $L$ disjoint segments in random\;
\nl Precompute EVD of $\bQ_{l,l}$ for $l = 1, \ldots, L$\;
\Repeat{a stopping criterion is met}{
\For {$l = 1, \ldots, L$}{
\nl Solve SCQP in (\ref{eq_SCQP_xl}) to update $\tilde{\bx}_l$ of length $K_l$ \;
\nl Solve SCQP in (\ref{eq_SCQP_alpha_l}) to update $[\alpha_l, \alpha_{\bar{l}}]$ \;
\nl Adjust $\alpha_{m \neq l} \leftarrow \frac{\alpha_{\bar{l}}}{\|\alpha_{m \neq l} \|_2}  \, \alpha_{m \neq l}$\; 
}
}
}
}
\end{algorithm}

\section{Linear Regression with Bound Constraint}\label{sec:regression}

Another problem, which can be formulated as SCQP, is the linear regression with a constrained bound on the regression error 
\be
\min_{\bx}  \quad & \|\bx\|^2   \quad 
\text{subject to} &   \|\by - \bA \bx \| \le  \delta  \label{eq_linreg_bound} \ , 
\ee 
where $\by$ is a vector of length $I$ of dependent variables, $\bA$ is a regressor matrix of size $I \times K$ and $\delta$ a nonnegative regression bound. 

It is obvious that if $\delta \ge \|\by\|$, then the zero vector $\bx = \0$ is a minimiser to (\ref{eq_linreg_bound}).
Therefore, in order to achieve a meaningful regression, the regression bound $\delta$ needs to be in the following range. 

\begin{lemma}[Range of the bound $\delta$]
The problem (\ref{eq_linreg_bound}) has a minimiser of nonzero entries when 
\be
\|{\bPi}_{\bA}^{\perp}  \, \by \|  \le \delta <  \| \by\|		\notag
\ee
where  ${\bPi}_{\bA}^{\perp}$ is an orthogonal complement of the column space of $\bA$.
\end{lemma}
\begin{proof}
Let $\bU$ be an orthogonal basis for the column space of $\bA$. Then 
\be
	\delta^2 \ge \|\by - \bA \bx \|^2 = \|\bU^T \by - \bU^T \bA \bx\|_F^2 + \|{\bPi}_{\bA}^{\perp}  \, \by\|^2  \ge \|{\bPi}_{\bA}^{\perp}  \, \by\|^2 \,. \notag 
\ee
\end{proof}

For simplicity, we assume that $\bA$ is full rank matrix, otherwise, we solve the problem with a compressed regressor matrix with a smaller bound 
\be
\min\quad \|\bx\|^2 \quad \text{subject to} \quad \|\hat{\by} - \hat{\bA} \bx\|  \le  \hat{\delta}	\notag
\ee	
where $\hat{\by} = \bU^T\by$, $\hat{\bA} = \bU^T \bA$, and $\hat{\delta}^2 = \delta^2 - \|{\bPi}_{\bA}^{\perp}  \, \by\|^2$.

We shall now derive an equivalent SCQP to the problem in (\ref{eq_linreg_bound}). We first show that the inequality sign in (\ref{eq_linreg_bound}) can be replaced by the equal sign.
\begin{lemma}\label{lem_linreg_circle} 
The minimiser to (\ref{eq_linreg_bound}) is the minimiser to the following problem 
\be
\min_{\bx}  \quad & \|\bx\|^2 \quad 
\mathrm{{s.t.}} &   \|\by - \bA \bx \|  =  \delta   \label{eq_linreg_circle}
\ee
\end{lemma}
See the proof in Appendix~\ref{sec:proof_lem_linreg_circle}.

%



The proposed algorithms to solve the problem in (\ref{eq_linreg_circle}) are presented for two cases, when the length of $\bx$ does not exceed the number of dependent variables, $K\le I$, and when $K > I$.

\subsection{The case when $ K \le I$}\label{sec:case1}

We first consider the case when the matrix of regressors $\bA$ is of full column rank, $ K \le I$.
Let $\bA = \bU \diag(\bs) \bV^T$ be an SVD of $\bA$, where  $\bV$ is an orthonormal matrix of size $K \times K$, and $\bs = [s_1, \ldots, s_K]> 0$. Hence ${\bPi}_{\bA}^{\perp}  = \bI - \bU \, \bU^T$.

Let $\hat{\by} = \bU^T \by$,  $\hat{\delta} = \sqrt{\delta^2  -   \|{\bPi}_{\bA}^{\perp}  \, \by\|^2}$, 
$\bz = \displaystyle \frac{1}{\hat{\delta}} (\hat{\by} -  \diag(\bs) \bV^T \bx)$, then 
\be
\bx &=& \bV \diag(\bs^{-1})(\hat{\by}- \hat{\delta} \bz) 	\notag \\
\|\bx\|_F^2 &=&  (\hat{\by} - \hat{\delta} \bz)^T \diag(\bs^{-2})  (\hat{\by} -\hat{\delta}  \bz) \notag \\
\|\by - \bA \bx \|^2 &=&  \|\bU^T \by - \diag(\bs) \bV^T \bx\|_F^2 + \|{\bPi}_{\bA}^{\perp}  \, \by\|^2  \notag \\ &=& 
\hat{\delta}^2\, \|\bz\|^2 + \|{\bPi}_{\bA}^{\perp}  \, \by\|^2 \, .\notag 
\ee
By this reparameterization, the problem in (\ref{eq_linreg_circle}) becomes an SCQP which can be solved in closed-form
\be
\min_{\bz}  \quad    \bz^T \diag(\hat{\delta} \bs^{-2}) \bz - 2  \, \hat{\by}^T \diag(\bs^{-2}) \bz    \notag  \quad
\text{s.t.} \;\;  \bz^T \bz  = 1 \notag.
\ee
 
\subsection{The case when $K > I$}\label{sec:case2}
  
For this case, we develop an iterative algorithm, at each iteration, wherebt the problem (\ref{eq_linreg_circle}) is rewritten as a subproblem with an invertible regressor matrix. 
To this end, we first generate an initial feasible point $\bx^{(0)}$ such that $\|\by - \bA \bx^{(0)}\| = \delta$.

We then select a sub matrix $\bA_{\calI}$ of $\bA$ such that $\bA_{\calI}$ is invertible, $2 \le \text{card}(\calI) \le I$  and at least one entry $x_{k\in {\calI}}^{(0)}$ is non-zero.
Let $\bz_{\calI}^{(0)} = \by - \sum_{k \notin \calI} \ba_k x_k^{(0)}$, then 
\be
	\|\by - \bA\bx\| = \|\bz_{\calI}^{(0)} - \bA_{\calI} \, \bx_{\calI}^{(0)} \|  = \delta.\notag 
\ee
By fixing the parameters $x_k$, $k \notin \calI$,  the new estimate $\bx_{\calI}^{(1)}$ of $\bx_{\calI}$ is a minimiser to the following problem  
 \be
\bx_{\calI}^{(1)} = \argmin_{\bx_{\calI}} \quad \|\bx_{\calI}\|^2 \quad \text{subject to} \quad \|\bz_{\calI}^{(0)} - \bA_{\calI}\bx_{\calI}\|  =  \delta.	\notag
 \ee
Since $\bA_{\calI}$ is invertible and $\bx_{\calI}^{(0)}$ is a non-zero point which holds the constraint, the above constrained QP has a non-zero global minimiser, which can be solved in closed-form as in the case in Section~\ref{sec:case1}. 
As a result of this update, the new estimate $\bx^{(1)}$ still holds the constraint, and 
\be
\|\bx^{(1)}\|^2= \|\bx^{(0)}\|^2 -  \|\bx_{\calI}^{(0)}\|^2  + \|\bx_{\calI}^{(1)}\|^2 \le  \|\bx^{(0)}\|^2 \,.	\notag
\ee
The algorithm then selects another index set $\calI$, and continues updating the entries $\bx_{\calI}$  by a non-zero $\bx_{\calI}^{(t)}$. This alternating update scheme generates a sequence of estimates $\bx^{(t)}$, which preserve the constraint $\|\by - \bA\bx^{(t)}\| = \delta$, while keeping their norm non-increasing,  $\|\bx^{(t)}\| \le \|\bx^{(t-1)}\| \le \ldots \le \|\bx^{(0)}\|$.

The linear regression with a bound error constraint has found novel applications in the error preserving correction methods for the Canonical Polyadic tensor Decomposition or the CPD with bounded norm of rank-1 tensors\cite{Phan_CPnormcorrection}.

 \section{Quadratic Programming with Elliptic Constraints}\label{sec:qpqc}

Consider a QP with multiple quadratic constraints, each representing an ellipsoid, so that the feasible set is an intersection of the ellipsoids. This problem has been extensively studied in the literature, and arise in many applications in phase recovery, power flow, MIMO detection, quadratic-assignment, sensor-network localization, max-cut problems. For comprehensive review of the problem and its applications, we refer to \cite{Nesterov2000,DBLP:journals/spm/LuoMSYZ10}.

\begin{definition}[Quadratic programming over ellipsoids]\label{prob_qpss}
Consider a positive semi-definite matrix $\bQ$ of size $K \times K$, a  vector $\bb$ of length $K$, and a set of $M$ positive semi-definite matrices $\bH_m$. The quadratic programming over ellipsoids solves the optimisation problem 
\be
&\min  \quad  &\frac{1}{2} \, \bx^T \, \bQ \bx + \bb^T \bx  \label{equ_qp_qc} \\
&\mathrm{{s.t.}}  \quad & \bx^T \, \bH_m \, \bx = 1, \quad m = 1, \ldots, M \, . \notag 
\ee
\end{definition}
The constraints in the above programming are given in a simple form without linear terms as in the objective function. In practice, however, the full quadratic forms can be converted to the homogenised form of the parameter vector $\left[1, \bx^T\right]^T$, e.g., $\left[1,\bx^T\right] \left[\begin{array}{@{}c@{\hspace{1ex}}c@{}} 0 & \bb^T\\[-1ex] \bb & \bQ\end{array}  \right] \, \left[\begin{array}{@{}c@{\hspace{1ex}}c@{}} 1 \\[-1ex] \bx\end{array}  \right] $ \cite{Nesterov2000}. 

In addition, the case with inequality constraints, i.e., $\bx^T \bH_m \bx \le 1$, can also be converted to the equality constraints by introducing additional variables $\bs = [s_1, s_2, \ldots, s_M]$ such that 
\be
1 = \bx^T \bH_m \bx + s_m^2 = [\bx^T, \bs^T] \left[\begin{array}{cc}\bH_m \\[-1ex]  & \ve_m \ve_m^T \end{array}\right]  \left[\begin{array}{@{}c@{}c@{}}\bx \\[-1ex]  \bs \end{array}\right]  \,.\notag 
\ee 

For the above quadratically constrained quadratic programming (QCQP) problem, we can apply relaxations to find approximate solutions, e.g., the Lagrangian and Semidefinite Programming (SDP) based relaxations. The SDP relaxation introduces a symmetric matrix of rank-1, $\bX = \bx \bx^T$, 
and relaxes the condition to the semidefinite condition $\bX \succeq \bx \bx^T$.
The quadratic objective and constraint functions can then be rewritten in linear form, see \cite{Baron72,Kim2003,Bao2011} for relaxations for QCQP.

Different from the existing methods, we introduce an augmented Lagrangian based algorithm for the problem in (\ref{equ_qp_qc}). The constraints over multiple ellipsoids are interpreted as a constraint over a sphere and an orthogonal projection. In order to achieve this, we define symmetric matrices $\bD_n = \bH_1 - \bH_{n+1}$, and rewrite the optimisation problem in (\ref{equ_qp_qc}) in the form of 
\be
&\min  \quad  &\frac{1}{2} \, \bx^T \, \bQ \bx + \bb^T \bx  \notag \\
&\text{s.t.}  \quad & \bx^T \, \bH_1 \, \bx = 1,  \; \;
   \bx^T \, \bD_n \, \bx = 0 , \quad n = 1, \ldots, N\notag 
\ee
or in the following form 
\be
&\min  \quad  &\frac{1}{2} \, \tilde{\bx}^T \, \tilde{\bQ} \, \tilde{\bx} + {\tilde{\bb}}^T \, \tilde{\bx}  \label{equ_qp_sphere_proj_td} \\
&\text{s.t.}  \quad & \tilde{\bx}^T \tilde{\bx} = 1 \;\;\text{and}\quad
  \tilde{\bx}^T \, \tilde{\bD}_n \, \tilde{\bx} = 0 , \quad n = 1, \ldots, N\notag 
\ee
after a reparameterization 
\begin{align}
	\tilde{\bx} &= \bF^T \bx 	\, , 	&\tilde{\bQ} &=  \bF^{-1}  \bQ \bF^{-1\,  T}  \, ,\notag\\
	\tilde{\bb} &= \bF^{-1}  \bb \,, 	&\tilde{\bD}_n &= \bF^{-1}  \bD_n  \bF^{-1\, T} \notag,
\end{align}
where $\bF$ is the Cholesky factor matrix of $\bH_1 = \bF \, \bF^T$.

For simplicity of notation, we will solve the problem in (\ref{equ_qp_sphere_proj_td}) with parameters $\bQ$, $\bZ$ and $\bb$ and variables $\bx$ with $\bx^T \bx = 1$, that is
\be
&\min  \quad  &\frac{1}{2} \, {\bx}^T \, {\bQ} \, {\bx} + {{\bb}}^T \, \tilde{\bx}  \label{equ_qp_sphere_proj} \\
&\text{s.t.}  \quad & {\bx}^T {\bx} = 1, \notag \\
&\quad &  {\bx}^T \, {\bD}_n \, {\bx} = 0 , \quad n = 1, \ldots, N\, .\notag 
\ee

\subsection{An Augmented-Lagrangian algorithm for QCQP}

  \setlength{\algomargin}{1em}
\begin{algorithm}[t!]
\SetFillComment
\SetSideCommentRight
\CommentSty{\footnotesize}
\caption{Augmented Lagrangian Algorithm for QCQP}\label{alg_qpqc}
\DontPrintSemicolon \SetFillComment \SetSideCommentRight
\KwIn{Matrices $\bQ$, $\bD_1$, \ldots, $\bD_N$  and $\bb$} 
\KwOut{$\bx$  minimises \quad $\frac{1}{2} \, \bx^T \bQ \bx + \bb^T \bx$, \quad s.t., \,\,$\bx^T \bx = 1$,\,\, $\bx^T \bD_n \bx = 0$} 
\SetKwFunction{card}{card} 
\SetKwFunction{qps}{QPS} 
\SetKwFunction{projD}{${\tt{\Pi}}_{\bD}^{\perp}$} 
\Begin{
\nl Initialize $\by$ and $\bz$ as zero vectors and $\gamma > 0$\;
\Repeat{a stopping criterion is met}{
\nl $\bT =  {\tt{reshape}}(\bz-  \frac{\by}{\gamma},[K \times K])$, $\bT_x = \frac{1}{2}(\bT + \bT^T)$\;
\nl $\bx = \qps(\bQ+\gamma \bI - 2 \gamma \bT_x, \bb)$\;
\nl $\bz = \projD(\bx \otimes \bx + \frac{\by}{\gamma})$\;
\nl $\by \leftarrow \by + \gamma (\bx \otimes \bx - \bz)$\;
\nl Adjust $\gamma \leftarrow \alpha \gamma $ if the objective function tents to a slow convergence
}
}	
\end{algorithm} 

In order to solve the QCQP in (\ref{equ_qp_sphere_proj}), we split the problem into two subproblems, each with a single constraint, by introducing an additional variable $\bz$
\be
&\min \quad &f(\bx)  +  g(\bz) \label{equ_f_g_x_z}   
\quad \text{s.t.} \quad  \bz - \bx \otimes \bx = \0  
\ee
where 
$f(\bx)$ is the function of $\bx$ over a sphere for the optimization problem 
\be
&\min  &\frac{1}{2} \bx^T \, \bQ \, \bx + \bb^T \, \bx \notag \,, \quad
\text{subject to}  \quad  \bx^T  \, \bx = 1, \notag 
\ee
and $\bg(\bz)$ represents the projection onto subspace span by orthogonal complement of $\bD = [\vtr{\bD_1}, \ldots, \vtr{\bD_N}]$, 
$	\bD^T \bz = \0$.

The augmented Lagrangian function of the problem (\ref{equ_f_g_x_z}) is given by 
\be
\calL(\bx, \by, \bz) = f(\bx) + g(\bz)  + \by^T (\bx \otimes \bx - \bz) + \frac{\gamma}{2} \|\bx \otimes \bx - \bz \|^2 	\notag
\ee
where $\gamma >0$. The algorithm consists of update rules for the variables $\bx$, $\bz$ and $\by$ 
\be
\bx &=& \argmin  \quad  f(\bx)  +  \by^T (\bx \otimes \bx - \bz) + \frac{\gamma}{2} \|\bx \otimes \bx - \bz \|^2  		\, \label{prob_solve_x} \,,\\
\bz &=& \argmin  \quad  g(\bz)  +  \by^T (\bx \otimes \bx - \bz) + \frac{\gamma}{2} \|\bx \otimes \bx - \bz \|^2  		\, \label{prob_solve_z}  \,,\\
\by &\leftarrow& \by + \gamma (\bx \otimes \bx - \bz) \, .
\ee

\subsection{Estimation of $\bx$}\label{sec::estimate_x}

The optimisation problem in (\ref{prob_solve_x}) can indeed be written as an SCQP, as follows 
\begin{align}
\bx &= \argmin  \; f(\bx)  +  \by^T (\bx \otimes \bx - \bz) + \frac{\gamma}{2} \|\bx \otimes \bx - \bz \|^2  		\,\notag \\
&= \argmin   \;  f(\bx)  + \frac{\gamma}{2} \|\bx \otimes \bx - \bz   +   \frac{\by}{\gamma}\|^2  		\, 
\label{equ_update_x_1} \\
&= \argmin   \;  f(\bx)  + \frac{\gamma}{2}  (\bx^T \bx)^2  - \gamma (\bx \otimes \bx)^T (\bz   -   \frac{\by}{\gamma}) \notag \\
&= \argmin \;  \frac{1}{2} \, \bx^T \, (\bQ  + \gamma \bI - 2\gamma \bT_x) \, \bx + \bb^T \bx   \quad \text{s.t.} \; {\bx^T \bx = 1} \,,\;\label{equ_update_x}
\end{align}
where $\bT_x = \frac{\bT + \bT^T}{2}$ is a symmetric matrix of size $K \times K$, 
$\bT = {\tt{reshape}}(\bz-  \frac{\by}{\gamma},[K \times K])$.
At each iteration to update $\bx$, we reshape the vector $(\bz- \frac{\by}{\gamma})$ to a matrix of size $K \times K$, and then construct  a symmetric matrix $\bT_x$. The matrix $\bQ$ is changed by a term $\gamma \bT_x$, while the vector $\bb$ is preserved. 
The last equation indicates that the vector $\bx$ can be found in closed-form using Algorithm~\ref{alg_qp_sphere}.

\subsection{Estimation of $\bz$}\label{sec::estimate_z}

The vector $\bz$ is updated as a minimiser to the following problem 
\be
\bz &=& \argmin  \quad  g(\bx)  +  \by^T (\bx \otimes \bx - \bz) + \frac{\gamma}{2} \|\bx \otimes \bx - \bz \|^2  		\,\notag\\
&=& \argmin  \quad  g(\bx) +  \frac{\gamma}{2} \|\bz - \bx \otimes \bx - \frac{\by}{\gamma} \|^2  		\,\notag\\
&=& \argmin  \quad   \|\bz - \bx \otimes \bx - \frac{\by}{\gamma} \|^2  	\qquad \text{s.t.}	 \;\;  \bD^T \bz = 0  \notag \\
&=&  {\tt{\Pi}}_{\bD}^{\perp}(\bx \otimes \bx + \frac{\by}{\gamma} ) \, , \label{equ_update_z}
\ee
  where ${\tt{\Pi}}_{\bD}^{\perp}(\bz)$ is the orthogonal projection of the vector $\bz$ onto the orthogonal complement of the column space of $\bD$, e.g., 
  \be
  {\tt{\Pi}}_{\bD}^{\perp}(\bz) = \bz - \bD \, (\bD^T\, \bD)^{-1} \, \bD^T \, \bz. \notag 
\ee

\subsection{Algorithm for QCQP}
  
An augmented Lagrangian based algorithm for QCQP including the updates in (\ref{equ_update_x}) and (\ref{equ_update_z}) is summarised in Algorithm~\ref{alg_qpqc}. The vectors $\by$ and $\bz$ are initialised as zeros, while the parameter $\gamma$ is set to a sufficiently high value. 
Experiments show that running the algorithm with a small $\gamma$ at the beginning will decrease the objective function quickly,
but it may make the algorithm unstable after several to a dozen of iterations. 
However, setting $\gamma$ to a too large value will slow down the convergence of the algorithm.
In order to obtain a good setting, we should run the algorithm for a few iterations for various values of $\gamma$,
then choose the setting which gives a good convergence result. The algorithm is then executed using the chosen parameters.
In our experience, $\gamma$ can be set to a fraction of the minimum condition number of $\bH_m$, while the associated matrix $\bH_m$ should be chosen to present the quadratic constraint, i.e., $\bx^T \bH_m \bx = 1$.

During the estimation process, the high value of parameter $\gamma$ should be reduced if the objective function becomes stable, or yields a slow convergence.
However, reducing $\gamma$ too much can cause a divergence, and the parameters should be corrected.

 \subsection{Linearisation for the update of $\bx$}\label{sec::linearisation_x}
 
As per derivation in (\ref{equ_update_x}), $\bx$ is updated as a minimiser of an SCQP. The algorithm iterates to update $\bx$ over $\bz$ and $\by$. As in Algorithm~\ref{alg_qp_sphere}, at each iteration to update $\bx$, one needs to compute the eigenvalue decomposition of the matrix $\bQ + \gamma \bI - 2\gamma \bT_x$, which is changed by a term $\bT_x$. In order to accelerate the update rule, we perform the following linearisation which can bypass the matrix $\bT_x$  in the quadratic term.

From (\ref{equ_update_x_1}), $\bx$ is updated as a minimiser to an SCQP
 \be
 \bx &=& \argmin   \quad  f(\bx)  + \frac{\gamma}{2} \|\bx \otimes \bx - \bz   +   \frac{\by}{\gamma}\|^2  		\,\notag \\
 &=& \argmin   \quad  f(\bx)  + \frac{\gamma}{2} \|\bx \bx^T -  \bT_x \|_F^2  		\,\notag .
 \ee
 Now, we replace the second term in the above problem by its linearization at the previous update denoted by $\bx_o$
 \begin{align}
\bx &= \argmin   \;\;  f(\bx)  + \frac{\gamma}{2} \left( c(\bx_{o}) + \left(\frac{\partial c(\bx)}{\partial \bx}(\bx_o)\right)^T \,  (\bx - \bx_o)  \right.
\notag \\  & \qquad\qquad\qquad\qquad\qquad \left.+ {\mu}\|\bx - \bx_o\| _F^2  \right) 		\,\notag \\
&= \argmin   \;\;  f(\bx)  + \frac{\gamma}{2}  \left(4 (\bx_o \bx_o^T - \bT_x)\bx_o \right)^T \,  \bx + \frac{\gamma\mu}{2} \|\bx - \bx_o\| _F^2  		\,\notag \\
&= \argmin   \;\;  f(\bx)  + 2\gamma  \, \bx_o^T (\bx_o \bx_o^T - \bT_x)\,  \bx + \frac{\gamma\mu}{2} \|\bx - \bx_o\| _F^2  		\,\notag \\
%
 %
 &= \argmin   \;  \frac{1}{2}  \, \bx^T \left( \bQ   +   \gamma\mu  \bI \right) \, \bx + ( \bb + \gamma(2-\mu) \bx_o  - 2\gamma \bT_x \bx_o)^T  \bx   \,  \label{equ_update_x_linearize} \\
 & \quad \text{s.t.} \quad \bx^T \bx = 1 \notag 
 \end{align}
where $c(\bx) = \|\bx \bx^T - \bT_x\|_F^2$ and $\mu > 0$.

The optimisation in (\ref{equ_update_x_linearize}) shows that the quadratic term adjusts $\bQ$ by a term $\gamma \mu \bI$, i.e., shifting its eigenvalues by $\gamma \mu$. So we need not decompose the matrix $\bQ$ again. Different from the update in (\ref{equ_update_x}), the linear term changes by the previous estimate of $\bx$.

For this new update (\ref{equ_update_x_linearize}), the algorithm is computationally cheaper, and still preserves the convergence as that in (\ref{equ_update_x}).

\subsection{Generating a positive definite matrix $\bH_1$}\label{sec:gen_psdmatrix}

In the conversion of the QCQP problem in (\ref{equ_qp_qc}) to the problem (\ref{equ_qp_sphere_proj}), the constraints over multiple ellipsoids are interpreted as a constraint over a sphere and orthogonality constraints.
Choosing a psd matrix $\bH_k$ from the set of matrices $\bH_n$ plays an important role and affects the entire estimating process.
Here, a simple condition is that the selected matrix $\bH_k$ should have a low condition number.
In some cases, it is better to generate a new psd matrix $\bH_{\alpha}$ rather than choosing one among $\bH_n$. 
The new matrix $\bH_{\alpha}$ should have a condition number as small as possible by solving an eigenvalue problem (EVP)  \cite{BEFB:94}
\be	
&\min_{\alpha,\gamma}  \quad  & \gamma  \, , \quad
\text{s.t.}  \quad 		\bI    <  {\bH}_{\alpha} < \gamma   \bI \notag
\ee 
 where 
$	{\bH}_{\alpha} = \sum_{n = 1}^{N}   \bH_n \alpha_n$.

The problem can also be formulated as a semidefinite programming (SDP) problem, which can be solved using the SEDUMI or TFOCS toolboxes.  
The generated matrix $\bH_{\alpha}$ is then scaled by a factor of $\displaystyle \left(\sum_{n = 1}^{N} \alpha_n \right)$, so that it satisfies the quadratic constraint $\bx^T \bH_{\alpha} \bx = 1$.

Alternatively, the matrix $\bH_{\alpha}$ can be generated so that its Frobenius norm is minimum, i.e., 
\be
	&\min 	 \quad  & \|{\bH}_{\alpha}\|_F^2 	\notag \\
	&\text{subject to} 	\quad & {\bH}_{\alpha} = \sum_{n = 1}^{N}   \bH_n \alpha_n  , \quad 
 \sum_{n = 1}^{N} \alpha_n = 1, \quad \alpha_n \ge 0.  \notag
\ee 
For the latter problem, we define a matrix $\bH = $ $[\vtr{\bH_1}, \ldots, \vtr{\bH_N}]$, and find a vector $\balpha = [\alpha_1, \ldots, \alpha_N]$ in a quadratic programming with a linear constraint
\be
	&\min 	 \quad  & \balpha^T  (\bH^T \bH) \, \balpha , \quad 
	\text{s.t.}   \;\;  \1^T\, \balpha = 1, \quad \alpha_n \ge 0. \notag
\ee


\begin{example}(\bf Effects of the condition number of the psd matrix involving in the quadratic constraint)\label{ex_qpqc}
\end{example}

In this example, we consider a QP problem of a psd matrix $\bQ$ of size $10 \times 10$ over quadratic constraints for three psd matrices $\bH_1$,  $\bH_2$ and  $\bH_3$. The linear term in the QP problem is with a zero vector $\bb$.
The matrices are randomly generated, and the condition numbers of $\bH_n$ are 185.7, 12403 and 1000.1, respectively. 
Since the second matrix $\bH_2$ has a high condition number, we generate a new psd matrix $\tilde{\bH}_2$ from $\bH_1$, $\bH_2$ and $\bH_3$ 
as described in Section~\ref{sec:gen_psdmatrix}. The new matrix has a low condition number of 8.9, and is used in place of the matrix $\bH_2$.

In the first analysis, we compare the performances and convergence behaviour of the QCQP algorithm (Algorithm~\ref{alg_qpqc}) when each matrix, either $\bH_i$ or $\tilde{\bH}_2$,
is selected to represent the quadratic constraint, i.e., $\bx^T \bH_i \bx = 1$, and the remaining two matrices represent the orthogonality constraints 
of the parameter vector $\bz$. There are four possible selections of the matrix. The parameter $\gamma$ is fixed to 0.1 in the test. 
Fig.~\ref{fig_qpqc} shows the objective values to illustrate the convergence and final performance, 
and the $\ell_2$-norm of the orthogonality constraints $\|\bD^T (\bx \otimes \bx) \|_2^2$ to verify if all the quadratic constraints $\bx^T \bH_i \bx = 1$ are achieved.
The results indicate that when the matrix with a high condition number, $\bH_2$ or $\bH_3$, plays as a quadratic constraint, the algorithm converges to a false local minima, which do not satisfy the orthogonality constraints $\bD^T (\bx \otimes \bx) = \0$. 

When running the optimisation with a quadratic constraint over the matrix $\bH_1$ or $\tilde{\bH}_2$, the algorithm converges to the same value of $4.8659 \times 10^{-4}$
with a norm  $\|\bD^T (\bx\otimes \bx )\|_2^2$ at level of  $10^{-7}$ and $10^{-10}$, respectively. An important result is that the algorithm needs only 526 iterations 
to achieve such high accuracy with the matrix $\tilde{\bH}_2$, while it needs at least 10000 iterations for the problem with a quadratic constraint over the matrix $\bH_1$.

\begin{figure}[t]
\centering
\includegraphics[width=.41\textwidth, trim = 0.0cm .0cm 0cm 0cm,clip=true]{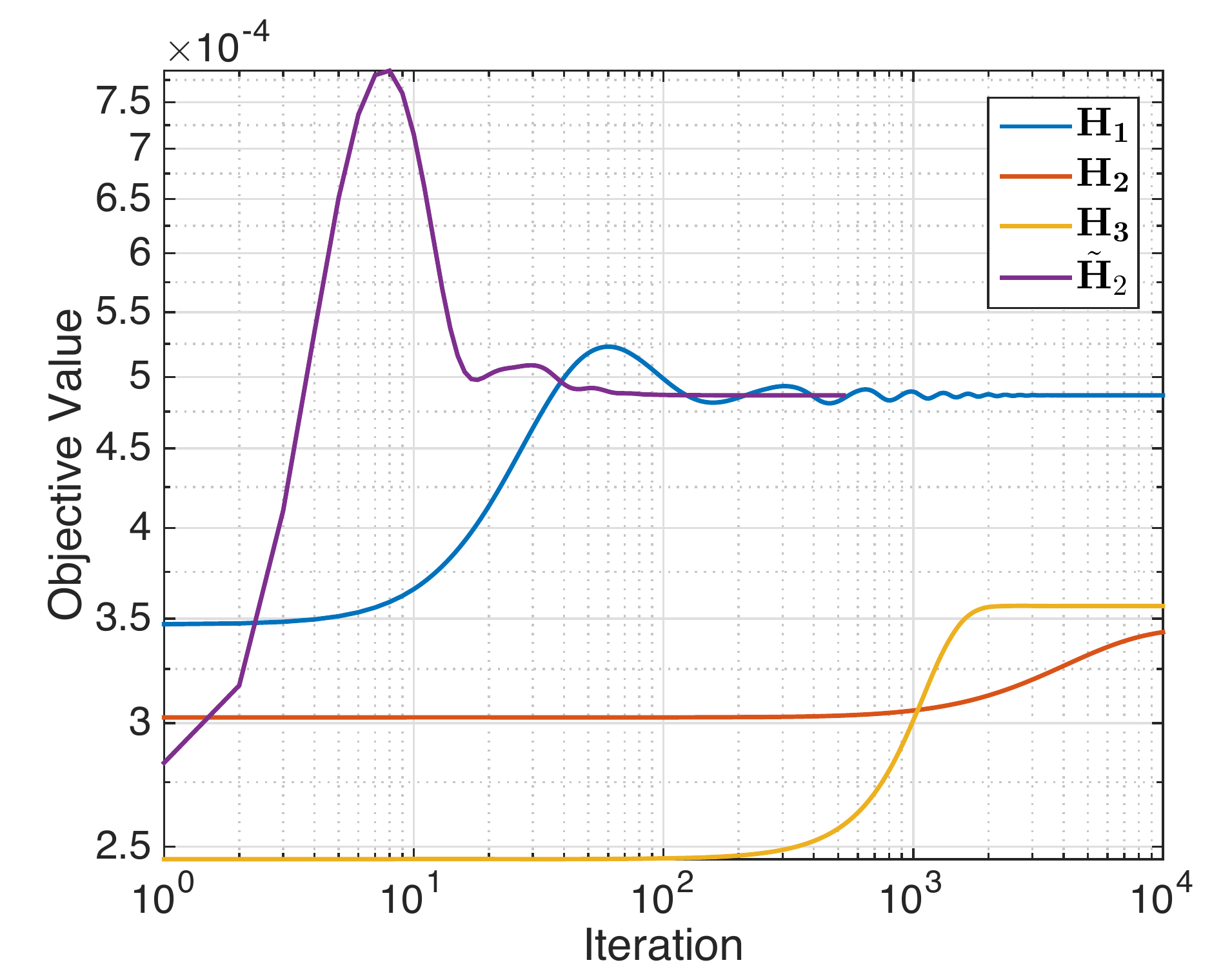}\\
\includegraphics[width=.43\textwidth, trim = 0.0cm .0cm 0cm 0cm,clip=true]{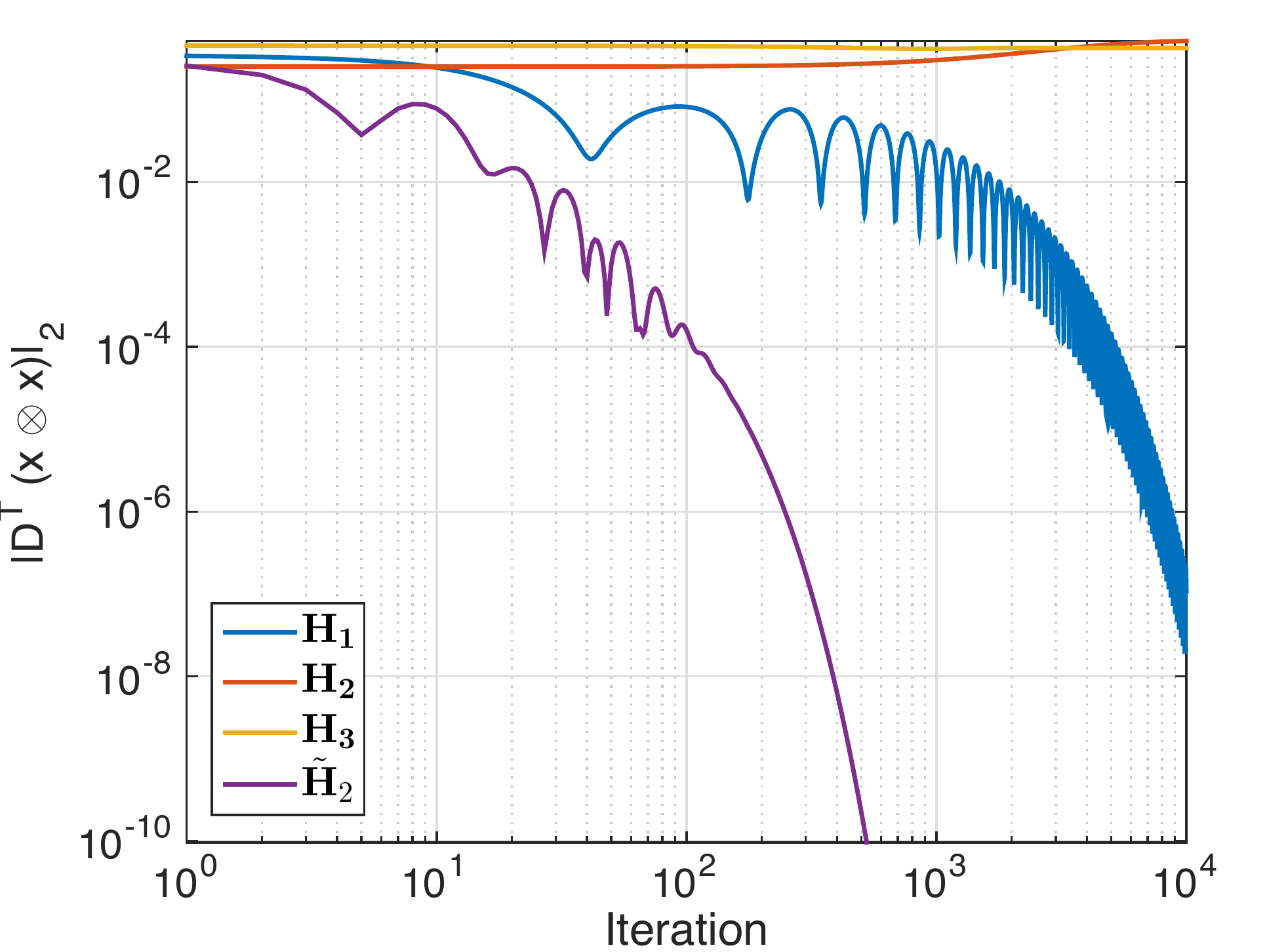}
\caption{Illustration of performances of the proposed algorithm in Example~\ref{ex_qpqc}. 
The algorithm demands a huge number of iterations, but converges to local minima, when the quadratic constraint $\bx^T \bH \bx = 1$
is constrained over the matrix $\bH_2$ or $\bH_3$. The algorithm quickly converges when the quadratic constraint is accompanied by the matrix $\tilde{\bH}_2$.
}
\label{fig_qpqc}
\end{figure}

\begin{figure}[t]
\centering
\includegraphics[width=.41\textwidth, trim = 0.0cm .0cm 0cm 0cm,clip=true]{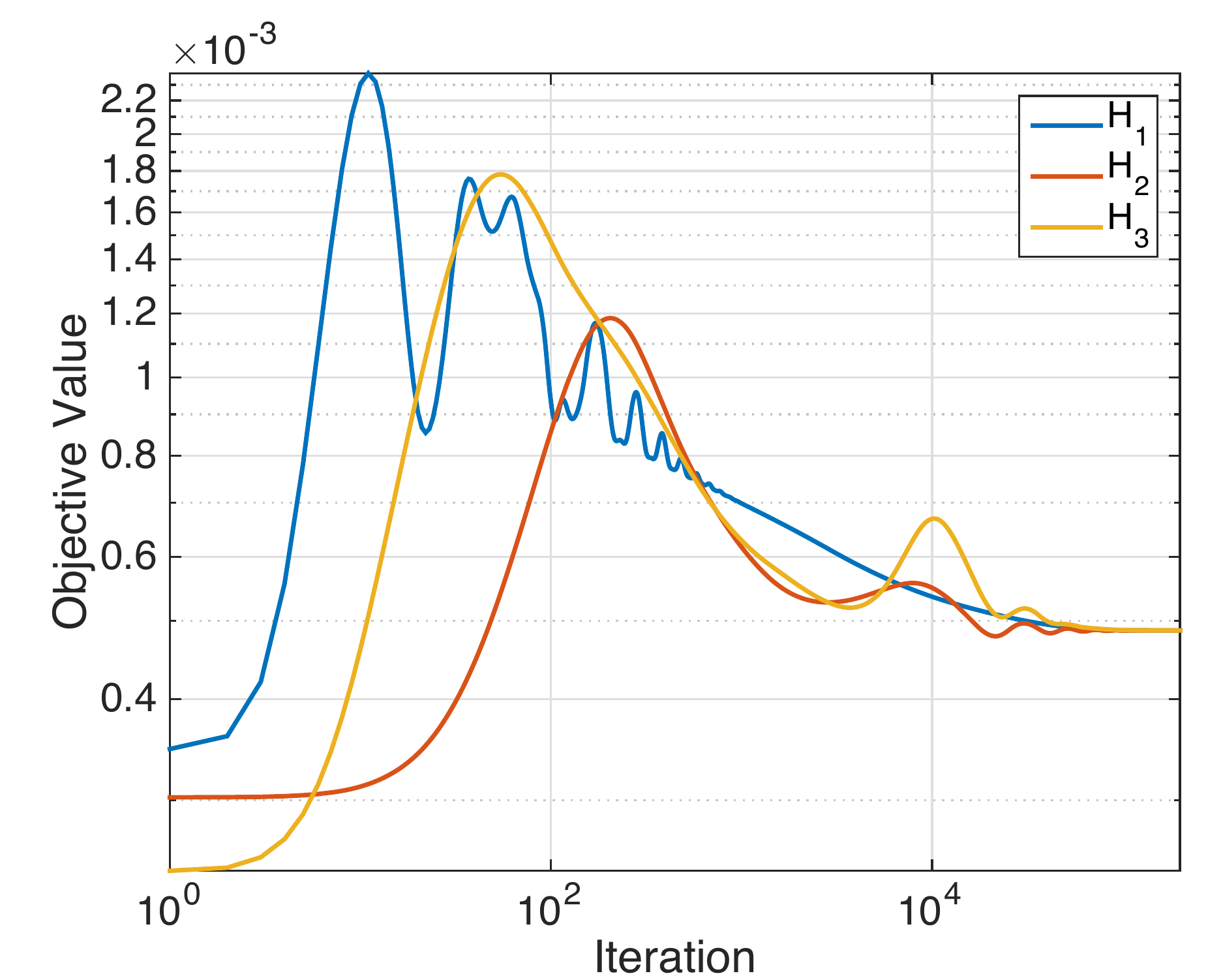}
\\
\includegraphics[width=.43\textwidth, trim = 0.0cm .0cm 0cm 0cm,clip=true]{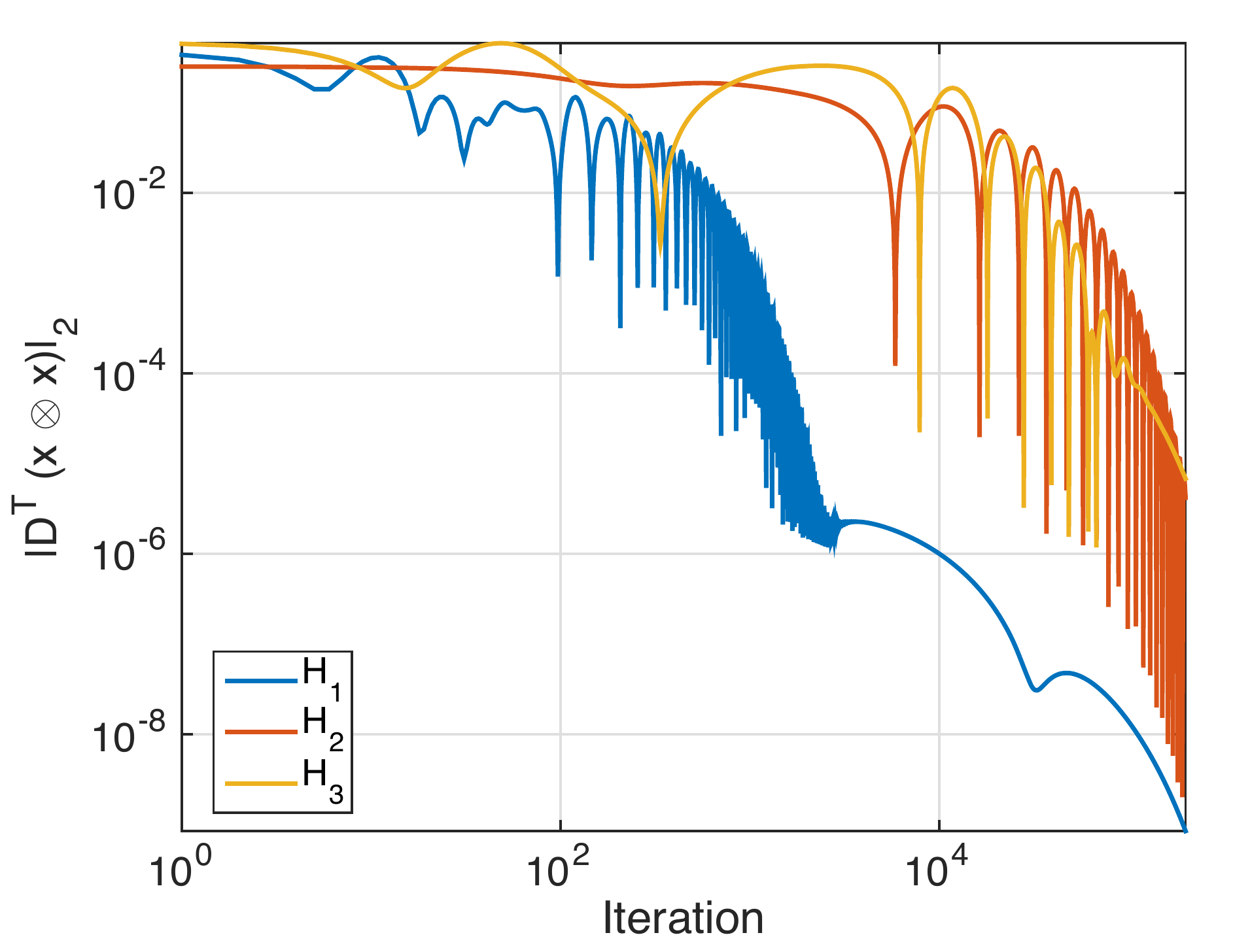}
\caption{Illustration of performances of the proposed algorithm in Example~\ref{ex_qpqc}. 
The algorithm converges slowly after 200000 iterations when running with a high step size $\gamma = 100$. 
}
\label{fig_qpqc_highgamma}
\end{figure}

The results imply that when the matrix involved in the quadratic constraint has a large condition number, 
the optimisation becomes hard, and the algorithm demands a large number of iterations.
It even can converge to a local minimum if the step size $\gamma$ is not chosen properly. 
This is clear because the QCQP conversion requires the Cholesky decomposition of an ill-conditioned matrix.
For this case, generating a new psd matrix with a lower condition number is suggested to replace the one of ill condition.
An alternative method is to run the algorithm with a relatively higher step size.
For example, when running the algorithm with $\gamma = 100$,
the algorithm converges to the (global) minimum, but it needs 200.000 iterations as illustrated in Fig.~\ref{fig_qpqc_highgamma}.

We note that the optimization problem in (\ref{equ_qp_sphere_proj}) can be solved using the interior-point algorithm. We verify this method in three optimisation problems, each corresponding to a matrix $\bH_i$. The results show that the method converges twice to a false local minimum with an objective value of  0.0017.

\begin{example}(\bf Effect of the step size $\gamma$)\label{ex_qpqc_2}
\end{example}
As shown in the previous example, a large step size $\gamma$ can be useful when the QCQP problem is hard.
In this example, we use the same matrices as in Example~\ref{ex_qpqc}, and compare convergence behaviour of the proposed algorithm when $\gamma$ is varied in the range of $[0.001, 1000 ]$.
We plot the objective function values, i.e.,  $\bx^T \bQ \bx$ to illustrate the convergence of the proposed algorithm, and the $\ell_2$-norms of the orthogonality constraints on $\bx \otimes \bx$, i.e., $\|\bD^T (\bx \otimes \bx)\|_2$ in Fig.~\ref{fig_qpqc_obj_vs_gamma}.
A relatively large step size, e.g.,  $\gamma = 10, 100, 1000$, enforces the orthogonality constraints on the parameter vector quickly, 
but it reduces the objective value slowly; hence, it slows down the overall convergence.
However, a very small step size, e.g., $\gamma  = 0.01$, may make the algorithm diverge, while the orthogonality constraints cannot be enforced on $\bx \otimes \bx$. 
Selection of an appropriate step size affects the overall convergence. Fortunately, we can choose the step sizes in quite a wide range. In this example, $\gamma = 0.01$ is possibly the best selection, but $\gamma = 1$ and 10 are also good choices, although the algorithm may need a more iterations.

\begin{figure}[t]
\centering
\includegraphics[width=.41\textwidth, trim = 0.0cm .0cm 0cm 0cm,clip=true]{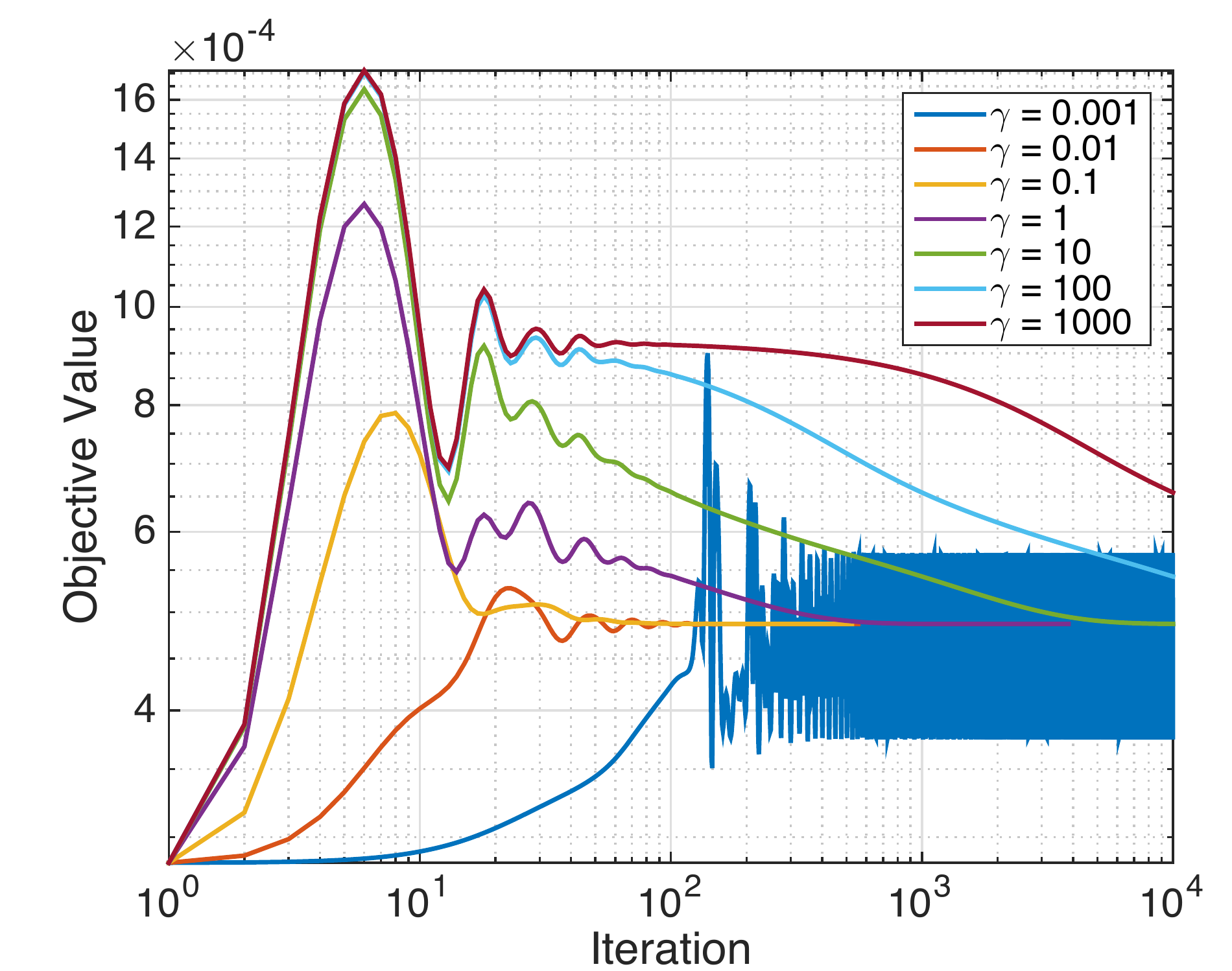}
\\
\includegraphics[width=.43\textwidth, trim = 0.0cm .0cm 0cm 0cm,clip=true]{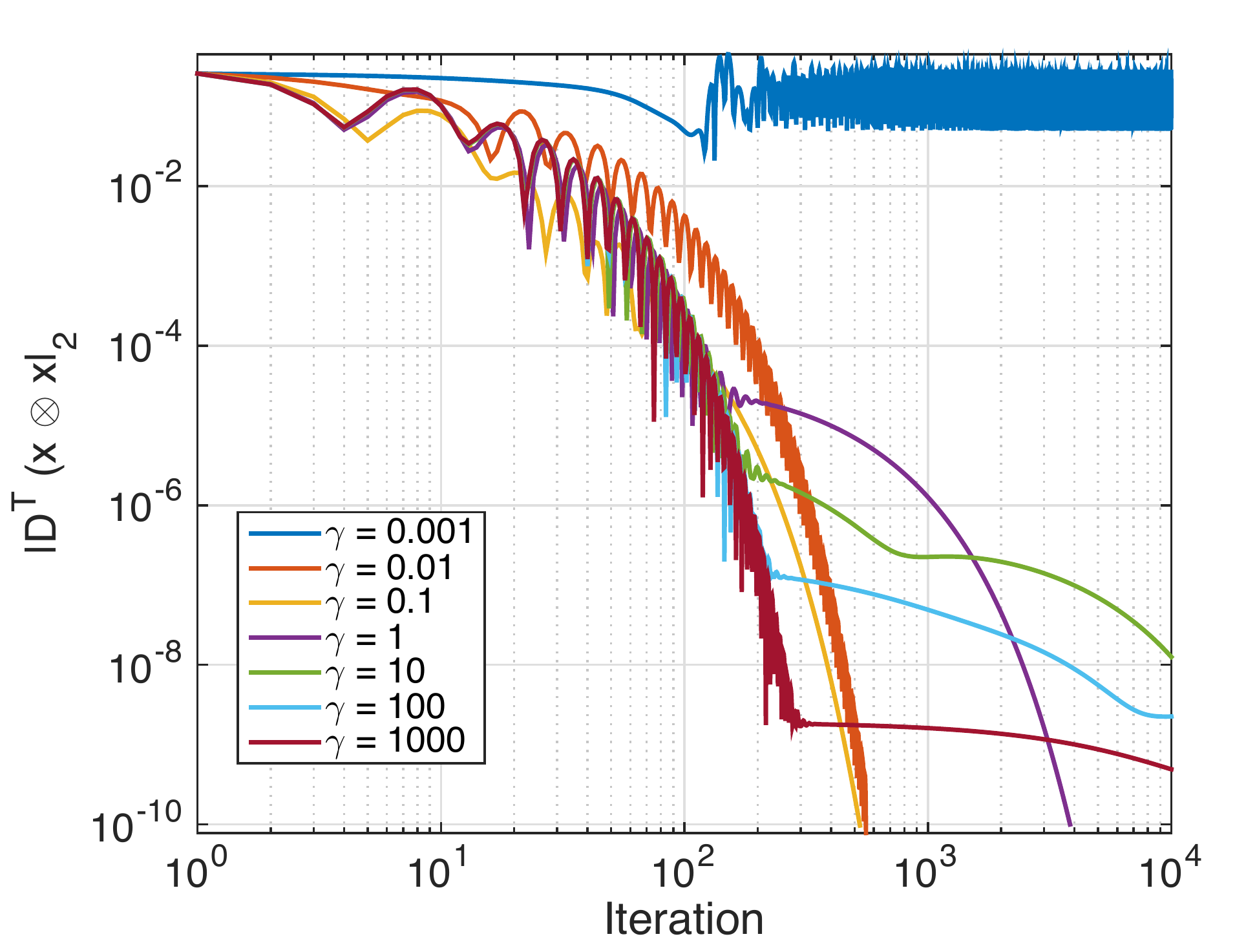}
\caption{Performance of the algorithm in Example~\ref{ex_qpqc} over a range of $\gamma$. 
A relatively large $\gamma$ enforces the orthogonality constraints quickly, 
but varies the objective value slowly; hence, it slows down the overall convergence.
A very small $\gamma$ makes the algorithm diverge. 
}
\label{fig_qpqc_obj_vs_gamma}
\end{figure}

\begin{example}(\bf Performance of the linearisation method)\label{ex_qpqc_3}
\end{example}
In this example, we verify performance of Algorithm~\ref{alg_qpqc}, but $\bx$ is updated using the linearization method in Section~\ref{sec::linearisation_x}. 
The parameters are initialised as in the previous examples.
The step-size $\mu$ is set to the step size $\gamma$, and is varied in the same range as in Example~\ref{ex_qpqc_2}.
The objective values and norm of the constraints are plotted in Fig.~\ref{fig_qpqc_obj_vs_gamma_linearsize}. Compared to the results shown in Fig.~\ref{fig_qpqc_obj_vs_gamma}, there is not much difference in the convergence of the proposed algorithm using the two update rules for $\bx$.

\begin{figure}[t]
\centering
\includegraphics[width=.41\textwidth, trim = 0.0cm .0cm 0cm 0cm,clip=true]{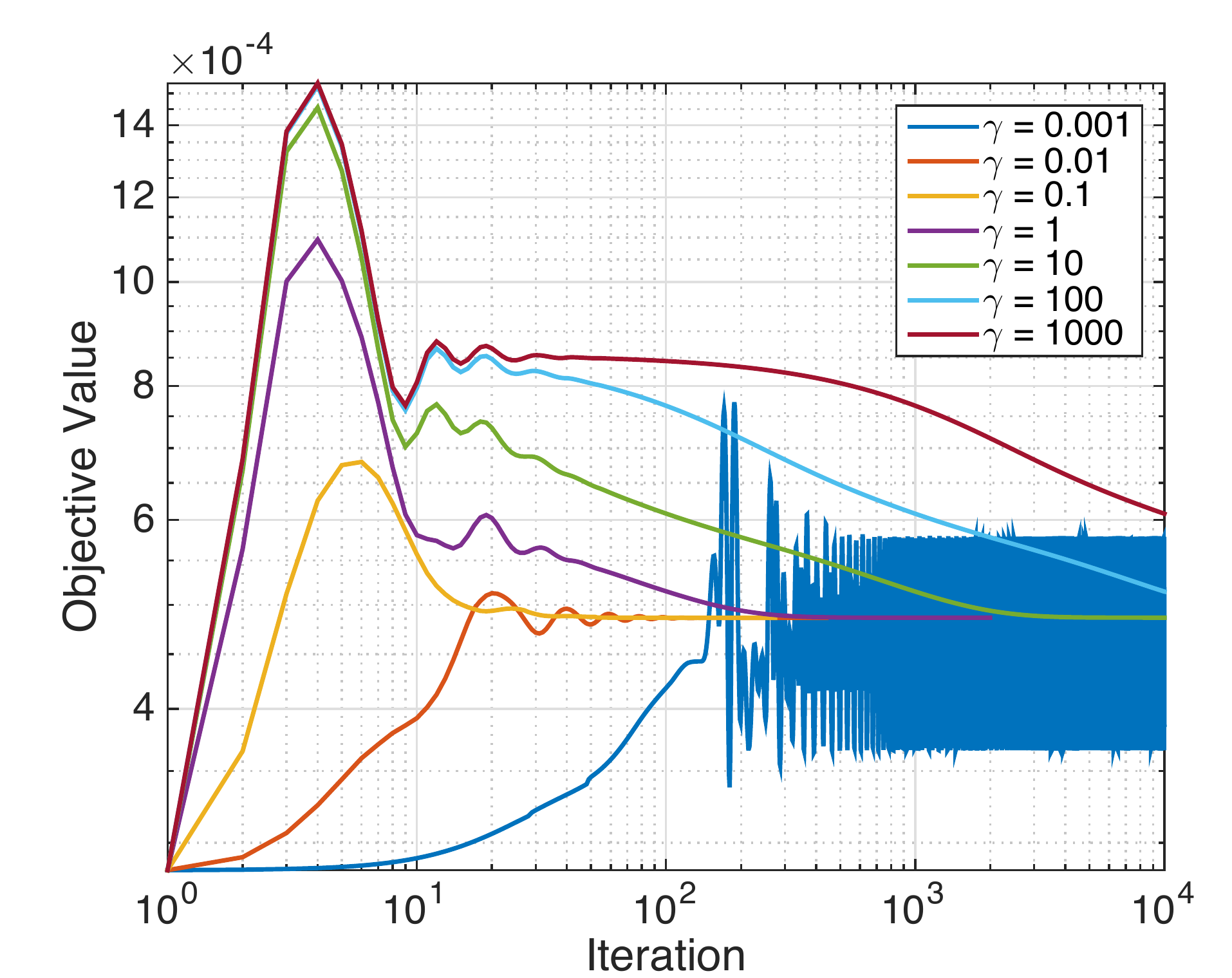}
\\
\includegraphics[width=.43\textwidth, trim = 0.0cm .0cm .8cm 0cm,clip=true]{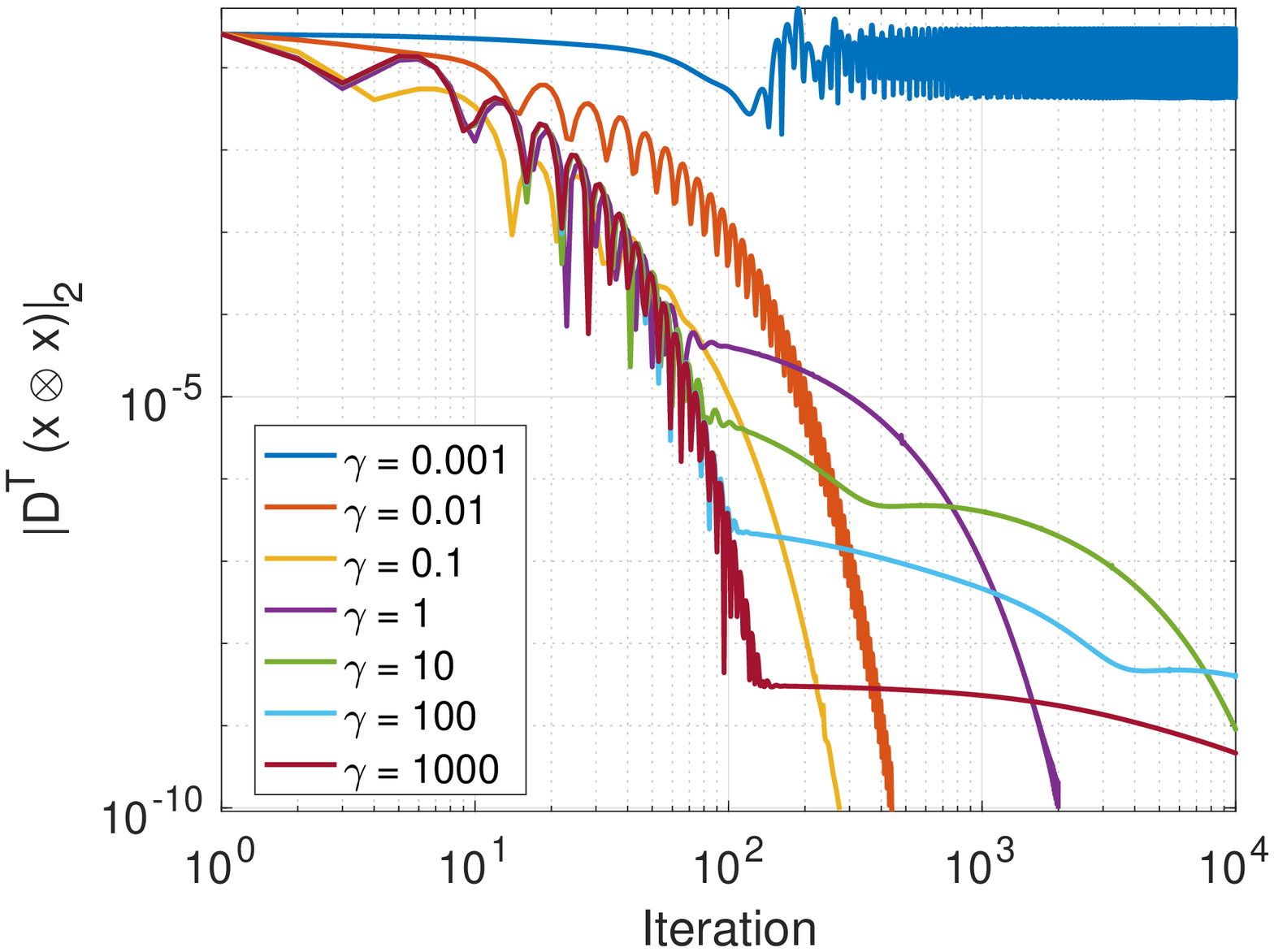}
\caption{Performance of the proposed algorithm using the linearisation update rule for $\bx$ in Example~\ref{ex_qpqc} for a range of values of $\gamma$. 
A relatively large value of $\gamma$ enforces the orthogonality constraints quickly, 
but varies the objective value slowly; hence, it slows down the overall convergence.
A very small $\gamma$ makes the algorithm diverge. 
}
\label{fig_qpqc_obj_vs_gamma_linearsize}
\end{figure}

\begin{example}\label{ex_qpqc_4mc}
\end{example}
In this example, we present results from 100 simulations with a similar settings to those in Example~\ref{ex_qpqc}.
In each run, the matrices are randomly generated. The step size $\gamma$ is set to $\alpha \, \kappa_{H}$, where 
$\alpha$ is in a range of $10^{-4}$ to $10^4$, and $\kappa_{H}$ denotes the smallest condition number of the matrices $\bH_m$.
We verify performance of the constrained QP problems in which the matrix $\bH$ is generated to have a minimum condition number. 
In addition, we compare the performance of the proposed algorithm with those using the interior-point algorithm for constrained nonlinear minimization.
In order to assess the performance, we compute the relative objective errors, i.e., a relative error between the objective value and the best (smallest) objective value among all objective values obtained by the considered methods in each run, error of the constraints, and the number of iterations.
Fig.~\ref{fig_qpqc_mc}  shows empirical cumulative distribution functions of the measures for two cases, with and without minimisation of the condition number. 
A setting of $\gamma$ is considered good, if the algorithm achieves a small relative error, e.g., less than $10^{-3}$, and a small constraint error, e.g., $\le 10^{-8}$. 

\begin{figure*}[t]
\centering
\subfigure[Empirical CDP of errors and number of iterations in solving QCQP with matrices $\bH_m$.]{
\includegraphics[width=.32\textwidth, trim = 0.0cm .0cm 0cm 0cm,clip=true]{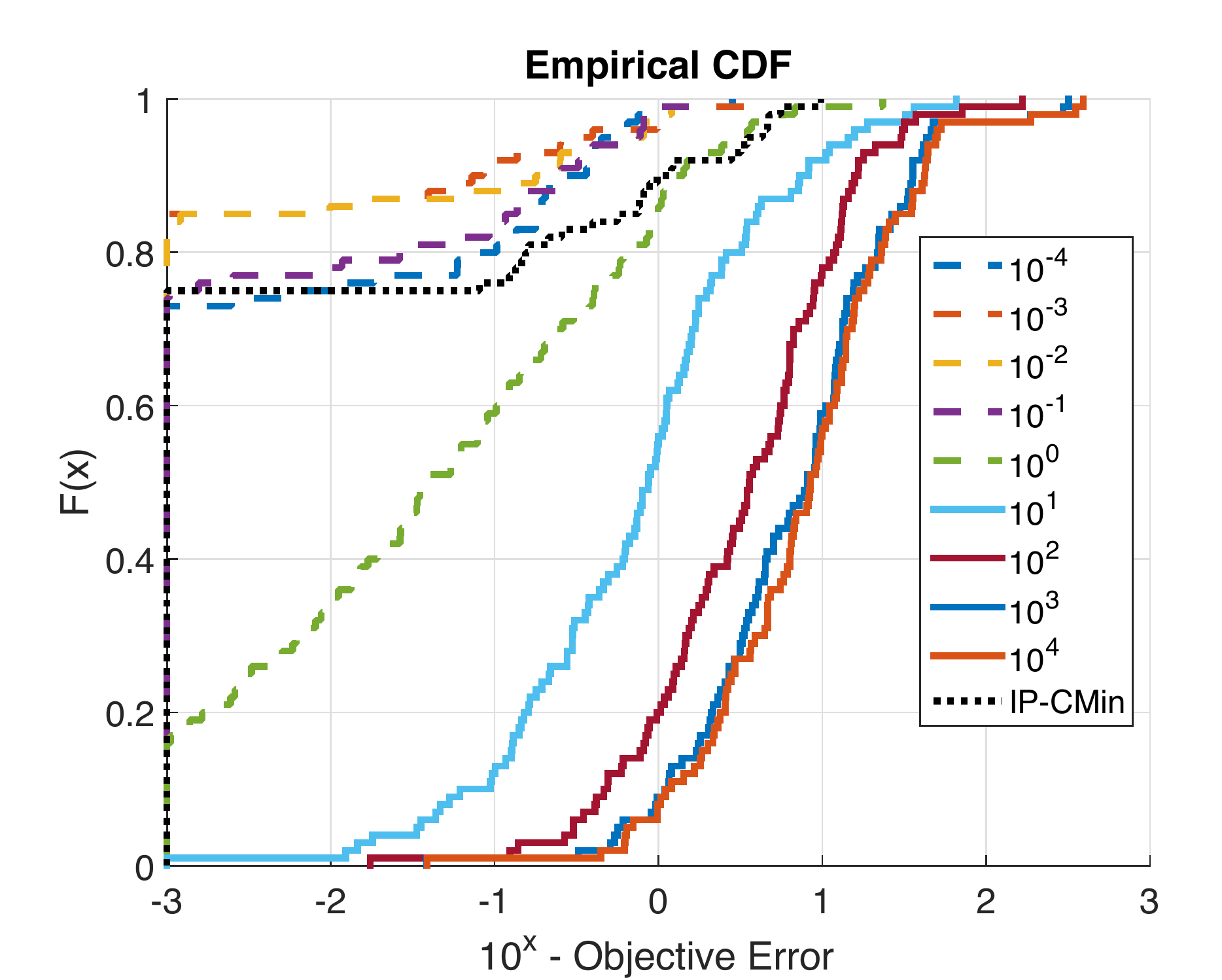}
\hfill
\includegraphics[width=.32\textwidth, trim = 0.0cm .0cm 0cm 0cm,clip=true]{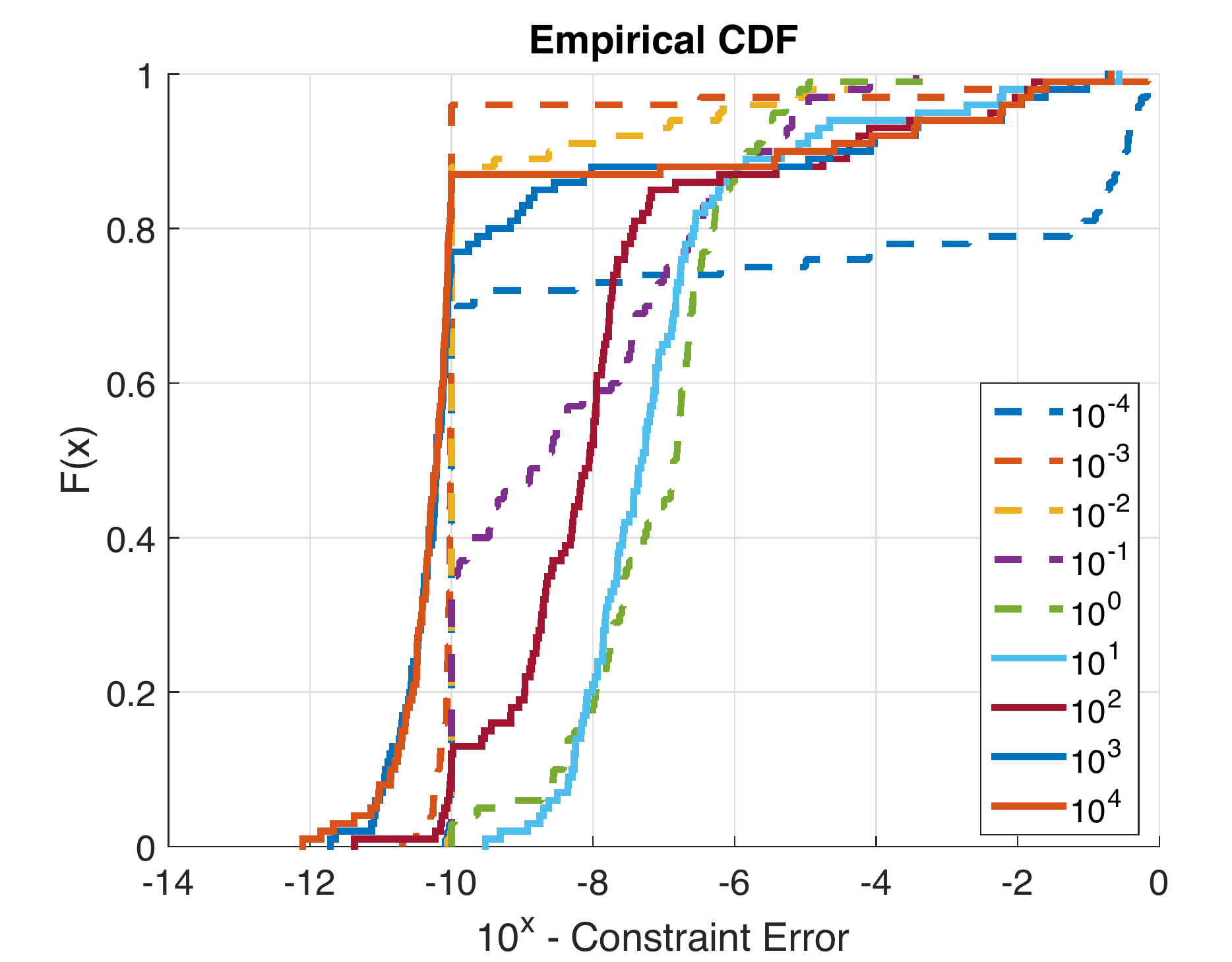}
\hfill
\includegraphics[width=.32\textwidth, trim = 0.0cm .0cm 0cm 0cm,clip=true]{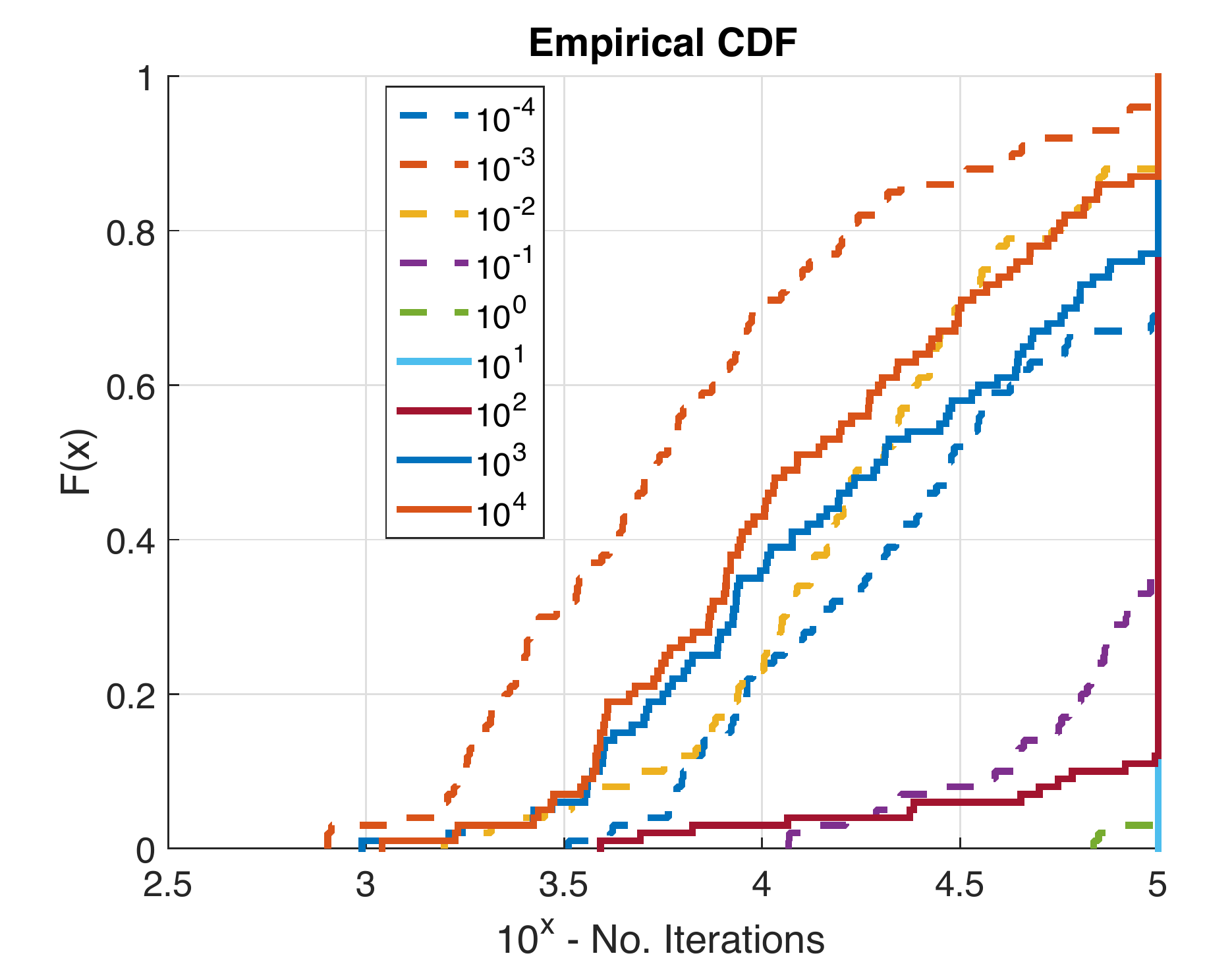}
}
\subfigure[Empirical CDP of errors and number of iterations in solving QCQP with a generated matrix $\tilde{\bH}$ having low condition number.]{
\includegraphics[width=.32\textwidth, trim = 0.0cm .0cm 0cm 0cm,clip=true]{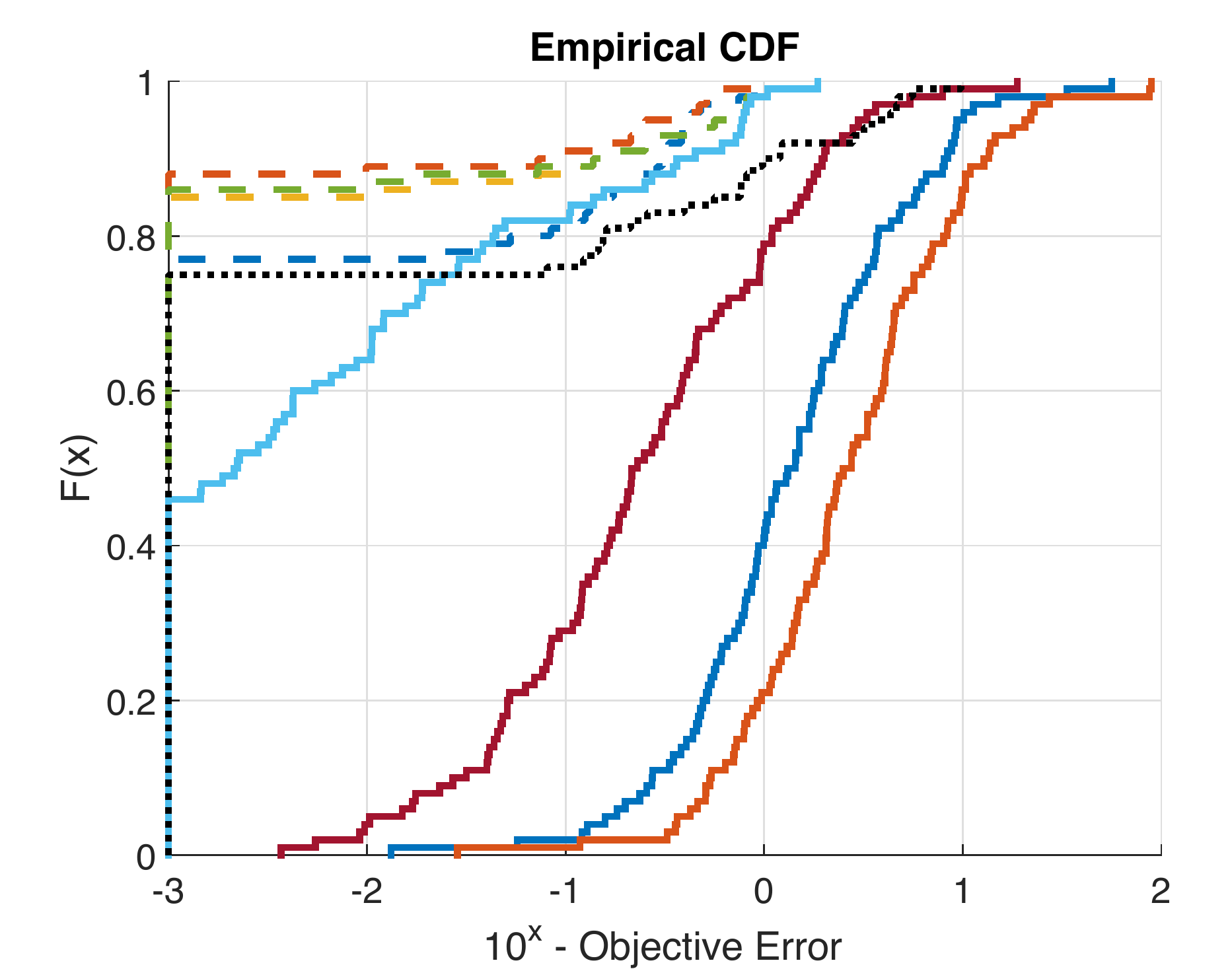}
\hfill
\includegraphics[width=.32\textwidth, trim = 0.0cm .0cm 0cm 0cm,clip=true]{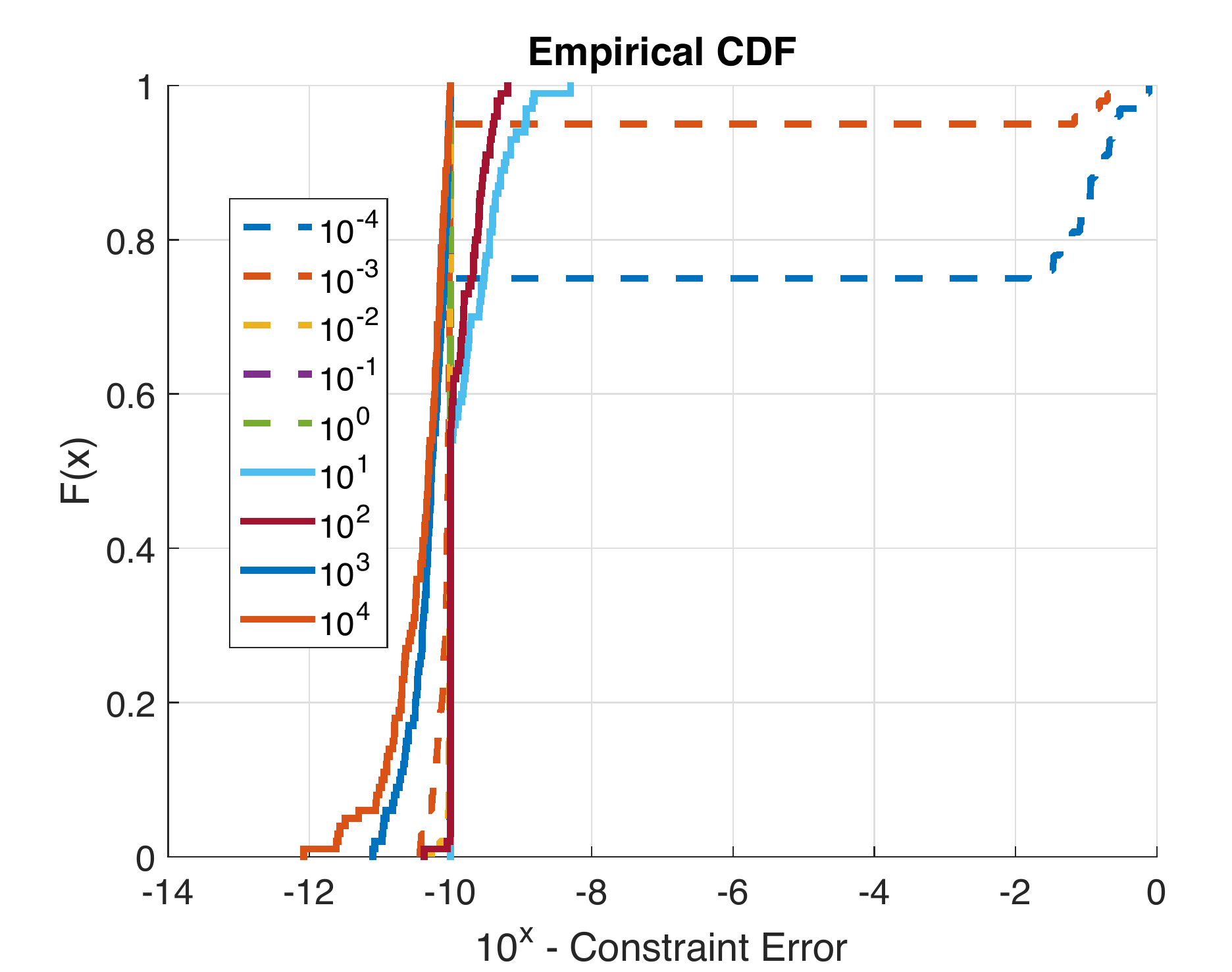}
\hfill
\includegraphics[width=.32\textwidth, trim = 0.0cm .0cm 0cm 0cm,clip=true]{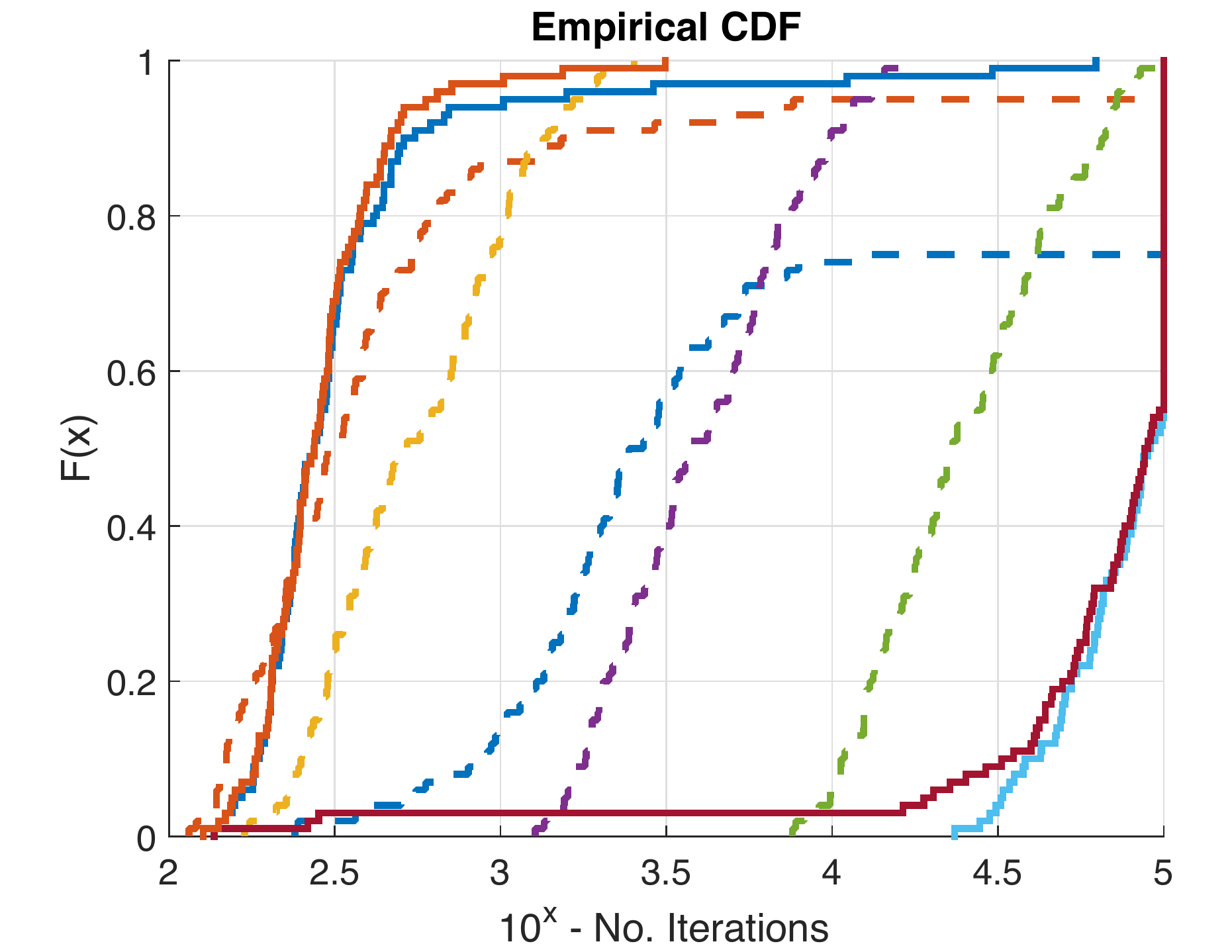}
}
\caption{Performance comparison in solving the QP with three quadratic constraints.
}
\label{fig_qpqc_mc}
\end{figure*}

As shown in Fig.~\ref{fig_qpqc_mc}, when the matrices $\tilde{\bH}$ have minimum condition numbers, the algorithm achieves good results with small relative errors, less than $10^{-4}$, with $\alpha  = \{10^{-4}, 10^{-3}, \ldots, 1\}$. Fig.~\ref{fig_qpqc_mc}(b-center) shows that in some runs the outcome vectors $\bx$ may not satisfy the constraints when $\alpha  = 10^{-4}$ and $10^{-3}$.
With the settings $\alpha = \{0.01, 0.1, 1\}$, the algorithm not only converges to the desired solution but also requires a fewer iterations, especially when $\alpha = 0.01$.
We note that when setting $\alpha$ to high values, e.g., $\ge 10$, 
the small constraint errors indicate that the outcome satisfies the constraints, but the algorithm does not converge to the global minimum within a predefined 100000 iterations, e.g., $\alpha = 10$ and $100$, or it stops because the objective function does not appear to improve significantly, e.g., for $\alpha \ge 1000$.

For the case without the correction of the condition number, although the algorithm converges with $\alpha = 10^{-3}, 10^{-2}$, it often demands a huge number of iterations, as illustrated in Fig.~\ref{fig_qpqc_mc}(a-right).

Compared to the performance of the interior point algorithm (IP), the results indicate that the IP algorithm attains a convergence ratio of 75\% to converge to the best solutions. The augmented Lagrangian algorithm with appropriate step-sizes, i.e., when $\alpha = \{0.001, 0.01, 0.1\}$, attains a convergence ratio of 89\%.

%

\section{Best Rank-1 Tensor Approximation to Symmetric Tensor of  Order-4}\label{sec:bestrank1order4}

We now present a novel application of the quadratic minimisation over a sphere to finding a best rank-1 tensor approximation of an order-4 symmetric tensor. 
The concept of the symmetric tensor is extended from the symmetric matrix, i.e., invariant under any permutation of its indices.
Symmetric tensors can be cumulant tensors, or derivative tensors  of the second Generalised Characteristic Functions\cite{Cardoso91,Yuen,DBLP:journals/corr/abs-1212-6663,DBLP:journals/tsp/AlmeidaLSC12}, or tensors representing similarity or interaction between groups of identities used for clustering \cite{ShashuaZass,Muti05}.

We consider an order-4 tensor $\tY$ which is symmetric, i.e.,
$y(i_1,i_2,i_3,i_4) = y(j_1,j_2,j_3,j_4) $, 
where $[j_1,j_2,j_3,j_4]$ is any permutation of indices $[i_1,i_2,i_3,i_4]$.
The best rank-1 tensor approximation to the tensor $\tY$ is to minimize  the following approximation error 
\be
\min_{\lambda, \bx} \qquad \|\tY - \lambda \, \bx \circ \bx  \circ \bx \circ \bx \|_F^2 \, \label{equ_cost_rank1_sym4}
\ee
where $\lambda \, \bx \circ \bx  \circ \bx \circ \bx$  represents the best rank-1 tensor to approximate $\tY$, and $\bx$ is a unit-length vector, $\bx^T \bx = 1$. For shorthand notation, we denote $\bx \circ \bx  \circ \bx \circ \bx = \bx^{(4)}$.
By expanding the Frobenious norm (\ref{equ_cost_rank1_sym4}) as
\be
\|\tY - \lambda \, \bx^{(4)} \|_F^2 = \|\tY\|_F^2 + \lambda^2 - 2 \, \lambda \, \langle \tY,     \bx^{(4)} \rangle		\notag 
\ee
it is straightforward to see that the optimal weight $\lambda^{\star}$ is the inner product between the tensor  $\tY$ and the rank-1 tensor  $  \bx^{(4)}  = \bx \circ \bx  \circ \bx \circ \bx$, that is, 
$	\lambda^{\star} = \langle \tY,     \bx^{(4)} \rangle	$.
Hence, the objective function is rewritten as
\be
\min_{\bx}  \quad \|\tY\|_F^2 - 2 \left(\tY   \bullet  \bx^{(4)} \right)^2 \notag .
\ee
For a positive $\lambda$, we maximise the inner product to give 
\be	
&\max \qquad 	& \langle \tY,     \bx^{(4)} \rangle	\, \label{equ_cost_rank1_sym4_2a} \quad
\text{s.t.} \quad    \bx^T \bx = 1 \notag 
\ee
and minimise the inner product for a negative $\lambda$, that is
\be	
&\min \qquad 	&\langle \tY,     \bx^{(4)} \rangle	 \, \label{equ_cost_rank1_sym4_2b} \quad
\text{s.t.} \quad   \bx^T \bx = 1 \,. \notag 
\ee
The final solution $\lambda$ is that with the largest absolute value. 
Both problems can be solved on a Riemannian or Stiefel manifold using e.g., the Trust-Region solver  \cite{manopt}.
Here, we propose another method to solve the two above problems.

Let $\bz = \bx \otimes \bx$, then the minimisation problem in (\ref{equ_cost_rank1_sym4_2b}) becomes 
\be
& \min \quad  & \bz^T \bQ \bz \quad \text{s.t.} \quad    \bz = \bx \otimes \bx \, ,  \; \;\text{and}\quad \bz^T \bz = 1
\notag 
\ee  
where $\bQ$ is a mode-(1,2) matricization of $\tY$.

The augmented Lagrangian function of the above problem now becomes
\be
\calL(\bx , \by , \bz) = f(\bz)  + \by^T (\bz - \bx \otimes \bx)  + \frac{\gamma}{2} \, \|\bz - \bx \otimes \bx \|^2
\ee
where $f(\bz)$ is the objective function of the minimization of $\bz^T \bQ \bz$ subject to $\bz^T \, \bz = 1$.
Variables $\bx$, $\bz$ and $\by$ are sequentially updated following the sequence  
\be
\bz &=& \argmin  \quad  f(\bz)  +  \by^T (\bz - \bx \otimes \bx ) + \frac{\gamma}{2} \|\bz - \bx \otimes \bx \|^2  		\, \notag \\
&=& \argmin  \quad  \frac{1}{2} \bz^T \bQ \bz  + ( \by - \gamma \bx \otimes \bx)^T \bz  \label{prob_solve_z4} \\ 
&& \text{subject to } \quad \bz^T \bz = 1 \notag  \\
\bx &=& \argmin   \quad  \by^T (\bz - \bx \otimes \bx) + \frac{\gamma}{2} \| \bz - \bx \otimes \bx  \|^2  		\, \notag \\
&=& \argmin   \quad    \| \bz + \frac{\by}{\gamma} - \bx \otimes \bx  \|^2  		\label{prob_solve_x4}\\
\by &\leftarrow& \by + \gamma (\bz - \bx \otimes \bx ) \, .
\ee
The unit-length vector $\bz$ is a minimiser to a SCQP, whereas $\bx$ is the eigenvector associated with the largest eigenvalue of the symmetric matrix $\bT_s = \frac{1}{2}(\bT  + \bT^T)$, where 
\be
	\bT =  \bZ + \frac{1}{\gamma} \bY\, ,		\notag 
\ee
$\bz = \vtr{\bZ}$ and $\by = \vtr{\bY}$.

\begin{example}(\bf Best rank-1 tensor approximation to a symmetric tensor of order-4)\label{ex_bestrank_sym4}
\end{example}
This example compares performance of our algorithm for best rank-1 tensor approximation for symmetric tensor of order-4, and the Riemannian trust-region solver in the Manopt toolbox \cite{manopt}.
We generate 1000 random tensors of size $I \times I \times I \times I$, where $I = 10$ or $20$, then matricize them so that they will become symmetric tensors of order-4. The tensors are normalized to have unit Frobenius norm.
For each run, the best approximation error $\varepsilon^{\star}$ is defined as the smallest  error among approximation errors of the two methods: Augmented Lagrangian method and the Riemannian trust-region
\be
	\varepsilon = \|\tY - \lambda \, \bx^{(4)}\|_F^2  = 1 - \lambda^2 \,.	\notag
\ee
Relative errors to the best approximation error $\frac{\varepsilon - \varepsilon^{\star} }{\varepsilon^{\star}}$ is then assessed to measure performance of the approximation.

\begin{figure}[t]
\centering
\subfigure[$I = 10$]{
\includegraphics[width=.45\textwidth, trim = 0.0cm 0cm 0cm .7cm,clip=true]{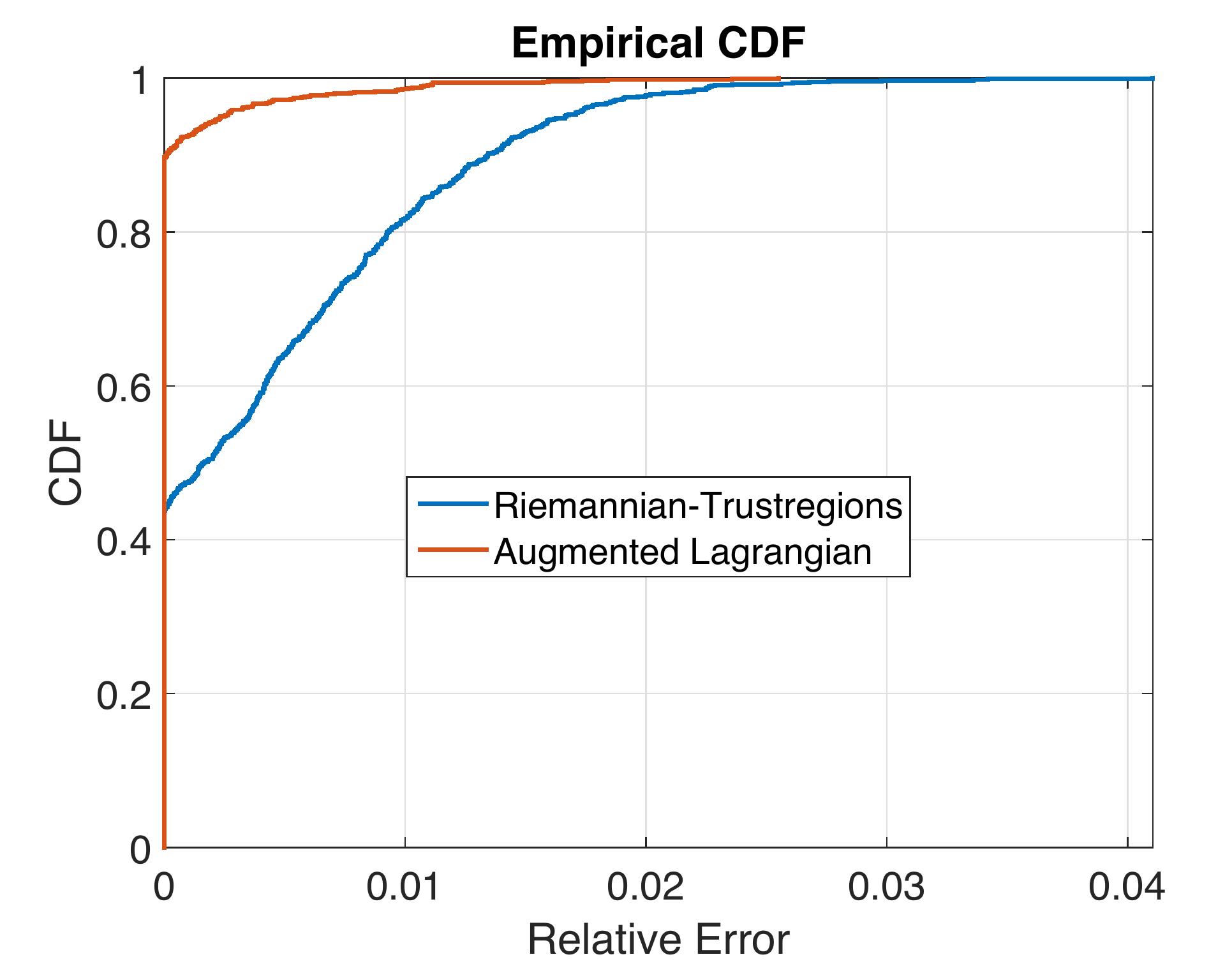}}
\hfill
\subfigure[$I = 20$]{
\includegraphics[width=.45\textwidth, trim = 0.0cm 0cm 0cm 0.7cm,clip=true]{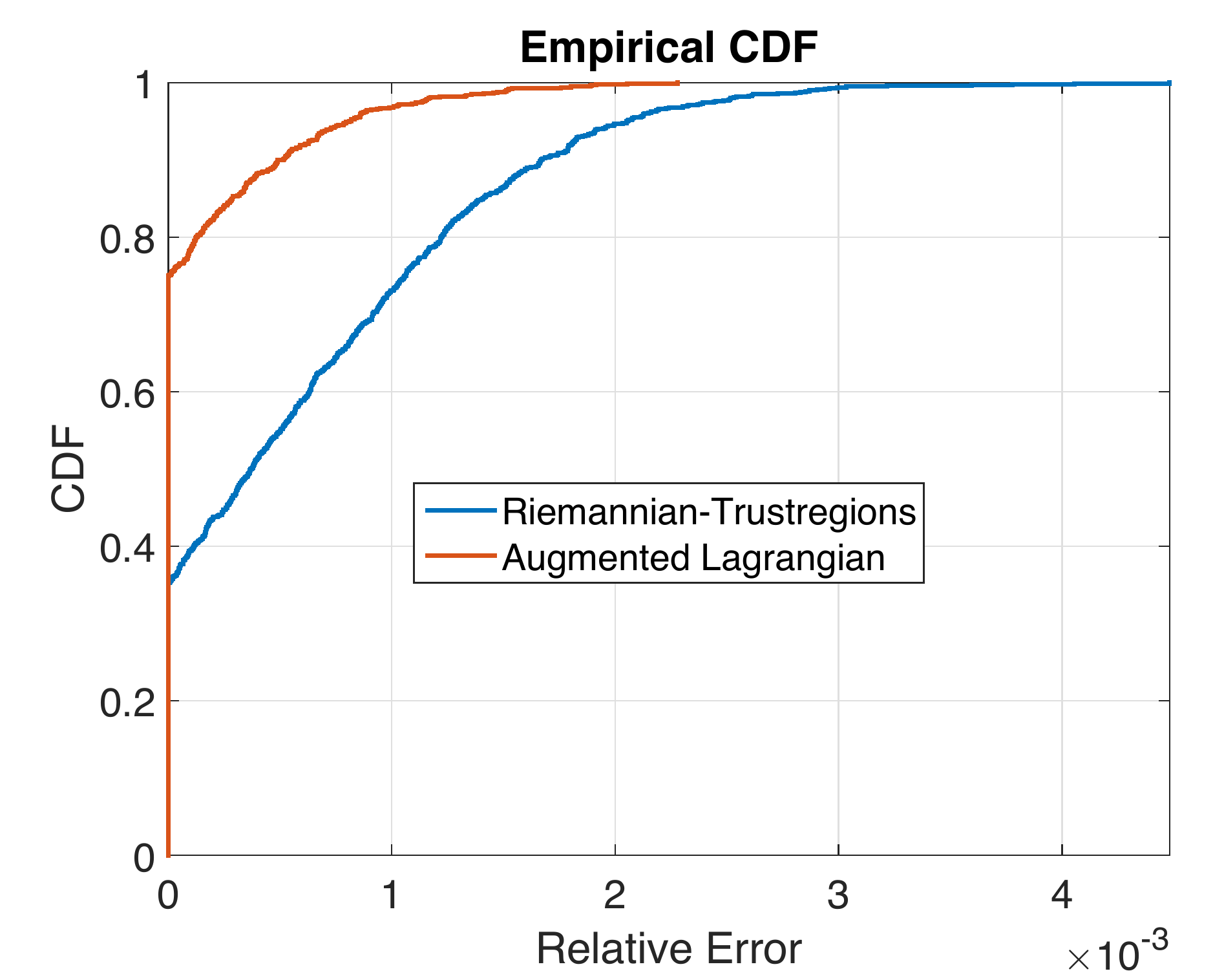}}
\caption{The empirical cumulative distribution functions of the relative errors of two algorithms based on augmented Lagrangian and Riemannian trust-region methods.}
\label{fig_best_rank1}
\end{figure}

Fig.~\ref{fig_best_rank1} shows the empirical cumulative distribution functions of 1000 relative errors. The results indicate that our algorithm based on the quadratic optimisation over sphere achieves a higher success rate. For example, for the case when $I = 20$, our algorithm attains an error less than 0.001 with a rate of 96.8\%, whereas the trust-region solver achieves a rate of 73.1\% for the same error range.
When $I = 10$, the Augmented Lagrangian algorithm has a success
rate of 92.5\% for a similar accuracy of $0.001$, while the trust-region algorithm has a quite low rate of $47.4\%$.

\section{Generalized Eigenvalue Decomposition with Eigen matrix of low rank structure}\label{sec:TT_geig}

We now address a constrained generalised eigenvalue decomposition which exploits the QCQP to derive an algorithm. The considered problem is stated below.

\begin{definition}[GEVD with eigen matrix having a low rank structure]\label{prob_gevd_lowrank_matrix}
Consider a positive semi-definite matrix $\bQ$ of size $IJ \times IJ$  and 
a positive definite matrix $\bS$ of size $IJ \times IJ$.
We solve the following optimisation problem 
\be
&\min  \quad  & \tr(\, \bX^T \, \bQ \, \bX)  \label{equ_gevd_lr} \quad
\mathrm{{s.t.}}  \quad   \bX^T  \, \bB \, \bX = \bI_{R} \notag 
\ee
to find a matrix $\bX$ of size $IJ \times R$, where each column of $\bX$  is a vectorisation of a product of two matrices 
\be
	\bx_r = \vtr{\bG_{r} \, \bA^T} \label{equ_xr_GrA}
\ee
and $\bG_r$ are matrices of size $I \times S$ and $\bA$ is of size $J \times S$.
\end{definition}

If we concatenate the matrices $\bG_r$ into an order-3 tensor of size $I \times S \times R$, the factor matrix $\bX$ is a mode-(1,2) matricization of an order-3 tensor $\tX$ of size $I \times J \times R$, defined as 
\be
	\tX = \tG  \, \times_2 \bA\, .	\notag
\ee
Because of scaling and rotation ambiguities, the matrix $\bA$ can always be normalized to be an orthogonal matrix. However, we do not exploit the orthogonality constraint on $\bA$ in its estimation, but perform orthogonal normalisation after each update. 
With the above interpretation, the matrix of eigenvectors, $\bX$, is considered a block matrix in the tensor train format of only two cores. 
For the GEVD in which the matrix $\bX$ is a block TT-matrix composed from more cores, the problem in (\ref{prob_gevd_lowrank_matrix}) becomes a local problem in an alternating algorithm to estimate the cores.
The problem in (\ref{equ_gevd_lr}) is a constrained GEVD.

For this simple case, we can express the factor matrix as $\bX = [\tX]_{(1,2)} = [\tX]_{(3)}^T$, and
\be
	\bX = (\bA \otimes \bI_I) \bG \label{equ_X_vG}
\ee
where $\bG = [\vtr{\bG_1},\ldots, \vtr{\bG_R}]$ and 
\be
	\bX = \left[ \bI_J \otimes \bG_{1}, \ldots, \bI_J \otimes \bG_{R}\right] \left(\bI_R \otimes \vtr{\bA^T}\right)   \label{equ_X_vA}\,.	\notag
\ee

We now show that $\tG$ can be estimated using a GEVD, and $\bA$ is a solution to a quadratic programming problem with quadratic constraints.
Our proposed algorithm alternates the estimation of $\tG$ and $\bA$.

\subsection{Update of $\bG_r$}\label{sec:update_Gr}

By exploiting the expression in (\ref{equ_X_vG}), while fixing the matrix $\bA$, we can find $\bG$ in a GEVD, given by
\be
	&\min  \quad  & \tr(\, \bG^T \, \bQ_G \, \bG)  \notag \label{equ_gevd_G} \quad 
	\text{s.t.}  \quad \bG^T  \, \bB_G \, \bG = \bI_{R} \notag 
\ee
where
\be
	\bQ_G &=& (\bA^T \otimes \bI_I) \,  \bQ \,  (\bA \otimes \bI_I)\,,\notag  \\
	\bB_G &=& (\bA^T \otimes \bI_I)  \, \bB  \, (\bA \otimes \bI_I).\notag 
\ee

\subsection{Update of $\bA$}\label{sec:update_A}

In order to derive the update rule for $\bA$, from (\ref{equ_X_vA}), we can rewritte the objective function as 
\be
	\tr(\bX^T \bQ \bX) &=&  \sum_{r = 1}^{R} \, \bx_r^T \bQ \, \bx_r \notag \\	
		&=& \sum_{r=1}^{R}  \vtr{\bA^T}^T \, ( \bI_J \otimes \bG_{r}) ^T \bQ \, ( \bI_J \otimes \bG_{r})  \vtr{\bA^T}  \notag \\
		&=&  \vtr{\bA^T}^T \, \bQ_A \,  \vtr{\bA^T} \label{eq_obj_gevd_A}
\ee
where 
\be
	\bQ_A &=& \sum_{r = 1}^{R} \, ( \bI_J \otimes \bG_{r}) ^T \bQ \, ( \bI_J \otimes \bG_{r})  \, .	\notag 
\ee
Similarly, the quadratic constraint  is rewritten for each pair of columns $\bx_r$ and $\bx_s$ as
\be
	\delta_{r,s}  &=& \bx_r ^T \bB \bx_s  \notag \\
		  &=&  \vtr{\bA^T}^T \, ( \bI_J \otimes \bG_{r}) ^T \bB \, ( \bI_J \otimes \bG_{s})  \vtr{\bA^T}   \notag \\
		  &=&  \vtr{\bA^T}^T \,   \bB_{r,s} \,  \vtr{\bA^T}   \,. \label{equ_constraint_A}
\ee
for $r, s  = 1, \ldots, R$, where 
\be
	\bB_{r,s} &=& ( \bI_J \otimes \bG_{r}) ^T \bB \, ( \bI_J \otimes \bG_{s})\, \label{eq_Brs}\, .
\ee
As a result of (\ref{eq_obj_gevd_A}) and (\ref{equ_constraint_A}), the vector $\vtr{\bA^T}$ is a minimiser of a quadratic programming problem with $R(R+1)/2$ quadratic constraints  
\be
	&\min \quad &\vtr{\bA^T}^T \bQ_A \, \vtr{\bA^T} 	\notag \\
	&\text{subject to} \quad &    \vtr{\bA^T}^T \, \bB_{r,r}\, \vtr{\bA^T} = 1,\quad  r = 1, \ldots, R\notag \\
	& &  \vtr{\bA^T}^T \, \bB_{r,s} \,  \vtr{\bA^T} = 0,\quad  1\le  s < r \le R\, . \notag 
\ee
Following the method in Section~\ref{sec:gen_psdmatrix}, we can generate a matrix $\bB$ with a low condition number from the matrices $\bB_{r,s}$, or choose a matrix with the smallest condition number among them, e.g., $\bB_{1,1}$. Denote by $\bF$ the factor matrix in the Cholesky decomposition of the matrix $\bB_{1,1} = \bF \bF^T$,
and introduce the following symmetric matrices of size $JS \times JS$
\be
	\bD_{r,r}  &=&  \bF^{-1} ( \bB_{1,1} - \bB_{r,r}) \bF^{-1 \, T} 	,  \qquad  r = 2, \ldots, R  \notag\\
	\bD_{r,s}  &=& \bF^{-1} (\bB_{r,s} + \bB_{r,s}^T ) \bF^{-1\, T}	,  \qquad  1\le  s < r \le R \notag
\ee
and 
\be
\tilde{\bQ}_A = \bF^{-1} \,  {\bQ}_A \, \bF^{-1 \, T}. \notag
\ee
We can then find $\ba = \bF^T \vtr{\bA^T}$ in the following constrained optimization based on Algorithm~\ref{alg_qpqc}
\be
	&\min \quad &\ba^T  \, \tilde{\bQ}_A \, \ba 	 \notag \label{eq_qpqc_a}\\
	&\text{subject to} \quad &    \ba^T  \ba = 1,   \notag\\
	& &  \ba^T \, \bD_{r,r} \,  \ba = 0,\quad  r = 2,\ldots, R, \notag\\
	& &  \ba^T \, \bD_{r,s} \,  \ba = 0,\quad   1 \le  s  <  r \le R\, .  \notag
\ee

\begin{example}{(\bf Discriminant analysis of hand-written digits)}\label{ex_digit_classify}
\end{example}
In this example, we illustrate an application of the proposed algorithm for solving the constrained GEVD in (\ref{equ_gevd_lr}).
More specifically, we perform a discriminant analysis on the training samples, which comprise handwritten images for digits 0, 1 and 2.
The data is taken from the MNIST dataset. Images are of size $28 \times 28$, and their Gabor features are computed for 8 orientations and 4 scales. 
The Gabor images are scaled down to size $16 \times 16$, then vectorized and concatenated into a matrix of size $256 \times 32$.
All digit images construct an order-3 tensor of size $256 \times 32 \times 900$, 300 images for each digit.
Ten random $10$-fold cross-validations are performed on 900 samples: 810 for training and 90 for the test set.

Denote the matrix of training samples by $\bY_{tr}$, which is of size $8192 \times 810$. We seek a projection matrix $\bX$ of size $8192 \times 2$ to extract 2 feature vectors $\bF = \bY_{tr}^T \,  \bX$
which maximises the Fisher score, a ratio of  the between and with-in distances
\be
\max \quad \frac{ \tr(\, \bF^T   \, \bS_b \, \bF)}{\tr(\bF^T  \, \bS_w \, \bF)}	\notag
\ee
where $\bS_b$ and $\bS_w$ are the between and with-in scattering matrices constructed for the training samples. 
Alternating to the maximisation of the trace-ratio, we solve the GEVD problem 
\be
&\max  \quad  & \tr(\, \bX^T \, \bY_{tr}  \, \bS_b \, \bY_{tr}^T \, \bX)  \, , \notag \label{equ_gevd_da} \\
&\text{subject to}  \quad & \bX^T  \, \bY_{tr}  \, \bS_w \, \bY_{tr}^T \, \bX = \bI_{2} \notag .
\ee

Because each digit is represented by a vector of length 8192 ($=256 \times 32$), which exceeds the number training samples of 810, the above ordinary linear discriminant analysis often leads to over-fitting, and it is not applicable. 
To this end, we apply the constrained GEVD in (\ref{equ_gevd_lr}).   
Columns of the matrix $\bX$ are constrained with a structure
\be
	\bx_r = \vtr{\bG_{r} \, \bA^T} \notag \label{equ_xr_GrA_da} \,.
\ee
In our example, $\bG_1$ and $\bG_2$ are of size $256 \times 2$ and $\bA$ is of size $32 \times 2$.
The extracted features for training samples are computed and used to train a simple LDA classifier. 
Using the proposed algorithm, we obtain a classification accuracy of 97.36\% averaged over $10 \times 10$-fold cross-validations.

Fig~\ref{fig_da_twoblocks} shows the scatter plot of samples plotted using the two feature vectors, demonstrating that digits 0, 1 and 2 are  distinguished.

\begin{figure}
\centering
\includegraphics[width=.48\textwidth, trim = 0.0cm .5cm 0cm 0cm,clip=true]{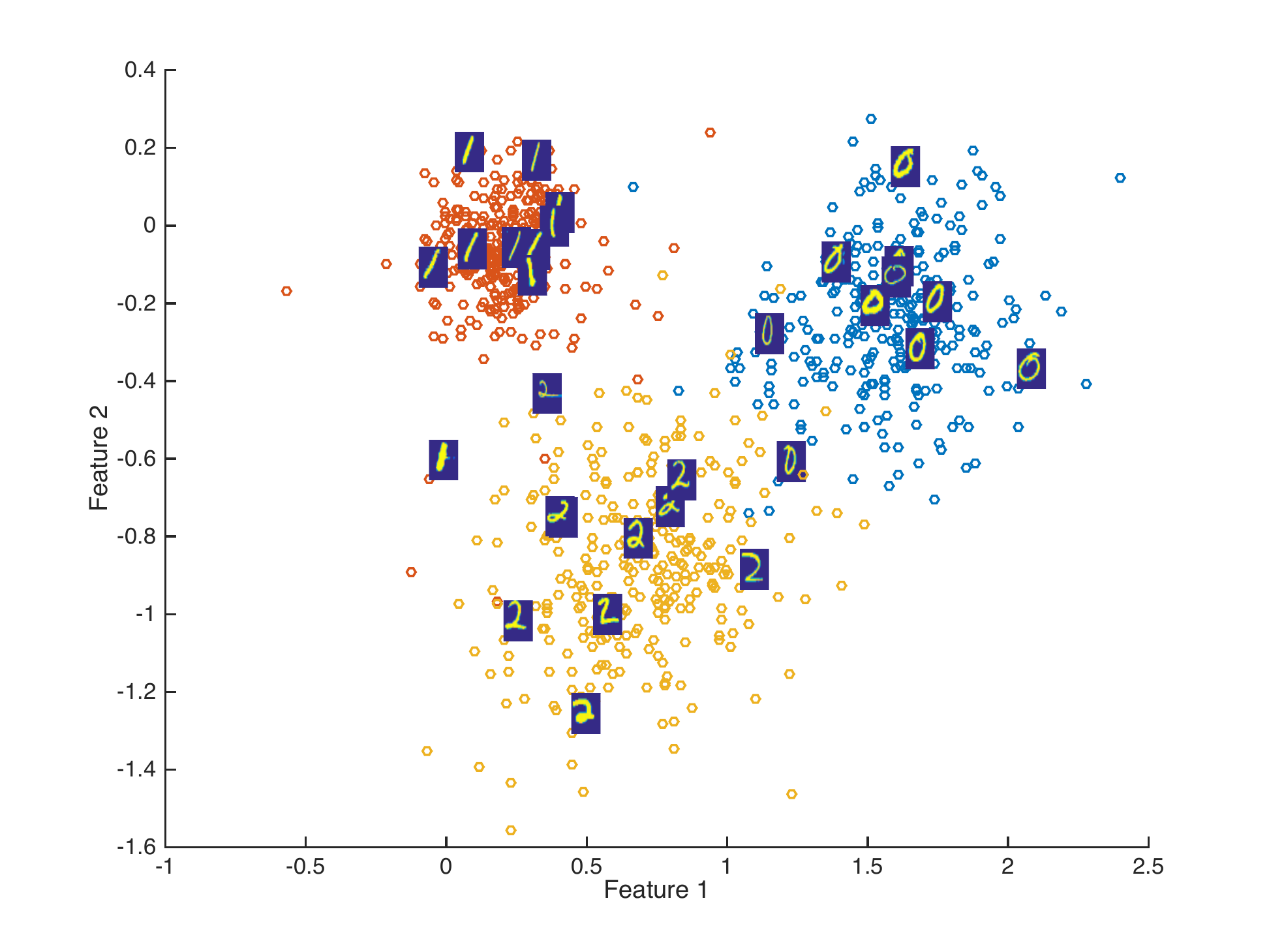}
\caption{Scattering plot of hand-written digits using two extracted features in Example~\ref{ex_digit_classify}.}
\label{fig_da_twoblocks}
\end{figure}

%
%
%
%

%

\section{Conclusions}\label{sec:conclusion}

We have introduced a robust solution to the SCQP problem by imposing 
an error bound on the root of the underlying secular equation. 
The method has been initially derived as SCQP for matrix variate data, together with  the related linear regression with an error bound constraint.
In addition, we proposed an algorithm for the QCQP problem which treats QCQP as  SCQP and an orthogonal projection. In the process, the quadratic term within the quadratic constraint is correctted by a term with a minimum condition number. This correction method has been shown to improve the convergence of the proposed algorithm. Applications of the SCQP and QCQP have been presented for image deconvolution, tensor decomposition and constrained GEVD.

\appendix
\section{Proof of Lemma~\ref{lem_sqp_simplify}}
\label{sec:proof_lem_sqp_simplify}
\begin{proof}
We consider a simple case when some eigenvalues are identical, e.g., $s_1 = s_2 = \cdots = s_L < s_{L+1} < \cdots < s_K$.
If $\bc_{1:L}$ are all zeros, the objective function is independent of $\tilde{\bx}_{1:L} = [\tilde{x}_1, \tilde{x}_2 \ldots, \tilde{x}_L]$. Hence, $\tilde{\bx}_{1:L}$ can be any point on the ball $\|\tilde{\bx}_{1:L}\|^2 = d^2 = 1 - \sum_{k = L+1}^{K} \tilde{x}_{k}^2$.
Otherwise, $\tilde{\bx}_{1:L}$ is a minimiser to a constrained linear programming while fixing the other parameters $\tilde{x}_{L+1}, \ldots, \tilde{x}_K$, in the form
\be
	\min \quad  \bc_{1:L}^T \, \tilde{\bx}_{1:L} \quad \text{subject to} \quad   \|\tilde{\bx}_{1:L}\| = d  		\notag
\ee
which yields 
\be	
\tilde{\bx}_{1:L}= \frac{-d}{\|\bc_{1:L}\|}  \bc_{1:L}\,.	\notag
\ee
For both cases, we can define  
\be
\bz &=& [-d, \tilde{x}_{L+1}, \ldots, \tilde{x}_{K}] ,\notag \\
\tilde{\bc} &=& [\|\bc_{1:L}\|, c_{L+1}, \ldots, c_{K}], \notag \\
\tilde{\bs} &=& [s_1, s_{L+1}, \ldots, s_K], \notag 
\ee
and perform a reparameterization to estimate $\bz$ from a similar SCQP but with distinct eigenvalues $\tilde{\bs}$, as
\be
&\min  \quad  & \frac{1}{2} \, \bz^T \, \diag(\tilde{\bs} )  \, {\bz} + \tilde\bc^T {\bz}\quad 
\text{subject to}  \quad   {\bz}^T {\bz} = 1 \notag \,.
\ee
\end{proof}

\section{Proof of Lemma~\ref{lem_root_lessthan1}}\label{sec:proof_lem_root_lessthan1}
\begin{proof}
It follows from the second derivative of $f(\lambda)$, given by
\be
f^{''}(\lambda) = - 2 \, \sum_{k} \frac{c_k^2}{(s_k - \lambda)^3}		\notag
\ee
that $f^{''}(\lambda)  < 0$ for all $\lambda < s_1 = 1$. That is, 
$f^{\prime}(\lambda)$ monotonically decreases with $\lambda <  s_1 = 1$.

In addition, since $s_k \ge 1$, for all $k$, we have 
\be
f^{\prime}(0) = 1 - \sum_{k = 1}^{K} \frac{c_k^2}{s_k^2}  \ge 1 - \sum_{k = 1}^{K} c_k^2  = 0		\notag 
\ee
and 
\be
f^{\prime}(1-|c_1|) &=& 1 - \sum_{k = 1}^{K} \frac{c_k^2}{(1-|c_1| - s_k)^2}  =  1 - \frac{c_1^2}{c_1^2} - \sum_{k = 2}^{K} \frac{c_k^2}{(1-|c_1| - s_k)^2}  \notag\\
&\le & - \sum_{k = 2}^{K} \frac{c_k^2}{(1-|c_1| - s_k)^2}   
\le 0 \, . 	\notag
\ee
This implies that $f^{\prime}(\lambda)$ has a unique root smaller than $1$. Moreover, the root lies in the interval $[0,1-|c_1|)$.
%
%
\end{proof}

\section{Proof of Lemma~\ref{lem_minimumroot}}
\label{sec:proof_lem_minimumroot}
\begin{proof}
Let $\lambda_1$ be a root which is smaller than $s_1 = 1$, and 
$\lambda_2$ be another root of $f^{\prime}(\lambda) $. Then, according to Lemma~\ref{lem_root_lessthan1}, $\lambda_2 > 1 > \lambda_1$, and 
\be
\sum_{k = 1}^{K} \frac{c_k^2}{(\lambda_1 - s_k)^2} = \sum_{k = 1}^{K} \frac{c_k^2}{(\lambda_2 - s_k)^2} = 1\,. \label{equ_lambda12_fzero}
\ee
It can be shown that 
\be
f(\lambda_2) - f(\lambda_1) &=& \lambda_2 - \lambda_1 + \sum_{k = 1}^{K} \frac{c_k^2}{\lambda_2 - s_k} - \frac{c_k^2}{\lambda_1 - s_k} \notag\\
&=& (\lambda_2 - \lambda_1) \left( 1 -  \sum_{k = 1}^{K} \frac{|c_k|}{s_k-\lambda_1} \frac{|c_k|}{s_k-\lambda_2}\right)  \notag  \\
&\ge &(\lambda_2 - \lambda_1)  \left( 1 - \sqrt{\sum_{k = 1}^{K} \frac{c_k^2}{(\lambda_1 - s_k)^2}}  \, \sqrt{ \sum_{k = 1}^{K} \frac{c_k^2}{(\lambda_2 - s_k)^2}} \right) \notag 
 \\
&=& (\lambda_2 - \lambda_1)  ( 1 - 1 \times 1) = 0.\label{eq_flambda1_2_b}
\ee
This inequality is obtained by applying the Cauchy-Schwarz inequality, whereas (\ref{eq_flambda1_2_b}) is obtained after replacing the optimal conditions in (\ref{equ_lambda12_fzero}).
The equality case does not occur because of $s_1 -\lambda_2 < 0$,
that is, $f(\lambda_2) > f(\lambda_1)$ 
and the minimiser $\lambda^{\star}$ of $f(\lambda)$ is the minimum root  $\lambda_1$ of $f^{\prime}(\lambda)$.
\end{proof}

\section{Proof of Lemma~\ref{lem_bound_lda_1term}}
\label{sec:proof_lem_bound_lda_1term}

\begin{proof}
We first show that the polynomials $p_i(t)$ have unique roots in $\displaystyle \left[{|c_1|},1 \right]$.
The second-derivative of $p_i(t)$ is given by 
\be
p_i^{\prime\prime}(t) = 12t^2 +12d_i t + 2(d_i^2-1) 	\notag
\ee
and has two roots $\bar{t}_{1,2}  = \displaystyle \frac{-3d_i \mp \sqrt{3d_i^2 + 6}}{6}$. 

If $d_i > 1$, 
the roots $\bar{t}_{1,2}$ are negative. Hence, the first derivative $p_i^{\prime}(t)$ monotonically increases in $[0, +\infty)$. In addition, since  
\be
p_i^{\prime}(0) = - 2c_1^2 d_i  \le 0		\notag 
\ee
 $p_i^{\prime}(t)$ has only one root in $[0, +\infty)$. 
 Together with the fact that 
\be
p_i(0) &=& -c_1^2 d_i^2 \le 0\notag \\
p_i(|c_1|) &=&  c_1^2 (c_1^2 - 1) \, \le  \, 0 \notag\\
p_i(1) &=& d_i(d_i+2)  (1 -  c_1^2)  \ge 0, \notag 
\ee
the polynomial $p_i(t)$ has a unique root in $\displaystyle \left[{|c_1|},1 \right]$.

If $d_i\le1$, the second root $\bar{t}_2$ is non-negative, $\bar{t}_2 \ge 0$. However, since $p_i^{\prime}(0) = - 2c_1^2 d_i \le 0$,  the first derivative $p_i^{\prime}(t)$  has only one root in $[0, +\infty)$. Again as for the case $d_1>1$, the polynomial $p_i(t)$ also has unique root in $\displaystyle \left[{|c_1|},1 \right]$.

As the definition of the root $t_2$, we can prove that derivative $f^{\prime}(s_1  - t_2)$ does not exceed zero, that is
\be
f^{\prime}(s_1-t_2) &=& 1 - \frac{c_1^2}{t_2^2} - \sum_{k = 2}^{K} \frac{c_k^2}{(s_k-s_1 + t_2)^2} 
\le 1 - \frac{c_1^2}{t_2^2} - \frac{\sum_{k = 2}^{K}  c_k^2}{(s_K-s_1 + t_2)^2}   
\notag\\
&=& 1 - \frac{c_1^2}{t_2^2} - \frac{1 - c_1^2}{(d_2 + t_2)^2}  
= \frac{p_2(t_2)}{t_2^2 (d_2 + t_2)^2} \notag \\
&=& 0. \label{equ_t1}
\ee
Similarly,  we have 
\be
f^{\prime}(s_1-t_1) &=& 1 - \frac{c_1^2}{t_1^2} - \sum_{k = 2}^{K} \frac{c_k^2}{(s_k-s_1 + t_1)^2}  
\ge  1 - \frac{c_1^2}{t_1^2} - \frac{\sum_{k = 2}^{K}  c_k^2}{(d_1 + t_1)^2} \notag\\
&=& 1 - \frac{c_1^2}{t_1^2} - \frac{1 - c_1^2}{(d_1 + t_1)^2}  
= \frac{p_1(t_1)}{t_1^2 (d_1 + t_1)^2} \notag \\
&=& 0.\label{equ_t2}
\ee
From (\ref{equ_t1}) and (\ref{equ_t2}), it follows that $f^{\prime}(t)$ has a root  in $[1 - t_1, 1-t_2]$. This root is unique and also the global  minimiser of $f(\lambda)$ in (\ref{func_lambda}).
This completes the proof. 
\end{proof}

\section{Proof of Lemma~\ref{lem_bound_lda_Lterms}}
\label{sec:proof_lem_bound_lda_Lterms}

\begin{proof}
First, similar to Lemma~\ref{lem_root_lessthan1}, the roots $\lambda_{l,L}^{\star}$ and $\lambda_{u,L}^{\star}$  are unique in the interval $[0, 1-|c_1|]$.
Taking into account that  $\sum_{k = 1}^{K} c_k^2 = 1$, and $s_K \ge s_k$ for all $k$, we have 
\be
f^{\prime}(\lambda) &=&  1 - \sum_{l = 1}^{L} \frac{c_l^2}{(s_l - \lambda)^2}  -  \sum_{k = L+1}^{K} \frac{c_k^2}{(s_k- \lambda)^2}  \notag \\
&\le &  1 - \sum_{l = 1}^{L} \frac{c_l^2}{(s_l - \lambda)^2}  -  \frac{\sum_{k = L+1}^{K}   c_k^2}{(s_K- \lambda)^2}  =  1 - \sum_{l = 1}^{L} \frac{c_l^2}{(s_l - \lambda)^2}  -  \frac{\tilde{c}_{L+1}^2}{(s_K- \lambda)^2}  \notag \\
&=& f^{(L)}_{u}(\lambda) \, .	\notag 
\ee
Similarly, we can derive 
$f^{\prime}(\lambda)  \ge  f^{(L)}_{l}(\lambda)$.  
It appears that the function values of $f^{\prime}(\lambda)$ at $\lambda_{l,L}^{\star}$ and $\lambda_{u,L}^{\star}$ are nonnegative and non-positive, respectively, 
\be
	f^{\prime}(\lambda_{l,L}^{\star}) &\ge& 	f^{(L)}_{l}(\lambda_{l,L}^{\star}) = 0  \, ,\notag \\
	f^{\prime}(\lambda_{u,L}^{\star}) &\le& 	f^{(L)}_{u}(\lambda_{u,L}^{\star}) = 0,\notag 
\ee
thus implying that
\be
	\lambda_{l,L}^{\star} \le \lambda^{\star} \, \le \lambda_{u,L}^{\star}.  \notag 
\ee
The sequence of inequalities in (\ref{equ_sequence_lambda}) can be proved in a similar way.
\end{proof}

\section{Proof of Lemma~\ref{lem_cn_zero}}
\label{sec:proof_lem_cn_zero}

\begin{proof}
By contradiction, assume that the variable $\tilde{x}_n^{\star}$ is non-zero.
%
%
Since there is only one $c_n = 0$, from (\ref{equ_1st_opt_condtion}), the multiplier $\lambda^{\star}$ must be equal to $s_n$, that is
\be
\lambda^{\star} = s_n,	\notag 
\ee
and the minimiser $\tilde{\bx}^{\star}$ is given by
\be
	\tilde{x}_{k}^{\star} &=& \frac{c_k}{s_n - s_k}  \, ,\quad k \neq n	\notag 
\ee
and $\tilde{x}_{n}^{\star}$ is derived from the unit-length condition of $\tilde{\bx}^{\star}$
\be
\tilde{x}_{n}^2 = 1 - \sum_{k\neq n} \tilde{x}_{k}^2 		\notag 
\ee
with an additional assumption that 
\be
	\sum_{k\neq n} \frac{c_k^2}{(s_n-s_k)^2} < 1\, .	\notag 
\ee



The objective function in (\ref{equ_qp_sphere2}) at ${\tilde\bx}^{\star}$, as well as the Lagrangian function at $({\tilde\bx}^{\star}, \lambda^{\star} = s_n)$ are given by 
\be
\calL({\tilde\bx}^{\star}, \lambda^{\star})   
&=&   \frac{1}{2} \left(s_n -  s_n \, \sum_{k \neq n} \frac{c_k^2 }{(s_n - s_k)^2} + \sum_{k \neq n} \frac{c_k^2 \, s_k}{(s_n - s_k)^2} \right)  + \sum_{k \neq n}   \frac{c_k^2 }{s_n - s_k} \notag\\
&=& \frac{1}{2}\left( s_n  + \sum_{k \neq n} \frac{c_k^2}{s_n - s_k} \right) \, . \label{equ_lagrang_xtlda}
\ee
Now, we consider a vector $\bar{\bx}$ whose $n$-th entry is zero, $\bar{x}_n = 0$, and the rest $(K-1)$ coefficients $\bar{\bx}_{n} = [\bar{x}_1, \ldots, \bar{x}_{n-1}, \bar{x}_{n+1},\ldots, \bar{x}_K]$ are minimiser to a reduced problem 
\be
&\min  \quad & \frac{1}{2} \sum_{k \neq n} s_k \, \tilde{x}_k^2   + \sum_{k \neq n}   c_k \, \tilde{x}_k    \notag \label{equ_qp_sphere2_reduced} \\
&\text{subject to}  \quad & \sum_{k \neq n}  \tilde{x}_k^2 = 1 \, .\notag 
\ee
According to the results in Section~\ref{sec::nonzero_cn}, when $c_k$, $k \neq n$, are non-zeros,
the Lagrangian function for this reduced problem at the minimiser $\bar{\bx}_{n}$ is given by
\be
\calL_n(\bar{\bx}_{n}, \lambda_n^{\star})   = \frac{1}{2}\left( \lambda_n^{\star}  +  \sum_{k \neq n}  \frac{c_k^2}{\lambda_n^{\star} - s_k} \right) \,,
\label{equ_lagrang_xbar}
\ee
where the optimal multiplier $\lambda_n^{\star} < s_1 = 1$. 
From (\ref{equ_lagrang_xtlda}) and (\ref{equ_lagrang_xbar}), it is apparent that 
\be
	\calL({\tilde\bx}^{\star}, \lambda^{\star})   >   \calL_n(\bar{\bx}_{n}, \lambda_n^{\star})  = \calL({\bar\bx}, \lambda^{\star}_n) \,,	\notag 
\ee
which contradicts with the claim that $\tilde{\bx}^{\star}$ is the minimiser to the problem (\ref{equ_qp_sphere2}).
This implies that the $n$-th variable of the minimiser must be zero, i.e., $\tilde{x}_n^{\star} = 0.$
\end{proof}

\section{Proof of Lemma~\ref{lem_c1_zero}}
\label{sec:proof_lem_c1_zero}

\begin{proof}
When $c_1 = 0$, from the first optimality condition in (\ref{equ_1st_opt_condtion}), 
we have 
\be	
	(s_1 - \lambda) \, \tilde{x}_1 =  0\, .	\notag 
\ee
Assume that $\tilde{\bx}^{\star}$ is a minimiser to the problem in (\ref{equ_qp_sphere2}) with a non-zero $\tilde{x}_1^{\star}$, then $\lambda = s_1 = 1$
and 
\be	
	\tilde{x}_k^{\star} = \frac{c_k}{\lambda-s_k} = \frac{c_k}{1-s_k} \notag 
\ee
for $k >1$.
From the unit-length constraint, it follows that $\tilde{x}_1$ can be deduced as 
\be
{(\tilde{x}_1^{\star})}^2 =  1- \sum_{k>1} ({\tilde{x}_k^{\star}})^2  = 1 - d		\notag 
\ee
which requires the condition $ 
d \le 1$.
Implying that, if $ d >1$, $\tilde{x}_1^{\star}$ must be zero, and the rest $(K-1)$ variables $[\tilde{x}_2^{\star},\ldots, \tilde{x}_K^{\star}]$ are minimiser to the reduced problem of (\ref{equ_qp_sphere2}).

When $d\le1$, there exists $\tilde{x}_1^{\star}$, and 
the objective function at $\tilde{\bx}^{\star}$ is given by
\be
\calL(\tilde{\bx}^{\star},s_1) = \frac{1}{2}\left(1 - \sum_{k>2} \frac{c_k^2}{s_k - 1}\right) \, .
\notag 
\ee

Now, we consider a vector $\tilde{\bx}$ whose $\tilde{x}_1 = 0$, and 
$\bar{\bx}  = [\tilde{x}_2, \ldots, \tilde{x}_K]$ is a minimiser to the reduced problem (\ref{prob_x_2_K}).  
Similar to the analysis in Section~\ref{sec::nonzero_cn}, the objective function of the reduced problem (\ref{prob_x_2_K}) 
achieves a global minimum at the minimum root $\bar{\lambda}$ of the first derivative of the Lagrangian function
\be
\calL_1(\bar{\bx},\bar{\lambda}) =  \frac{1}{2}\left(\bar{\lambda} - \sum_{k>2} \frac{c_k^2}{ s_k - \bar{\lambda}} \right)\, ,		\notag 
\ee
where $\bar{\lambda}$ is smaller than $s_2$.

Since the second derivative of $\calL_1(\bar{\bx},\lambda)$ w.r.t. $\lambda$ is negative for all $\lambda < s_2$, the function $\calL_1(\bar{\bx},\lambda)$ is concave in $(-\infty, s_2)$. It then follows that 
\be
\calL(\tilde{\bx}^{\star},s_1) < \calL_1(\bar{\bx},\bar{\lambda})\, ,		\notag 
\ee 
and $\tilde{\bx}^{\star}$ is the global minimiser.
Note that $\tilde{x}_1^{\star}$ can be $\sqrt{1-d}$ or $-\sqrt{1-d}$.

\end{proof}

\section{Proof of Lemma~\ref{lem_linreg_circle}}
\label{sec:proof_lem_linreg_circle}

\begin{proof} 
 Let  $\bx^{\star}$ be a minimiser to the problem (\ref{eq_linreg_bound})
  \be
 	\bx^{\star} = \argmin_{\bx} \quad \|\bx\|^2 \quad \text{s.t.} \quad \|\by - \bA \bx\| \le \delta .	\notag
 \ee
It is obvious that if there are zero entries in $\bx^{\star}$, we can omit columns of $\bA$ corresponding to these entries, and the regression problem formulated for the remaining sub matrix of $\bA$ has a non-zero minimiser. 
 Hence, we can assume that entries of $\bx^{\star}$ are nonzeros.
Let $\bz = \by - \sum_{k = 2}^{K} \ba_k x_k^{\star} $, then $x_1^{\star}$ is a minimiser to the optimisation w.r.t. $x_1$, that is
\be
x_1^{\star}  = \argmin_{x_1}   \quad x_1^2  \quad \text{s.t.} \quad \|\bz - \ba_1 x_1 \| \le \delta  \label{eq_prob_x1}.
\ee
The constraint function can be written as 
\be
c(x_1) =  \|\bz - \ba_1 x_1\|^2 - \delta^2 =   \|\ba_1\|^2 \, x_1^2 - 2 (\ba_1^T \bz)  \, x_1  + \|\bz\|^2 - \delta^2.	\notag
 \ee
 Since $c(x_1^{\star}) \le 0$, $c(x_1)$ must have two roots $t_{-}$ and $t_{+}$. Moreover, it is clear from (\ref{eq_prob_x1}) that $\|\bz\|^2  > \delta^2$ otherwise $x_1^{\star} = 0$. Hence, the two roots $t_{-}$ and $t_{+}$ have the same signs because 
 \be
 t_{-} t_{+} = \frac{\|\bz\|^2 - \delta^2}{\|\ba_1\|^2} > 0\, .		\notag
 \ee 
As a result, the minimiser  to (\ref{eq_prob_x1}) must be one of the two roots, $x_1^{\star} = \min(|t_{-}|,|t_{+}|)$, 
and the inequality condition becomes the equality one.
\end{proof}


\bibliographystyle{spmpsci}      


\end{document}